\documentclass[mn-ohne,fleqn]{w-art}        
\usepackage{times}
\usepackage{w-thm}
 \numberwithin{equation}{section}
  \usepackage[ngerman,english]{babel}
\usepackage[]{graphicx}
 \usepackage{amssymb}
 \def\cprime{$'$}
 \makeatletter
\newenvironment{proofb}[1][\indent\emph{P\,r\,o\,o\,f}]{\par
  \pushQED{\qed}%
  \normalfont \topsep0\p@\@plus6\p@\relax
  \trivlist
  \item[\hskip\labelsep
        \itshape
    #1\@addpunct{.}]\ignorespaces
}{%
  \popQED\endtrivlist\@endpefalse
}

\@abstractcopyrightfalse                            

  \def\@oddfoot{%
    \if@finallayout%
      \normalfont\fontsize{8}{9\p@}\normalfont%
      \bfseries\journal@url \hfil
      \normalfont\fontsize{6}{7\p@}\normalfont%
      \copyright\ \copyrightyear\ \copyrightholder
    \else
      \normalfont\fontsize{6}{7\p@}\normalfont%
      \hfil \phantom{Copyright line will be provided by the publisher} 
    \fi}%
  \def\@evenfoot{%
    \if@finallayout%
      \normalfont\fontsize{6}{7\p@}\normalfont%
      \copyright\ \copyrightyear\ \copyrightholder\hfil
      \normalfont\fontsize{8}{9\p@}\normalfont%
      \bfseries\journal@url
    \else
      \normalfont\fontsize{6}{7\p@}\normalfont%
      \hfil \phantom{Copyright line will be provided by the publisher} 
    \fi}%
 \makeatother
 \newcommand*{\beq}{$$\refstepcounter{equation}}
 \newcommand*{\eeq}{\eqno(\theequation)$$}
 \newcommand*{\bal}{\begin{aligned}}
 \newcommand*{\eal}{\end{aligned}}
 \newcommand*{\qa}{,\qquad}
 \newcommand*{\qb}{,\quad}
 \newcommand*{\mf}[1]{\boldsymbol{#1}} 
 \newcommand*{\ci}{\mathaccent"7017 } 
 \newcommand*{\cqqncel}[2]{\ooalign{$\hfil#1\mkern1mu/\hfil$\crcr$#1#2$}}
 \newcommand*{\notsim}{\mathrel{\mathpalette\cqqncel\sim}}
 \newcommand*{\hb}[1]{\hbox{$#1$}} 
 \newcommand*{\sdot}{\!\cdot\!}
 \newcommand*{\sco}{\kern2pt\colon\kern2pt}
 \newcommand*{\sn}{\kern1pt|\kern1pt}
 \newcommand*{\ssm}{\!\setminus\!}
 \newcommand*{\vsdot}{\hbox{$|\sdot|$}}
 \newcommand*{\Vsdot}{\hbox{$\|\sdot\|$}}
 \newcommand{\vd}{\hbox{$\vert\kern-.2ex\vert\kern-.2ex\vert$}}
 \newcommand*{\Vvsdot}{\hbox{$\vd\sdot\vd$}}
 \newcommand*{\hh}[1]{{\textbf{#1}}}
 \newcommand*{\pe}{\hbox{$[{}\sdot{},{}\sdot{}]$}}
 \newcommand*{\pr}{\hbox{$(\cdot,\cdot)$}}
 \newcommand*{\prsn}{\hbox{$(\cdot\sn\cdot)$}}
 \newcommand*{\pw}{\hbox{$\dl{}\sdot{},{}\sdot{}\dr$}}
 \newcommand*{\mfpw}{\hbox{$\mf{\dl}\!\!\!\mf{\dl}{}\sdot{},{}\sdot{}\mf{\dr}\!\!\!\mf{\dr}$}}
 \newcommand*{\mfdl}{\hbox{$\mf{\dl}\!\!\!\mf{\dl}$}}
 \newcommand*{\mfdr}{\hbox{$\mf{\dr}\!\!\!\mf{\dr}$}}
 \newcommand*{\bid}[2]{\Bigl({{#1}\atop{#2}}\Bigr)}
 \newcommand*{\ol}{\overline}
 \newcommand*{\wh}{\widehat}
 \newcommand*{\wt}{\widetilde}
 \newcommand*{\thdot}[3]{\vbox{\ialign{##\crcr    
                    {\kern#2ex\raise#3ex\hbox{$\scriptscriptstyle\bt$}}\crcr\noalign{\kern-1pt\nointerlineskip}
                    $\hfil\displaystyle{#1}\hfil$\crcr}}}
 \newcommand*{\thB}{\rlap{\thdot{B}{.7}{.4}}{\ph{B}}}               
 \newcommand*{\thmfB}{\rlap{\thdot{\mf{B}}{.8}{.4}}{\ph{\mf{B}}}}   
 \newcommand*{\thmfgF}{\rlap{\thdot{\mf{\gF}}{.8}{.4}}{\ph{\mf{\gF}}}}   
 \newcommand*{\thgK}{\rlap{\thdot{\gK}{.5}{.4}}\ph{\gK}}            
 \newcommand*{\thS}{\thdot{S}{.5}{.4}}
 \newcommand*{\thV}{\thdot{V}{.5}{.4}}                              
 \newcommand*{\thchi}{\thdot{\chi}{.5}{.4}}                        
 \newcommand*{\thia}{\rlap{\thdot{\ia}{.5}{.4}}\ph{\ia}}            
 \newcommand*{\thka}{\rlap{\thdot{\ka}{.5}{.4}}{\ph{\ka}}}                           
 \newcommand*{\thwtka}{\rlap{\thdot{\wt{\ka}}{.4}{.4}}{\ph{\ka}}}
 \newcommand*{\ithka}{{\thdot{\scriptstyle\ka}{.5}{.4}}}            
 \newcommand*{\thpi}{\thdot{\pi}{.5}{.4}}                           
 \newcommand*{\thvp}{\rlap{\thdot{\vp}{.5}{.4}}\ph{\vp}}            
 \newcommand*{\thpsi}{\rlap{\thdot{\psi}{.5}{.5}}\ph{\psi}}         
 \newcommand*{\thrho}{\rlap{\thdot{\rho}{.5}{.4}}\ph{\rho}}         
 \newcommand*{\ph}{\phantom}
 \newcommand*{\vph}{\vphantom}
 \newcommand*{\hr}{\hookrightarrow}
 \newcommand*{\ra}{\rightarrow}
 \newcommand*{\dl}{\langle}
 \newcommand*{\dr}{\rangle}
 \newcommand*{\bwt}{{\textstyle\bigwedge}}
 \newcommand*{\sdh}{\stackrel{d}{\hookrightarrow}}
 \newcommand*{\bdry}{\mathop{\mathrm{bdry}}\nolimits}
 \newcommand*{\card}{\mathop{\mathrm{card}}\nolimits}
 \newcommand*{\tdiv}{\mathop{\rm div}\nolimits}
 \newcommand*{\divgrad}{\mathop{\rm div\,grad}\nolimits}
 \newcommand*{\dom}{\mathop{\rm dom}\nolimits}
 \newcommand*{\grad}{\mathop{\rm grad}\nolimits}
 \newcommand*{\Ho}{{*}}
 \newcommand*{\Hodge}{\rm Hodge}
 \newcommand*{\id}{{\rm id}}
 \newcommand*{\im}{\mathop{\rm im}\nolimits}
 \newcommand*{\pro}{{\rm pr}}
 \newcommand*{\supp}{\mathop{\rm supp}\nolimits}
 \newcommand*{\trace}{\mathop{\rm trace}\nolimits}
 \newcommand*{\Lis}{{\mathcal L}{\rm is}}
 \newcommand*{\loc}{{\rm loc}}
 \newcommand*{\is}{\subset}
 \newcommand*{\Chr}{\Ga_{ij}^k}
 \newcommand*{\bt}{\bullet}
 \newcommand*{\es}{\emptyset}
 \newcommand*{\iy}{\infty}
 \newcommand*{\mt}{\mapsto}
 \newcommand*{\pl}{\partial}
 \newcommand*{\pa}{\partial^\alpha}
 \newcommand*{\sh}{\sharp}
 \newcommand*{\cona}{\kern-1pt}
 \newcommand*{\coT}{\kern-1pt}
 \newcommand*{\coU}{\kern-1pt}
 \newcommand*{\coV}{\kern-1pt}
 \newcommand*{\coW}{\kern-1pt}
 \newcommand*{\cimfW}{\rlap{${\ci{\mf{W}}}$}{\ph{\mf{W}}}}
 \newcommand*{\cimfgF}{\rlap{${\ci{\mf{\gF}}}$}{\ph{\mf{\gF}}}}
 \newcommand*{\al}{\alpha}
 \newcommand*{\ba}{\beta}
 \newcommand*{\Da}{\Delta}
 \newcommand*{\da}{\delta}
 \newcommand*{\Ga}{\Gamma}
 \newcommand*{\ga}{\gamma}
 \newcommand*{\mfga}{\rlap{$\mf{\ga}$}{\kern.3pt\mf{\ga}}}
 \newcommand*{\ia}{\iota}
 \newcommand*{\ka}{\kappa}
 \newcommand*{\lda}{\lambda}
 \newcommand*{\na}{\nabla}
 \newcommand*{\Om}{\Omega}
 \newcommand*{\om}{\omega}
 \newcommand*{\sa}{\sigma}
 \newcommand*{\ta}{\theta}
 \newcommand*{\vp}{\varphi}
 \newcommand*{\Bl}{(B,\ell)}
 \newcommand*{\BHmE}{(\BH^m,E)}
 \newcommand*{\ciBHmE}{(\ci\BH^m,E)}
 \newcommand*{\ME}{(M,E)}
 \newcommand*{\ciME}{(\ci M,E)}
 \newcommand*{\Mg}{(M,g)}
 \newcommand*{\MBK}{(M,\BK)}
 \newcommand*{\MV}{(M,V)}
 \newcommand*{\ciMV}{(\ci M,V)}
 \newcommand*{\ciMVs}{(\ci M,V')}
 \newcommand*{\Rm}{(\BR^m)}
 \newcommand*{\RmE}{(\BR^m,E)}
 \newcommand*{\SV}{(S,V)}
 \newcommand*{\XE}{(X,E)}
 \newcommand*{\BXE}{(\BX,E)}
 \newcommand*{\BXkE}{(\BX_\ka,E)}
 \newcommand*{\ciBXE}{(\ci\BX,E)}
 \newcommand*{\ciBXkE}{(\ci\BX_\ka,E)}
 
 \newcommand*{\XY}{(X,Y)}
 \newcommand*{\cXcX}{(\cX,\cX)}
 \newcommand*{\cXcY}{(\cX,\cY)}
 \newcommand*{\YX}{(Y,X)}
 \renewcommand*{\ldots}{\mathinner{\ldotp\ldotp\ldotp}}
 \newcommand*{\BC}{{\mathbb C}}
 \newcommand*{\BH}{{\mathbb H}}
 \newcommand*{\BJ}{{\mathbb J}}
 \newcommand*{\BK}{{\mathbb K}}
 \newcommand*{\BN}{{\mathbb N}}
 \newcommand*{\BR}{{\mathbb R}}
 \newcommand*{\BS}{{\mathbb S}}
 \newcommand*{\BX}{{\mathbb X}}
 \newcommand*{\BZ}{{\mathbb Z}}
 \newcommand*{\cD}{{\mathcal D}}
 \newcommand*{\cL}{{\mathcal L}}
 \newcommand*{\cS}{{\mathcal S}}
 \newcommand*{\cT}{{\mathcal T}}
 \newcommand*{\cX}{{\mathcal X}}
 \newcommand*{\cY}{{\mathcal Y}}
 \newcommand*{\gF}{{\mathfrak F}}
 \newcommand*{\gK}{{\mathfrak K}}
 \newcommand*{\gL}{{\mathfrak L}}
 \newcommand*{\gM}{{\mathfrak M}}
 \newcommand*{\gN}{{\mathfrak N}}
 \newcommand*{\gS}{{\mathfrak S}}
 \newcommand*{\gT}{{\mathfrak T}}
 \newcommand*{\gU}{{\mathfrak U}}
 \newcommand*{\sA}{{\mathsf A}}
 \newcommand*{\sC}{{\mathsf C}}

 \makeatletter
 \newif\ifinany@
 \newcount\column@
 \def\column@plus{%
    \global\advance\column@\@ne
 }
 \newcount\maxfields@
 \def\add@amps#1{%
    \begingroup
        \count@#1
        \DN@{}%
        \loop
            \ifnum\count@>\column@
                \edef\next@{&\next@}%
                \advance\count@\m@ne
        \repeat
    \@xp\endgroup
    \next@
 }
 \def\Let@{\let\\\math@cr}
 \def\restore@math@cr{\def\math@cr@@@{\cr}}
 \restore@math@cr
 \def\default@tag{\let\tag\dft@tag}
 \default@tag

 \newbox\strutbox@
 \def\strut@{\copy\strutbox@}
 \addto@hook\every@math@size{%
  \global\setbox\strutbox@\hbox{\lower.5\normallineskiplimit
         \vbox{\kern-\normallineskiplimit\copy\strutbox}}}

 \renewcommand{\start@aligned}[2]{%
    \RIfM@\else
        \nonmatherr@{\begin{\@currenvir}}%
    \fi
    \null\,%
    \if #1t\vtop \else \if#1b \vbox \else \vcenter \fi \fi \bgroup
        \maxfields@#2\relax
        \ifnum\maxfields@>\m@ne
            \multiply\maxfields@\tw@
            \let\math@cr@@@\math@cr@@@alignedat
        \else
            \restore@math@cr
        \fi
        \Let@
        \default@tag
        \ifinany@\else\openup\jot\fi
        \column@\z@
        \ialign\bgroup
           &\column@plus
            \hfil
            \strut@
            $\m@th\displaystyle{##}$%
           &\column@plus
            $\m@th\displaystyle{{}##}$%
            \hfil
            \crcr
 }
 \renewenvironment{aligned}[1][c]{%
    \start@aligned{#1}\m@ne
 }{%
    \crcr\egroup\egroup
 }
 \makeatother
\begin{document}
\DOIsuffix{mana.DOIsuffix}
\Volume{248}
\Month{01}
\Year{2005}
\pagespan{1}{}
    \Receiveddate{15 November 2005}
    \Reviseddate{30 November 2005}
    \Accepteddate{2 December 2005}
    \Dateposted{3 December 2005}
    \vspace*{3\baselineskip}    
\title[Function Spaces on Singular Manifolds]{Function Spaces on Singular
Manifolds}
\author[H. Amann]{H. Amann\footnote{e-mail: {\sf herbert.amann@math.uzh.ch}}}
\address{Math.\ Institut, Universit\"at Z\"urich,
Winterthurerstr.~190, CH--8057 Z\"urich, Switz\-er\-land}
\begin{abstract}
It is shown that most of the well-known basic results for Sobolev-Slobodeckii
and Bessel potential spaces, known to hold on bounded smooth domains
in~$\BR^n$, continue to be valid on a wide class of Riemannian
manifolds with singularities and boundary, provided suitable weights, which
reflect the nature of the singularities, are introduced. These results are
of importance for the study of partial differential equations on piece-wise
smooth domains.
\end{abstract}
\maketitle                 
\section{Introduction}\label{sec-J}
It is our principal concern in this paper to develop a satisfactory theory of
spaces of functions and tensor fields on Riemannian manifolds which may
have a boundary and may be non-compact and non-complete.
Such a theory has to extend the basic results known for function spaces on
subdomains of~$\BR^n$ with smooth boundary to this more general setting,
that is to say, embedding and interpolation properties,
point-wise multiplier and trace theorems, duality characterizations and,
last but not least, intrinsic local descriptions.

\par
Our research is motivated by~---~and provides the basis for~---~ the study
of elliptic and parabolic boundary value problems on piece-wise smooth
manifolds, on
domains in~$\BR^n$ with piece-wise smooth boundary, in particular. Such
domains occur in a wide variety of problems modeling physical, chemical,
biological, and engineering processes by means of differential 
and pseudodifferential equations. 
In this connection Sobolev spaces play a predominant role, as is well-known
from the theory of partial differential equations on smooth domains. In
the presence of singularities, say edges on the boundary, solutions of
differential equations lose their smoothness near these singularities.
Since the seminal work of V.A.~Kondrat{\cprime}ev~\cite{Kon67a} on elliptic
boundary value problems in domains with conical points it is known that an
appropriate setting for the study of such problems is provided by Sobolev
spaces with weights reflecting the nature of the singularity. This has
since been exploited by numerous authors and there is a large number
of papers and monographs devoted to elliptic problems on non-smooth
domains. 
Besides of the early papers by V.G.~Maz{\cprime}ya and 
B.A.~Plamenevski{\u\i} \hbox{\cite{MaP72a}\nocite{MaP73a}--\cite{MaP77a}}, 
the first successful approaches to this kind of problems, we cite only the 
following few books and refer the reader to the
references therein for further information:
P.~Grisvard~\cite{Gri85a}, M.~Dauge~\cite{Dau88a},
S.A.~Nazarov and B.A.~Plamenevski{\u\i}~\cite{NaP94a},
V.A.~Kozlov, V.G.~Maz{\cprime}ya, and J.~Rossmann~\cite{KMR97a}, 
V.G.~Maz{\cprime}ya, and J.~Rossmann~\cite{MaR10a} 
(and many more papers and books by V.G.~Maz{\cprime}ya and coauthors), and
the numerous contributions of B.-W.~Schulze and co-workers on the
\hbox{$L_2$-theory} of elliptic pseudo-differential boundary problems on
singular manifolds for which \cite{Schu91a} may stand representatively.

\par
Weighted Sobolev spaces of a different type occur as solution spaces
for degenerate elliptic equations. This fact has triggered a large
amount of research on weighted Sobolev and related function spaces, e.g.,
A.~Kufner~\cite{Kuf85a}, H.~Triebel~\cite{Tri78a},
H.-J. Schmeisser and H.~Triebel~\cite{SchT87a}, and the references
therein. Since that work is not directly related to the subject of
our paper we do not give more details or cite more recent references.

\par
In Section~\ref{sec-S} we give a precise definition of our concept of a
singular manifold~$M$. It will be seen that, to a large extent, $M$~is
determined by a `singularity function'
\hb{\rho\in C^\iy\bigl(M,(0,\iy)\bigr)}. The behavior of~$\rho$ at the
`singular ends' of~$M$, that is, near that parts of~$M$ at which $\rho$~gets
either arbitrarily small or arbitrarily large, 
reflects the singular structure of~$M$. It turns out
that the basic building blocks for a useful theory of function spaces on
singular manifolds are weighted Sobolev spaces based on the singularity
function~$\rho$. More precisely, we denote by~$\BK$ either $\BR$ or~$\BC$.
Then, given
\hb{k\in\BN},
\ \hb{\lda\in\BR}, and
\hb{p\in(1,\iy)}, the weighted Sobolev space
\hb{W_{\coW p}^{k,\lda}(M)=W_{\coW p}^{k,\lda}\MBK} is the completion
of~$\cD(M)$, the space of smooth functions with compact support in~$M$,
in~$L_{1,\loc}(M)$ with respect to the norm
\beq\label{J.n}
u\mt
\Bigl(\sum_{i=0}^k\big\|\rho^{\lda+i}\,|\na^iu|_g\big\|_p^p\Bigr)^{1/p}.
\eeq
Here $\na$~denotes the Levi-Civita covariant derivative and
$|\na^iu|_g$~is the `length' of the covariant tensor field~$\na^iu$
naturally derived from the Riemannian metric $g$ of~$M$. Of course,
integration is carried out with respect to the volume measure of~$M$.
It turns out that $W_{\coW p}^{k,\lda}(M)$ is well-defined,
independently~---~in the sense of equivalent norms~---~of the
representation of the singularity structure of~$M$ by means of the
particular singularity function.

\par
A~very special and simple example of a singular manifold is provided by a
bounded smooth domain whose boundary contains a conical point. More
precisely, suppose $\Om$~is a bounded domain in~$\BR^m$ whose topological
boundary, $\bdry(\Om)$, contains the origin, and
\hb{\Ga:=\bdry(\Om)\ssm\{0\}} is a smooth
\hb{(m-1)}-dimensional submanifold of~$\BR^m$ lying locally on one side
of~$\Om$. Also suppose that
\hb{\Om\cup\Ga} is near~$0$ diffeomorphic to a cone
\hb{\{\,ry\ ;\ 0<r<1,\ y\in B\,\}}, where $B$~is a smooth compact
submanifold of the unit sphere in~$\BR^m$. Then, endowing
\hb{M:=\Om\cup\Ga} with the Euclidean metric, we get a singular
manifold with a single conical singularity, as considered in
\cite{NaP94a} and~\cite{KMR97a}, for example. In this case the weighted
norm~\eqref{J.n} is equivalent to
$$
u\mt\Bigl(\sum_{|\al|\leq k}
\|r^{\lda+|\al|}\pa u\|_{L_p(\Om)}^p\Bigr)^{1/p},
$$
where $r(x)$~is the Euclidean distance from
\hb{x\in M} to the origin. Moreover,
$W_{\coW p}^{k,\lda}(M)$ coincides with the
space $V_{\coV p,\lda+k}^k(\Om)$ employed by
S.A. Nazarov and B.A. Plamenevski{\u\i} (cf.~p.~319 of \cite{NaP94a}) and,
in the case
\hb{p=2}, by
V.A. Kozlov, V.G. Maz{\cprime}ya, and J.~Rossmann (see Section~6.2
of~\cite{KMR97a}, for example).

\par
As mentioned above, the theory of function spaces on singular
manifolds is built on the weighted Sobolev spaces $W_{\coW p}^{k,\lda}(M)$.
We define weighted Sobolev spaces of negative order by duality, and Bessel
potential spaces, $H_p^{s,\lda}(M)$, and Besov spaces,~$B_{p,p}^{s,\lda}(M)$,
by complex and real interpolation, respectively. A~basic result, which
renders the theory useful, is the fact that these spaces can be
characterized locally by their `classical' non-weighted counterparts
on~$\BR^m$ and on half-spaces. This implies, in particular,
\hb{H_p^{k,\lda}(M)=W_{\coW p}^{k,\lda}(M)} for
\hb{k\in\BN}.

\par
A~linear differential operator on a Riemannian manifold is of the form
\hb{\sum_{i=0}^ka_i\cdot\na^iu}, where $a_i$~is a contravariant tensor field
of order~$i$ and
\hb{{}\cdot{}}~denotes complete contraction. In order to study continuity
properties of such operators in the weighted function spaces under
consideration we have to have  at our disposal point-wise multiplier
theorems for tensor fields. Thus it is mandatory to study spaces of tensor
fields on singular manifolds.

\par
In the particular case where we can choose the constant map~$1$
as singularity function, our spaces reduce to non-weighted
Sobolev spaces~$W_{\coW p}^k(M)$, Bessel potential spaces~$H_p^s(M)$, and
Besov spaces~$B_{p,p}^s(M)$, respectively. This is, for example, the
case if $M$~is a complete Riemannian manifold without boundary and with 
bounded geometry (that is, $M$~has a positive injectivity radius and all 
covariant derivatives of the curvature tensor are bounded). 
To the best of our
knowledge, this is the only class of Riemannian manifolds for which a
general theory of function spaces has been developed so far.
More precisely: 

\par
Integer order Sobolev spaces, with particular emphasis on the
validity of Sobolev's embedding theorem, have been treated by
Th.~Aubin \hbox{\cite{Aub76a}\nocite{Aub82a}--\cite{Aub98a}} in the
case of compact manifolds with boundary, and for complete Riemannian
manifolds without boundary, making essential use of curvature estimates
and the positivity of the injectivity radius. Also
see E.~Hebey \cite{Heb99a} and~\cite{HeR08a} for the case where $M$~has no 
boundary. 

\par
Bessel potential spaces~$H_p^s(M)$,
\ \hb{1<p<\iy},
\ \hb{s\in\BR}, on complete Riemannian manifolds without boundary have been
introduced and investigated by R.S.~Strichartz~\cite{Stri83a} as domains of
the fractional powers of
\hb{1-\Da_M}, where $\Da_M$~is the Laplace-Beltrami operator. H.~Triebel
\cite{Tri86a}, \cite{Tri87a} (see also~\cite{Tri92a}) established a general
theory of Triebel-Lizorkin and Besov spaces on complete Riemannian
manifolds without boundary and with bounded geometry. 
His work makes use of a distinguished coordinate system
based on the exponential map and of mapping properties of the
Laplace-Beltrami operator.

\par
None of the above techniques is available in our situation, where $M$~may
be not complete or may not have bounded geometry. 
In particular, relevant properties of the
Laplace-Beltrami operator are not at our disposal, even in the case where
$M$~has no boundary.  Anyhow, they would not be helpful in the presence of a
boundary.

\par 
B.~Ammann, R.~Lauter,  and V.~Nistor~\cite{ALN04a} introduce a class of 
complete non-compact Riemannian manifolds without boundary and with bounded 
geometry, called Lie manifolds. This class encompasses, in particular, 
manifolds with cylindrical ends and manifolds being Euclidean at infinity. 
In B.~Ammann, A.D. Ionescu,  and V.~Nistor~\cite{AIN06a} Bessel potential 
spaces on suitable open subsets of Lie manifolds~---~called Sobolev spaces 
therein and denoted by~$W^{s,p}$~---~ are being investigated to some extent. 
Lie manifolds are useful for the study of regularity properties of elliptic 
differential operators on polyhedral domains in which case the authors are 
led to introduce weighted Bessel potential spaces, the weight being 
equivalent to the distance to the non-smooth boundary points (also see 
\cite{ALN07a}, \cite{AN07a}, and the references therein for related 
research). The results of the present paper apply to Lie manifolds and 
polyhedral domains as well and greatly extend and sharpen the 
investigations of these authors; in particular, as far as the trace theorem 
is concerned. 

\par
There seem to be only very few general results on spaces of tensor fields.
J.~Eichhorn~\cite{Eich88a} studies integer order Sobolev spaces of
differential forms on complete Riemannian manifolds without boundary 
and with bounded geometry; 
also see~\cite{Eich08a}. Some results on
Sobolev spaces of differential forms on compact manifolds with boundary
can be found in G.~Schwarz~\cite{Schwa95a}. Of course, there are many
`ad~hoc' results in the literature, predominantly on \hbox{$L_2$-Sobolev}
spaces, for Riemannian manifolds (without boundary) possessing specific
geometries.

\par
Section~\ref{sec-U} is of technical nature. There we review some concepts
from differential geometry, mainly to fix notation. Then we prove basic
estimates related to the singularity structure of the manifold. They are
fundamental for the construction of universal retractions by which we can
transplant the well-established theory of function spaces on~$\BR^m$ to the
singular manifold. For this we first
have to establish a localization procedure for tensor-field-valued
distributions on~$M$. This is done in Sections \ref{sec-D} and~\ref{sec-L}.
In Section~\ref{sec-W} we show that this localization procedure induces a
corresponding retraction-coretraction system on Sobolev spaces. Then, by
interpolation, we extend the retraction-coretraction theorem to Bessel
potential and Besov spaces of positive order.

\par
After having introduced weighted H\"older spaces in Section~\ref{sec-H}, we
prove in Section~\ref{sec-P} point-wise multiplier theorems.
Section~\ref{sec-T} is devoted to the trace theorem, and in the following
section we characterize spaces with vanishing traces. This puts us in
position to define, in Section~\ref{sec-N}, spaces of negative order by
duality. All spaces
under consideration possess the retraction-coretraction property induced
from the localization procedure for tensor-field-valued sections
constructed in Section~\ref{sec-L}. By means of this property we can then,
in Sections \ref{sec-I} and~\ref{sec-E}, respectively,
easily prove interpolation and embedding theorems for weighted spaces of
tensor fields on singular manifolds.

\par
Section~\ref{sec-F} is concerned with spaces of differential forms. In
particular, we establish mapping properties of the exterior differential
and codifferential operators, and, as an application, of the
gradient and divergence operators. These results are of importance in the
study of differential operators on singular manifolds. Such
investigations, which will be carried out elsewhere, rely fundamentally on
the retraction-coretraction theorems established in this paper.

\par
For simplicity, and being oriented towards differential equations, we
restrict our considerations essentially to weighted Sobolev-Slobodeckii
spaces. However, we include some brief remarks concerning possible
extensions to spaces of Triebel-Lizorkin type.
\section{Singular Manifolds}\label{sec-S}
By~a \emph{manifold} we always mean a smooth, that is, $C^\iy$~manifold
with (possibly empty) boundary such that its underlying topological space
is separable and metrizable. Thus, in the context of manifolds,  we work in
the smooth category. A~manifold need not be connected, but all connected
components are of the same dimension.

\par
We denote by~$\BH^m$  the closed right half-space
\hb{\BR^+\times\BR^{m-1}} in~$\BR^m$, where
\hb{\BR^0=\{0\}}. We set
\hb{Q:=(-1,1)\is\BR}. If $\ka$~is a local chart for an
\hbox{$m$-dimensional} manifold~$M$, then we write~$U_{\coU\ka}$ for the
corresponding coordinate patch~$\dom(\ka)$.
A~local chart~$\ka$ is \emph{normalized} if
\hb{\ka(U_{\coU\ka})=Q^m} whenever
\hb{U_{\coU\ka}\is\ci{M}}, the interior of~$M$, whereas
\hb{\ka(U_{\coU\ka})=Q^m\cap\BH^m} if
\hb{U_{\coU\ka}\cap\pl M\neq\es}. We put
\hb{Q_\ka^m:=\ka(U_{\coU\ka})} if $\ka$~is normalized.

\par
An atlas~$\gK$ for~$M$ has \emph{finite multiplicity} if there exists
\hb{k\in\BN} such that any intersection of more than $k$ coordinate
patches is empty. It is \emph{uniformly shrinkable} if it consists of
normalized charts and there exists
\hb{r\in(0,1)} such that
\hb{\big\{\,\ka^{-1}(rQ_\ka^m)\ ;\ \ka\in\gK\,\big\}} is a cover of~$M$.

\par
Given an open subset~$X$ of $\BR^m$ or~$\BH^m$ and a Banach space~$E$
over~$\BK$, we write
\hb{\Vsdot_{k,\iy}} for the usual norm of $BC^k\XE$, the Banach space of
all
\hb{u\in C^k\XE} such that $|\pa u|_E$~is uniformly bounded for
\hb{\al\in\BN^m} with
\hb{|\al|\leq k}.

\par
By~$c$ we denote constants
\hb{\geq1} whose numerical value may vary from occurrence to occurrence;
but $c$~is always independent of the free variables in a given formula,
unless an explicit dependence is indicated.

\par
Let $S$ be a nonempty set. On~$\BR^S$, the space of all real-valued
functions on~$S$, we introduce an equivalence relation~%
\hb{{}\sim{}} by setting
\hb{f\sim g} iff there exists
\hb{c\geq1} such that
\hb{f/c\leq g\leq cf}. By~$\mf{1}$ we denote the constant function
\hb{s\mt1}, whose domain will always be clear from the context.

\par
The Euclidean metric on~$\BR^m$,
\ \hb{(dx^1)^2+\cdots+(dx^m)^2}, is denoted by~$g_m$. The same symbol is used
for its restriction to an open subset~$U$ of $\BR^m$ or~$\BH^m$, that is,
for~$\ia^*g_m$, where
\hb{\ia\sco U\hr\BR^m} is the natural embedding. Here and below, we employ
the standard notation for pull-back and push-forward operations.

\par
Let
\hb{M=\Mg} be an \hbox{$m$-dimensional} Riemannian manifold. Suppose
\hb{\rho\in C^\iy\bigl(M,(0,\iy)\bigr)} and $\gK$~is an atlas for~$M$. Then
$(\rho,\gK)$ is~a \emph{singularity datum for}~$M$ if
\beq\label{S.sd}
\bal
\rm{(i)}  \qquad    &\gK\text{ is uniformly shrinkable, has finite
                     multiplicity, and is orientation preserving if $M$ is
                     oriented.}\\
\rm{(ii)} \qquad    &\|\wt{\ka}\circ\ka^{-1}\|_{k,\iy}\leq c(k),
                     \ \ \ka,\wt{\ka}\in\gK,
                     \ \ k\in\BN.\\
\rm{(iii)}\qquad    &\ka_*(\rho^{-2}g)\sim g_m,
                     \ \ \ka\in\gK.\\
\rm{(iv)} \qquad    &\|\ka_*(\rho^{-2}g)\|_{k,\iy}\leq c(k),
                     \ \ \ka\in\gK,
                     \ \ k\in\BN.\\
\rm{(v)}  \qquad    &\|\ka_*\rho\|_{k,\iy}\leq c(k)\rho_\ka, 
                     \ \ \ka\in\gK,
                     \ \ k\in\BN,
                     \text{\ where }
                     \rho_\ka:=\ka_*\rho(0)=\rho\bigl(\ka^{-1}(0)\bigr).\\
\rm{(vi)} \qquad    &1/c\leq\rho(p)/\rho_\ka\leq c,
                     \ \ p\in U_{\coU\ka},
                     \ \ \ka\in\gK.
\eal
\eeq
In (ii) and in similar situations it is understood that only
\hb{\ka,\wt{\ka}\in\gK} with
\hb{U_{\coU\ka}\cap U_{\coU\wt{\ka}}\neq\es} are being considered.
Condition~(iii) reads more explicitly:
$$
\ka_*\rho^2(x)\,|\xi|^2/c
\leq\ka_*g(x)(\xi,\xi)\leq c\ka_*\rho^2(x)\,|\xi|^2
\qa x\in Q_\ka^m
\qb \xi\in\BR^m
\qb \ka\in\gK.
$$
Note that the finite multiplicity of~$\gK$ and the separability of~$M$
imply that $\gK$~is countable.

\par
 Let $(\rho,\gK)$ and $(\wt{\rho},\wt{\gK})$ be singularity data for~$M$. Set
 $$
 \gN(\ka):=\{\,\wt{\ka}\in\wt{\gK}
 \ ;\ U_{\coU\wt{\ka}}\cap U_{\coU\ka}\neq\es\,\}
 \qa \ka\in\gK.
 $$
 Then $(\rho,\gK)$ and $(\wt{\rho},\wt{\gK})$ are \emph{equivalent} if
\beq\label{S.eq}
\bal
\rm{(i)}  \qquad    &\rho\sim\wt{\rho};\\
\rm{(ii)} \qquad    &\card\bigl(\gN(\ka)\bigr)\leq c,
                     \ \ \ka\in\gK;\\
\rm{(iii)}\qquad    &\|\wt{\ka}\circ\ka^{-1}\|_{k,\iy}\leq c(k),
                     \ \ \ka\in\gK,
                     \ \ \wt{\ka}\in\wt{\gK},
                     \ \ k\in\BN.
\eal
\eeq
A~\emph{singularity structure},~$\gS(M)$, for~$M$ is a maximal family of
equivalent singularity data. A~\emph{singularity function} for~$M$ is a
function
\hb{\rho\in C^\iy\bigl(M,(0,\iy)\bigr)} such that there exists an
atlas~$\gK$ with
\hb{(\rho,\gK)\in\gS(M)}. The set of all singularity functions is the
\emph{singularity type},~$\gT(M)$, of~$M$. By~a \hh{singular manifold} we
mean a Riemannian manifold~$M$ endowed with a singularity
structure~$\gS(M)$. Then $M$~is said to be \emph{singular of
type}~$\gT(M)$. If
\hb{\rho\in\gT(M)}, then it is convenient to set
\hb{[\![\rho]\!]:=\gT(M)}. A~singular manifold of type~$[\![\mf{1}]\!]$ is
called \emph{uniformly regular}.

\par
Let $\Mg$ be singular of type~$[\![\rho]\!]$. It follows from
\eqref{S.sd}\hbox{(i)--(iv)} that then $(M,\rho^{-2}g)$ is uniformly
regular. Suppose
\hb{\rho\notsim\mf{1}}. Then either 
\hb{\inf\rho=0} or 
\hb{\sup\rho=\iy}, or both. Hence $M$~is not compact but has singular ends. 
It follows from~\eqref{S.sd}(iii) that the diameter of the coordinate 
patches converges either to zero or to infinity near the singular ends in 
a manner controlled by the singularity type~$\gT(M)$.
\begin{examples}\label{exa-S.ex}
\hh{(a)}\quad
Every compact Riemannian manifold is uniformly regular.

\par
\hh{(b)}\quad
Let $M$ be an \hbox{$m$-dimensional} Riemannian submanifold of~$\BR^m$
possessing a compact boundary. Then $M$~is uniformly regular.

\par
\hh{(c)}\quad
\hb{\BR^m=(\BR^m,g_m)} and 
\hb{\BH^m=(\BH^m,g_m)} are uniformly regular. 
\begin{proofb}
For
\hb{\BX\in\{\BR^m,\BH^m\}} and
\hb{z\in\BZ^m\cap\BX} we set
\hb{Q_z^m:=Q^m} if either
\hb{\BX=\BR^m} or
\hb{z\in\ci\BH^m}; otherwise we let
\hb{Q_z^m:=Q^m\cap\BH^m}. We put
\hb{U_{\coU z}:=z+Q_z^m} and
\hb{\ka_z(x):=x-z} for
\hb{z\in\BZ^m\cap\BX} and
\hb{x\in U_{\coU z}}. Then $(\mf{1},\gK)$, where
\hb{\gK:=\{\,\ka_z\ ;\ \BZ^m\cap\BX\,\}}, is a singularity datum for~$\BX$.
\end{proofb}
\hh{(d)}\quad
Let $\Mg$ be singular of type~$[\![\rho]\!]$ and
\hb{\vp\sco M\ra N} a diffeomorphism. Then $(N,\vp_*g)$ is singular of
type~$[\![\vp_*\rho]\!]$. Assume $(\rho,\gK)$ is a singularity datum
for~$M$ and set
\hb{\vp_*\gK:=\{\,\vp_*\ka\ ;\ \ka\in\gK\,\}}. Then
$(\vp_*\rho,\vp_*\gK)$ is a singularity datum for~$N$.

\par
\hh{(e)}\quad
Let $M$ be singular of type~$[\![\rho]\!]$. Suppose
\hb{\pl M\neq\es}. Denote by
\hb{\thia\sco\pl M\hr M} the natural injection and endow~$\pl M$ with
the induced Riemannian metric
\hb{g_{\pl M}:=\thia\,^*g}. Suppose
\hb{\ka\sco U_{\coU\ka}\ra\BR^m} is a local chart for~$M$ with
\hb{U_{\coU\ithka}:=\pl U_{\coU\ka}=U_{\coU\ka}\cap\pl M\neq\es}. Put
$$
\thka:=\ia_0\circ(\thia\,^*\ka)\sco U_{\coU\ithka}\ra\BR^{m-1},
$$
where
\hb{\ia_0\sco\{0\}\times\BR^{m-1}\ra\BR^{m-1}},
\ \hb{(0,x')\mt x'}. Let $\gK$ be a normalized atlas for~$M$. Then
a normalized atlas for~$\pl M$ is given by
\hb{\thgK:=\{\,\thka\ ;\ \ka\in\gK,\ \pl U_{\coU\ka}\neq\es\,\}}, the one
\emph{induced by}~$\gK$. Assume $(\rho,\gK)$ is a singularity datum
for~$M$. Set
\hb{\thrho:=\thia\,^*\rho=\rho\sn\pl M}. Then $(\thrho,\thgK)$ is a
singularity datum for~$\pl M$. Thus $\pl M$~is singular of
type~$[\![\thrho]\!]$.

\par
\hh{(f)}\quad 
If $M$~is a complete Riemannian manifold without boundary 
and with bounded geometry, then $M$~is uniformly regular.
\begin{proofb}
This follows from Lemma~2.2.6 in~\cite{Aub82a}, for example.
\end{proofb}
\end{examples}
In order to describe nontrivial classes of singular manifolds we need some
preparation. Let $N$ be a complete Riemannian
manifold without boundary and of dimension~$n$. Suppose $M$~is an
\hbox{$m$-dimensional}
submanifold of~$N$. Denote by~$\ol{M}$ the closure of~$M$ in~$N$. Then
\hb{\cS(M):=\ol{M}\ssm M} is the \emph{singularity set } of~$M$ (in~$N$).
Thus
\hb{\ol{M}=\ci M\cup\pl M\cup\cS(M)} and $\cS(M)$~is closed in~$N$. In
particular, $M$~is not complete if
\hb{\cS(M)\neq\es}.

\par
We assume now that $M$~can be described, locally in the
neighborhood of~$\cS(M)$, by model cusps and wedges over such cusps. More
precisely: suppose
\hb{d\in\BN^\times:=\BN\ssm\{0\}} and $B$~is a submanifold
of~$\BS^{d-1}$, the unit sphere in~$\BR^d$. Then
$$
K_1^d(B):=\{\,ry\in\BR^d\ ;\ 0<r<1,\ y\in B\,\},
$$
where
\hb{y\in B} is identified with its image in~$\BR^d$ under the natural
embedding
\hb{\BS^{d-1}\hr\BR^d}, is called \emph{model cone over}~$B$ in~$\BR^d$.

\par
Next, let
\hb{1<\al<\iy} and assume now that $B$~is a submanifold of~$Q^{d-1}$, where
\hb{d\geq2}. Then
$$
K_\al^d(B):=\bigl\{\,(r,r^\al y)\in\BR^d\ ;\ 0<r<1,\ y\in B\,\bigr\}
$$
is the \emph{model \hbox{$\al$-cusp}} in~$\BR^d$. To allow for a unified
treatment we call~$K_1^d$, in abuse of language, model \hbox{$1$-cusp}. Then,
given
\hb{\al\in[1,\iy)} and
\hb{\ell\in\BN},
$$
K_\al^d\Bl:=K_\al^d(B)\times Q^\ell
$$
is the \emph{model $(\al,\ell)$-wedge over}~$B$ in~$\BR^{d+\ell}$. Here and
below, all references to~$Q^\ell$ have to be neglected if
\hb{\ell=0}. Thus
\hb{K_\al^d(B,0)=K_\al^d(B)}, and a model cusp is a specific instance of a
model wedge.

\par
If
\hb{b:=\dim(B)}, then $K_\al^d\Bl$ is a submanifold of~$\BR^{d+\ell}$ of
dimension
\hb{b+1+\ell} and boundary $K_\al^d(\pl B,\ell)$. Thus
\hb{\pl K_\al^d\Bl=\es} if
\hb{\pl B=\es}, which is the case, in particular, if
\hb{\al=1} and
\hb{B=\BS^{d-1}}, or if
\hb{b=0}.

\par
Now we suppose
\hb{0\leq\ell\leq m-1} and $S$~is an \hbox{$\ell$-dimensional} submanifold
of~$N$ without boundary, contained in~$\cS(M)$. We also suppose
\hb{\al\in[1,\iy)} and $B$~is an
\hb{(m-\ell-1)}-dimensional submanifold of~$\BS^{m-\ell-1}$ if
\hb{\al=1},
or of~$Q^{m-\ell-1}$ if
\hb{\al>1}. Then $S$~is called
$(\al,\ell)$\emph{-wedge of\/~$M$ over~$B$} if for each
\hb{p\in S} there exists a normalized local chart~$\vp$ for~$N$ at~$p$
such that
\hb{\cS(M)\cap U_\vp=S\cap U_\vp},
$$
\vp(M\cap U_\vp)=\bigl(K_\al^{m-\ell}\Bl\times\{0\}\bigr)\cap Q^n,
$$
and
$$
\vp(S\cap U_\vp)=\bigl(\{0\}\times Q^\ell\bigr)\times\{0\}.
$$
Thus an $(\al,\ell)$-wedge of~$M$ over~$B$ looks locally like the model
wedge $K_\al^{m-\ell}\Bl$ in~$\BR^m$.

\par
Finally, $M$~is called \emph{relatively compact} 
(\emph{sub-}) \emph{manifold} (\emph{of}~$N$)
\emph{with smooth cuspidal singularities} if $\ol{M}$~is compact,
\hb{\cS(M)\neq\es}, and for each connected component~$\Ga$  of~$\cS(M)$
there exist
\hb{\al\in[1,\iy)},
\ \hb{\ell\in\{0,\ldots,m-1\}}, and a compact manifold~$B$ such that
$\Ga$~is an $(\al,\ell)$-wedge of~$M$ over~$B$.

\par
\hangindent=-180pt\hangafter=0
In the adjacent figure we have depicted a three-dimensional 
relatively compact submanifold~$M$
of~$\BR^3$ with smooth cuspidal singularities. More precisely,
$\cS(M)$~consists of $5$ connected components, namely of one
\hbox{$2.5$-cusp}, one $(2,1)$-wedge (the upper rim), and three
$(1,1)$-wedges (one at the bottom of the figure and two on the inner
plateau).

\par
\hangindent=-180pt\hangafter=0
Let $M$ be a relatively compact 
submanifold of~$N$ with smooth cuspidal singularities.
Denote by~$\mf{\Ga}$ the set of connected components of~$\cS(M)$. Since
$\cS(M)$~is closed in~$\ol{M}$, it is compact. Hence $\mf{\Ga}$~is a
finite set and each
\hb{\Ga\in\mf{\Ga}} is a compact submanifold of~$N$ without boundary.
\noindent
\hfill
\rlap{\hskip0em\smash{\vbox{
\begin{picture}(128,110)(-7,0)
\put(11,-10){\makebox(0,0)[bl]{\includegraphics[height=110pt]{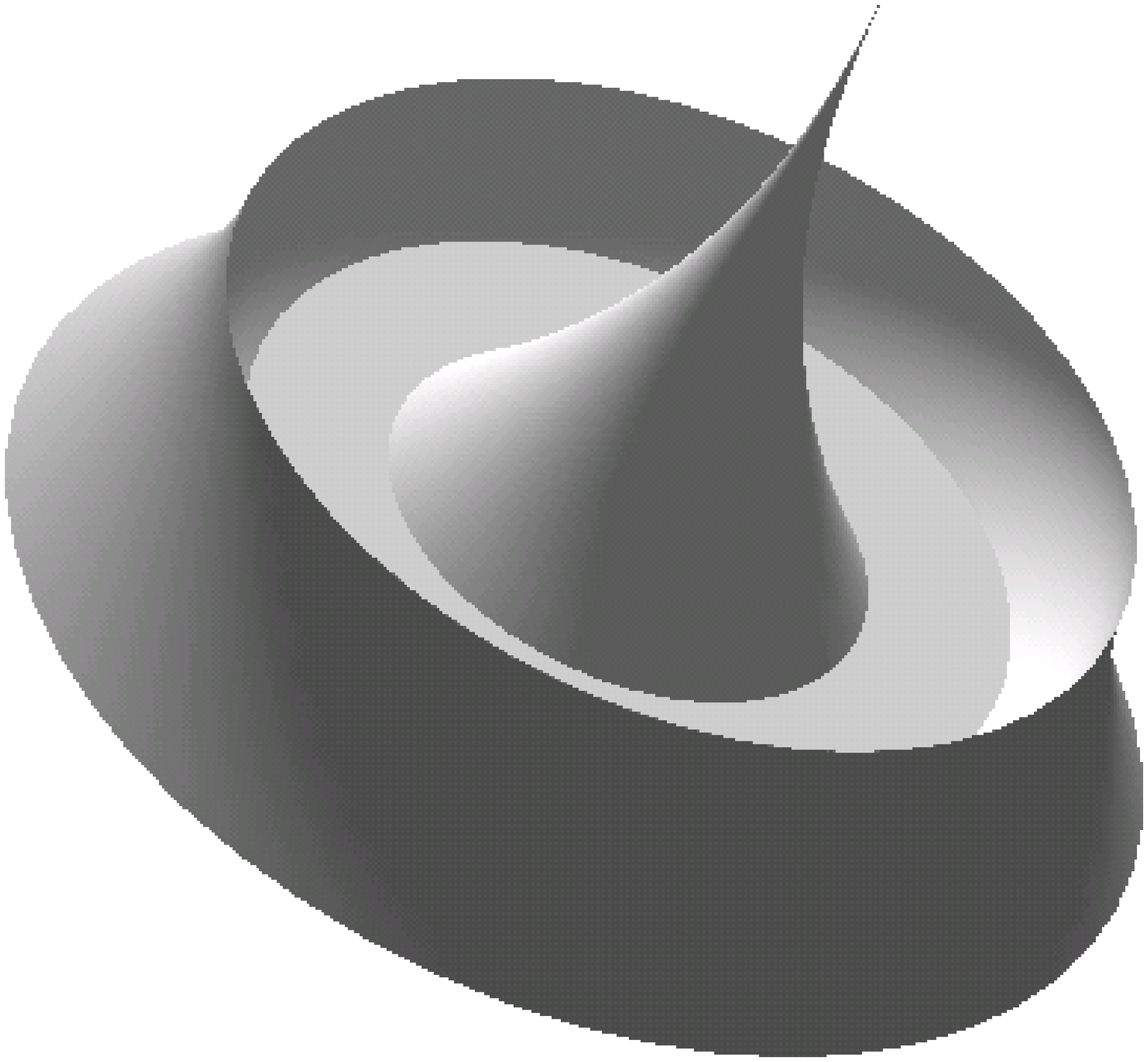}}}
\end{picture}
\vskip 1.5\baselineskip}}}

\par 
Given a nonempty subset~$S$ of~$\cS(M)$, we denote by~$d_N(p,S)$ the
Riemannian distance in~$N$ from
\hb{p\in N} to~$S$. For each
\hb{\Ga\in\mf{\Ga}} we can find a relatively compact open
neighborhood~$U_\Ga$ in~$N$ such that
\hb{d_N\bigl(p,\cS(M)\bigr)=d_N(p,\Ga)} for
\hb{p\in U_\Ga} and
$d_N(\cdot,\Ga)$  is smooth on~$U_\Ga$. Moreover, there exists a unique
\hb{\al_\Ga\in[1,\iy)} such that $\Ga$~is an
$\bigl(\al_\Ga,\dim(\Ga)\bigr)$-wedge
of~$M$ over some compact manifold~$B_\Ga$ of dimension
\hb{m-\dim(\Ga)-1}.
\begin{theorem}\label{thm-S.cusp}
Let $M$ be a relatively compact  
manifold with smooth cuspidal singularities. 

\par 
Choose 
\hb{\rho\in C^\iy\bigl(M,(0,1]\bigr)} satisfying
\hb{\rho(p)\sim\bigl(d_N(p,\Ga)\bigr)^{\al_\Ga}} for~$p$ near
\hb{\Ga\in\mf{\Ga}}. Then
$M$~is a singular manifold of type\/~$[\![\rho]\!]$.
\end{theorem}
\begin{proof}
H.~Amann~\cite{Ama11c}.
\end{proof}
In the case of the
manifold~$M$ depicted above, $\rho$~behaves near~$\cS(M)$ like the
power~$\al$ of the Euclidean distance in~$\BR^3$ to~$\cS(M)$, where
\hb{\al=2.5} near the vertex  of the cusp,
\hb{\al=2} near the upper rim, and
\hb{\al=1} near the remaining three wedges.

\par
For manifolds with
non-smooth cuspidal singularities we refer to~\cite{Ama11c}. There it
is no longer assumed that $B_\Ga$~is a
compact manifold, but $B_\Ga$~itself can have (non-) smooth
cuspidal singularities. This covers the case of corners and intersecting
wedges.
In addition, in~\cite{Ama11c} we consider singular manifolds which are 
not relatively compact; for example: subdomains of~$\BR^m$ with 
`outlets to infinity'.
\section{Tensor Fields and Uniform Estimates}\label{sec-U}
It is the purpose of this section to provide technical estimates on which
much of what follows is based. First we prepare some results on tensor
bundles and covariant derivatives. For general background information we
refer to J.~Dieudonn\'e~\cite{Die69b}, for instance.

\par
Let
\hb{M=\Mg} be an \hbox{$m$-dimensional} Riemannian manifold. We denote
by $TM$ and~$T^*M$ the (complexified, if
\hb{\BK=\BC}) tangent and cotangent bundle, respectively. Then, given
\hb{\sa,\tau\in\BN},
$$
T_\tau^\sa M:=TM^{\otimes\sa}\otimes T^*M^{\otimes\tau}
$$
is the $(\sa,\tau)$-tensor bundle of~$M$, that is, the vector bundle of all
tensors on~$M$ being contravariant of order~$\sa$ and covariant of
order~$\tau$. We use obvious conventions if
\hb{\sa=0} or
\hb{\tau=0}. In particular,
\hb{T_0^0M=M\times\BK}, a~trivial vector bundle. We write~$\cT_\tau^\sa M$
for the $C^\iy(M)$-module of all smooth sections of~$T_\tau^\sa M$, the
smooth $(\sa,\tau)$-tensor fields on~$M$. For abbreviation,
\hb{\cT M:=\cT_0^1M} and
\hb{\cT^*M:=\cT_1^0M}.

\par
For
\hb{\nu\in\BN^\times} we set
\hb{\BJ_\nu:=\{1,\ldots,m\}^\nu}. Then, given local coordinates
\hb{\ka=(x^1,\ldots,x^m)} and setting
$$
\frac\pl{\pl x^{(i)}}
:=\frac\pl{\pl x^{i_1}}\otimes\cdots\otimes\frac\pl{\pl x^{i_\sa}}
\qb dx^{(j)}:=dx^{j_1}\otimes\cdots\otimes dx^{j_\tau}
$$
for
\hb{(i)=(i_1,\ldots,i_\sa)\in\BJ_\sa},
\ \hb{(j)\in\BJ_\tau}, the local representation of
\hb{a\in\cT_\tau^\sa M} with respect to these coordinates is given by
\beq\label{U.aloc}
a=a_{(j)}^{(i)}\frac\pl{\pl x^{(i)}}\otimes dx^{(j)}
\eeq
with
\hb{a_{(j)}^{(i)}\in C^\iy(U_{\coU\ka})}. Here and below, we use the
summation conventions whereby expressions are summed over all possible
values of repeated indices.

\par
We write
\hb{g_\flat\sco\cT M\ra\cT^*M} for the conjugate linear (fiber-wise
defined) Riesz isomorphism. Thus
\beq\label{U.gb}
\dl g_\flat X,Y\dr=g\YX
\qa X,Y\in\cT M,
\eeq
where
\beq\label{U.dua}
\pw\sco\cT^*M\times\cT M\ra C^\iy(M)
\eeq
is the (fiber-wise defined) duality pairing. The inverse of~$g_\flat$,
denoted by~$g^\sh$, satisfies
$$
\dl\al,Y\dr=g(Y,g^\sh\al)
\qa \al\in\cT^*M
\qb X\in\cT M.
$$
Denoting by~$g^*$ the adjoint Riemannian metric on~$T^*M$ it follows from
\eqref{U.gb} that
\beq\label{U.ginv}
\dl\al,g^\sh\ba\dr=\dl g_\flat g^\sh\al,g^\sh\ba\dr
=g(g^\sh\ba,g^\sh\al)=g^*(\al,\ba)
\qa \al,\ba\in\cT^*M.
\eeq
From this we obtain, in local coordinates,
\beq\label{U.gdl}
g_\flat X=g_{ij}\ol{X}\vph{X}^j\,dx^i
\qb g^\sh\al=g^{ij}\ol{\al}_j\frac\pl{\pl x^i}
\qquad\text{for}\quad
X=X^i\frac\pl{\pl x^i}
\qb \al=\al_j\,dx^j,
\eeq
where
\hb{g=g_{ij}\,dx^i\otimes dx^j} and $[g^{ij}]$~is the inverse of the
matrix~$[g_{ij}]$ .

\par
We let
\beq\label{U.dup}
\pw\sco\cT_\tau^\sa M\times\cT_\sa^\tau M\ra C^\iy(M)
\eeq
be the natural extension of \eqref{U.dua}. Thus, given
\hb{p\in M}, we write $(T_\sa^\tau M)_p$ for the fiber
of $T_\sa^\tau M$ over~$p$. Then, for decomposable tensors
\hb{u\otimes\al\in(T_\tau^\sa M)_p} and
\hb{v\otimes\ba\in(T_\sa^\tau M)_p},
$$
\dl u\otimes\al,v\otimes\ba\dr_p
:=\prod_{i=1}^\sa\dl\ba_i,u_i\dr_p\prod_{j=1}^\tau\dl\al_j,v_j\dr_p,
$$
where
\hb{u=u_1\otimes\cdots\otimes u_\sa\in(T_0^\sa M)_p}  and
\hb{\al=\al_1\otimes\cdots\otimes\al_\tau\in(T_\tau^0M)_p}, etc. Hence
$$
(T_\tau^\sa M)'=T_\sa^\tau M
$$
with respect to the `tensor product duality pairing'~\eqref{U.dup}. This is
consistent with
\hb{(TM)'=T^*M}.

\par
Suppose
\hb{\sa+\tau\geq1}. We put
\beq\label{U.G}
(G_\sa^\tau a)(\al_1,\ldots,\al_\tau,X_1,\ldots,X_\sa)
:=a(g_\flat X_1,\ldots,g_\flat X_\sa,g^\sh\al_1,\ldots,g^\sh\al_\tau)
\eeq
for
\hb{a\in\cT_\tau^\sa M},
\ \hb{\al_1,\ldots,\al_\tau\in\cT^*M}, and
\hb{X_1,\ldots,X_\sa\in\cT M}. This induces a conjugate linear bijection
$$
G_\sa^\tau\sco T_\tau^\sa M\ra T_\sa^\tau M
\qb (G_\sa^\tau)^{-1}=G_\tau^\sa.
$$
Consequently,
\beq\label{U.gG}
\prsn_g\sco T_\tau^\sa M\times\cT_\tau^\sa M\ra C^\iy(M)
\qb (a,b)\mt\dl a,G_\sa^\tau b\dr
\eeq
is an inner product (a~vector bundle metric) on~$T_\tau^\sa M$, the
\emph{inner product induced by}~$g$. It follows from \eqref{U.gdl} that,
in local coordinates,
\beq\label{U.gabl}
(a\sn b)_g=g_{(i)(j)}g^{(k)(\ell)}a_{(k)}^{(i)}\ol{b}\vph{b}_{(\ell)}^{(j)}
\qa a,b\in\cT_\tau^\sa M,
\eeq
where
\beq\label{U.gg}
g_{(i)(j)}:=g_{i_1j_1}\cdots g_{i_\sa j_\sa}
\qb g^{(k)(\ell)}:=g^{k_1\ell_1}\cdots g^{k_\tau\ell_\tau}
\eeq
for
\hb{(i),(j)\in\BJ_\sa} and
\hb{(k),(\ell)\in\BJ_\tau}. Of course,
\hb{(a\sn b)_g=a\ol{b}} for
\hb{a,b\in\cT_0^0M=C^\iy(M)}. Clearly,
$$
\vsdot_g\sco\cT_\tau^\sa M\ra C(M)
\qb a\mt\sqrt{(a\sn a)_g}
$$
is called (vector bundle) \emph{norm} induced by~$g$. (We do not
notationally indicate  the dependence on~$(\sa,\tau)$. This will be clear
from the context.) Note that
\hb{|a|_g^2=g^*(a,a)} for
\hb{a\in T_1^0M}. For this reason we also write~$|a|_{g^*}$ for~$|a|_g$ if
\hb{a\in T_\tau^0M}.

\par
Let
\hb{\vp\sco M\ra N} be a diffeomorphism onto some manifold~$N$. Then one
verifies
$$
\vp_*\bigl((a\sn b)_g\bigr)=(\vp_*a\sn\vp_*b)_{\vp_*g}.
$$

\par
We denote by
\hb{\na=\na_{\cona g}} the (complexified, if
\hb{\BK=\BC}) Levi-Civita connection on~$TM$. It has a unique extension
over~$\cT_\tau^\sa$ satisfying, for
\hb{X\in\cT M},
\beq\label{U.LT}
\bal
\rm{(i)}  \qquad    &\na_{\cona X}f=\dl df,X\dr,
                     \ \ f\in C^\iy(M);\\
\rm{(ii)} \qquad    &\na_{\cona X}(a\otimes b)=\na_{\cona X}a
                     \otimes b+a\otimes\na_{\cona X}b,
                     \ \ a\in\cT_{\tau_1}^{\sa_1}M,
                     \ \ b\in\cT_{\tau_2}^{\sa_2}M;\\
\rm{(iii)}\qquad    &\na_{\cona X}\dl a,b\dr=\dl\na_{\cona X}a,b\dr
                     +\dl a,\na_{\cona X}b\dr,
                     \ \ a\in\cT_\tau^\sa M,
                     \ \ b\in\cT_\sa^\tau M.
\eal
\eeq
Then the covariant (Levi-Civita) derivative is the linear map
$$
\na=\na_{\cona g}\sco\cT_\tau^\sa M\ra\cT_{\tau+1}^\sa M
\qb a\mt\na a,
$$
defined by
$$
\dl\na a,b\otimes X\dr:=\dl\na_{\cona X}a,b\dr
\qa b\in\cT_\sa^\tau M
\qb X\in\cT M.
$$
Since it satisfies
\hb{\na g=0}, it commutes with $g_\flat$ and~$g^\sh$. From this we
infer
\beq\label{U.ip}
\na_{\cona X}(a\sn b)_g=(\na_{\cona X}a\sn b)_g+(a\sn\na_{\cona X}b)_g
\qa a,b\in\cT_\tau^\sa M
\qb X\in\cT M.
\eeq
Thus $\na$~is a metric connection on
\hb{\cT_\tau^\sa M=\bigl(T_\tau^\sa M,\prsn_g\bigr)}.

\par
Let
\hb{\vp\sco M\ra N} be a diffeomorphism. The uniqueness of the Levi-Civita
connection implies
$$
\vp_*(\na_{\cona g}a)=\na_{\cona\vp_*g}(\vp_*a)
\qa a\in\cT_\tau^\sa M.
$$
For
\hb{k\in\BN} we define
$$
\na^k\sco\cT_\tau^\sa M\ra\cT_{\tau+k}^\sa M
\qb a\mt\na^ka
$$
by
\hb{\na^0a:=a} and
\hb{\na^{k+1}:=\na\circ\na^k}.

\par
Now we are ready for the proof of the needed estimates. In the following,
$dV_{\coV g}$~denotes the Lebesgue volume measure for~$M$. Furthermore,
given
\hb{a\in\cT_\tau^\sa M} and a local chart~$\ka$, we write~$[\ka_*a]$ for
the
\hb{(m^\sa\times m^\tau)}-matrix whose general entry equals
\hb{(\ka_*a)_{(j)}^{(i)}=(a\circ\ka^{-1})_{(j)}^{(i)}}, with
\hb{(i)\in\BJ_\sa} and
\hb{(j)\in\BJ_\tau}.
\begin{lemma}\label{lem-U.g}
Let $(\rho,\gK)$ be a singularity datum for~$\Mg$. Then the following
estimates hold uniformly with respect to
\hb{\ka\in\gK}:
\begin{itemize}
\item[(i)]
${}$
\hb{\ka_*g\sim\rho_\ka^2g_m},
\ \hb{\ka_*g^*\sim\rho_\ka^{-2}g_m}.
\item[(ii)]
${}$
\hb{\rho_\ka^{-2}\,\|\ka_*g\|_{k,\iy}+\rho_\ka^2\,\|\ka_*g^*\|_{k,\iy}
   \leq c(k)},
\ \hb{k\in\BN}.
\item[(iii)]
${}$
\hb{\ka_*(dV_{\coV g})\sim\rho_\ka^m\,dV_{\coV g_m}}.
\item[(iv)]
If\/
\hb{r,\sa,\tau\in\BN}, then
\ \hb{\sum_{i=0}^r|\na_{\cona\ka_*g}^i(\ka_*a)|_{g_m}
\sim\sum_{|\al|\leq r}|\pa[\ka_*a]|_{g_m}}
\ for
\hb{a\in\cT_\tau^\sa M}.
\item[(v)]
Given
\hb{\sa,\tau\in\BN},
$$
\ka_*(|a|_g)\sim\rho_\ka^{\sa-\tau}\,|\ka_*a|_{g_m}
\qa a\in\cT_\tau^\sa M,
$$
and
$$
|\ka^*b|_g\sim\rho_\ka^{\sa-\tau}\ka^*(|b|_{g_m})
\qa b\in\cT_\tau^\sa Q_\ka^m.
$$
\end{itemize}
\end{lemma}
\begin{proof}
(1) The first part of claim~(i) is immediate from \eqref{S.sd}(iii)
and~(vi).

\par
(2) By~(i) and the symmetry of~$g$ the spectrum of the matrix~$[\ka_*g]$ is
contained in an interval of the form
\hb{\rho_\ka^2[1/c,c]} for
\hb{\ka\in\gK}. Hence $[\ka_*g]^{-1}$ has its spectrum in
\hb{\rho_\ka^{-2}[1/c,c]} for
\hb{\ka\in\gK}. This implies the second part of statement~(i) and
\beq\label{U.gin}
\|\ka_*g^*\|_\iy\leq c\rho_\ka^{-2}
\qa \ka\in\gK.
\eeq
Furthermore,
\beq\label{U.rg}
\rho_\ka^{-2}\ka_*g
=\Bigl(\frac{\ka_*\rho}{\rho_\ka}\Bigr)^2\ka_*(\rho^{-2}g).
\eeq
Thus assertion~(ii) follows from \eqref{S.sd}\hbox{(iv)--(vi)},
\eqref{U.gin}, \eqref{U.rg}, Leibniz' rule, and the formulas for
derivatives of inverses (cf.~Lemma~1.4.2 in H.~Amann~\cite{Ama09a}).

\par
(3) Writing, as usual,
\hb{\sqrt{g}:=\sqrt{\det[g]}}, statement~(iii) follows from~(i) and
\hb{\ka_*(dV_{\coV g})=\sqrt{\ka_*g}\,dV_{\coV g_m}}.

\par
(4) Recall that, setting
\hb{\na_{\cona i}:=\na_{\cona\pl_i}} with
\hb{\pl_i=\pl/\pl x^i},
\beq\label{U.dX}
\na_{\cona i}X=(\pl_iX^k+\Ga_{ij}^kX^j)\,\frac\pl{\pl x^k}
\qa X=X^k\,\frac\pl{\pl x^k},
\eeq
where the Christoffel symbols~$\Chr$ are given by
\beq\label{U.Ch}
2\Chr=g^{k\ell}(\pl_ig_{\ell j}+\pl_jg_{\ell i}-\pl_\ell g_{ij}).
\eeq
Suppose
\hb{a\in\cT_\tau^\sa M} has the local representation~\eqref{U.aloc}.
Correspondingly,
$$
\na a
=\na_ka_{(j)}^{(i)}\frac\pl{\pl x^{(i)}}\otimes dx^{(j)}\otimes dx^k.
$$
Then it follows from \eqref{U.LT} and \eqref{U.dX} that
\beq\label{U.dai}
\na_ka_{(j)}^{(i)}=\pl_ka_{(j)}^{(i)}
+\sum_{s=1}^\sa\Ga_{k\ell}^{i_s}a_{(j)}^{(i_1,\ldots,\ell,\ldots,i_\sa)}
-\sum_{t=1}^\tau\Ga_{kj_t}^{\ell}a_{(j_1,\ldots,\ell,\ldots,i_\tau)}^{(i)},
\eeq
where $\ell$~is at position~$s$ in the first sum and at position~$t$ in the
second sum (and the terms are added up from
\hb{\ell=1} to
\hb{\ell=m}). We set
\hb{\na_{(k)}:=\na_{k_r}\circ\cdots\circ\na_{k_1}} and
\hb{\pl_{(k)}:=\pl_{k_r}\circ\cdots\circ\pl_{k_1}} for
\hb{(k)\in\BJ_r} and
\hb{r\in\BN^\times}. Then, writing
\hb{\na^ra=\bigl(\na_{(k)}a_{(j)}^{(i)}\bigr)\,\frac\pl{\pl x^{(i)}}
\otimes dx^{(j)}\otimes dx^{(k)}}, we obtain from \eqref{U.dai}
\beq\label{U.dak}
\na_{(k)}a_{(j)}^{(i)}
=\pl_{(k)}a_{(j)}^{(i)}+b_{(j)(k)}^{(i)},
\eeq
where $b_{(j)(k)}^{(i)}$~is a linear combination of the elements of
$$
\bigl\{\,\pa a_{(\wt{\jmath})}^{(\wt{\imath})}
\ ;\ |\al|\leq r-1,\ (\wt{\imath})\in\BJ_\sa,
\ (\wt{\jmath})\in\BJ_\tau\,\bigr\},
$$
the coefficients being polynomials in the derivatives of the
Christoffel symbols of order at most
\hb{r-1-|\al|}.

\par
We deduce from (ii) and \eqref{U.Ch}
\beq\label{U.Chl}
\|\Chr\circ\ka^{-1}\|_{\ell,\iy}\leq c(\ell)
\qa 1\leq i,j,k\leq m
\qb \ka\in\gK
\qb \ell\in\BN.
\eeq
Hence \eqref{U.dak} implies
$$
\sum_{i=0}^r|\na_{\cona\ka_*g}^i(\ka_*a)|_{g_m}
\leq c\sum_{|\al|\leq r}|\pa[\ka_*a]|_{g_m}
\qa a\in\cT_\tau^\sa M
\qb \ka\in\gK.
$$
By solving system~\eqref{U.dak} for $\pa a_{(j)}^{(i)}$ we obtain an
analogous expression for $\pl_{(k)}a_{(j)}^{(i)}$ in terms
of~$\na_{(\ell)}(\ka_*a)$,
\ \hb{\ell\in\BJ_\sa},
\ \hb{0\leq\sa\leq r-1}. Thus, invoking~\eqref{U.Chl} once more, we get
the second half of assertion~(iv).

\par
(5) The first part of~(v) follows from \eqref{U.gabl}, \eqref{U.gg},
and~(ii). The second part is then deduced by applying this result to
\hb{a:=\ka^*b}.
\end{proof}
From \eqref{S.sd}(v) and~(vi) and Lemma~\ref{lem-U.g}(ii) we find by the
arguments of step~(2)
\beq\label{U.bg}
\big\|\ka_*\big((\rho^{-2}g)^*\big)\big\|_{k,\iy}\leq c(k)
\qa \ka\in\gK
\qb k\in\BN.
\eeq
This, in combination with \eqref{S.sd}(iii) and~(iv), is close to the
statement that 
all covariant derivatives of the curvature tensor of 
$(M,\rho^{-2}g)$ are bounded. 
Note however that, taking
\eqref{S.eq} into consideration, \eqref{S.sd}(iv) and \eqref{U.bg} are
only true for atlases in~$\gS(M)$.

\par
Let $M$ be a manifold and $\gK$~an atlas for it consisting of normalized
charts. A~family
\hb{\bigl\{\,(\pi_\ka,\chi_\ka)\ ;\ \ka\in\gK\,\bigr\}} is~a (uniform)
\emph{localization system subordinate to}~$\gK$ if
\beq\label{U.LS}
\bal
\rm{(i)}  \qquad    &\pi_\ka\in\cD\bigl(U_{\coU\ka},[0,1]\bigr)\text{ and }
                     \{\,\pi_\ka^2\ ;\ \ka\in\gK\,\}
                     \text{ is a partition of unity subordinate to }
                     \{\,U_{\coU\ka}\ ;\ \ka\in\gK\,\};\\
\rm{(ii)} \qquad    &\chi_\ka=\ka^*\chi\text{ with }
                     \chi\in\cD\bigl(Q^m,[0,1]\bigr)
                     \text{ and }\chi\sn\supp(\ka_*\pi_\ka)=\mf{1};\\
\rm{(iii)}\qquad    &\|\ka_*\pi_\ka\|_{k,\iy}+\|\ka_*\chi_\ka\|_{k,\iy}
                     \leq c(k),
                     \ \ \ka\in\gK,
                     \ \ k\in\BN.
\eal
\eeq

\par
The crucial assumption, besides~(i), is the uniform
estimate~(iii). Assumption~(ii) will
simplify some formulas. In principle, it would suffice to require that
$\chi_\ka$~be a cut-off function for~$\supp(\pi_\ka)$.

\par
It should also be noted that, for the purpose of this paper, we could
replace~$\pi_\ka^2$ in~\eqref{U.LS}(i) by~$\pi_\ka$. In fact,
then some of the computations below would even become simpler. However,
in applications to differential equations it will be important that we can
use a partition of unity whose square root is smooth. For this reason
we employ condition~\eqref{U.LS}(i).
\begin{lemma}\label{lem-U.LS}
Let $(\rho,\gK)$ be a singularity datum for~$M$. Then there exists a
localization system subordinate to~$\gK$.
\end{lemma}
\begin{proof}
Fix
\hb{r\in(0,1)} such that
\hb{r\gU:=\bigl\{\,\ka^{-1}(rQ_\ka^m)\ ;\ \ka\in\gK\,\bigr\}} is a cover
of~$M$. Choose
\hb{\wt{\pi}\in\cD\bigl(Q^m,[0,1]\bigr)} with
\hb{\wt{\pi}\sn rQ^m=\mf{1}}. Set
\hb{\wt{\pi}_\ka:=\ka^*\wt{\pi}}. Since $r\gU$~covers~$M$ and has finite
multiplicity,
$$
1\leq\sum_\ka\wt{\pi}_\ka^2(p)\leq c
\qa p\in M.
$$
Put
\hb{\pi_\ka:=\wt{\pi}_\ka\Big/\sqrt{\sum_{\wt{\ka}}\wt{\pi}_{\wt{\ka}}^2}}.
Then
\hb{\pi_\ka\in\cD\bigl(U_{\coU\ka},[0,1]\bigr)} and
\hb{\sum_\ka\pi_\ka^2=\mf{1}}, where $\ka_*(\pi_\ka)$~has its support
in $\supp(\wt{\pi})$. Fix
\hb{\chi\in\cD\bigl(Q^m,[0,1]\bigr)} with
\hb{\chi\sn\supp(\wt{\pi})=\mf{1}}. Set
\hb{\chi_\ka:=\ka^*\chi}. Then conditions \eqref{U.LS}(i) and~(ii) are
satisfied. The validity of \eqref{U.LS}(iii) is a consequence
of~\eqref{S.sd}(ii).
\end{proof}
\section{Distribution Sections}\label{sec-D}
Given locally convex spaces $\cX$ and~$\cY$, we denote by~$\cL\cXcY$ the
space of continuous linear maps from~$\cX$ into~$\cY$, and
\hb{\cL(\cX):=\cL\cXcX}. By $\Lis\cXcY$ we mean the set of all topological
isomorphisms in~$\cL\cXcY$. If $\cX$ and~$\cY$ are Banach spaces, then
$\cL\cXcY$~is endowed with the uniform operator norm. We write
\hb{\pw_\cX} for the duality pairing between $\cX'$ and~$\cX$, that is,
$\dl x',x\dr_\cX$ is the value of
\hb{x'\in\cX'} at
\hb{x\in\cX}.

\par
Let
\hb{M=\Mg} be a Riemannian manifold. Suppose
\hb{V=(V,\pi,M)} is a \hbox{$\BK$-vector} bundle over~$M$. For a subset $S$
of~$M$ we denote by~$V_S$ the restriction of~$V$ to~$S$, that is,
\hb{V_S=\pi^{-1}(S)}. If
\hb{k\in\BN\cup\{\iy\}} and $S$~is open in~$M$, then $C^k\SV$ is the
$C^k(S)$-module of \hbox{$C^k$-sections} over~$S$.

\par
We denote by
\hb{V'=V^*} the dual vector bundle and by
\hb{\pw} the fiber-wise defined duality pairing between $V'$ and~$V$. We
also assume that $V$~is equipped with an inner product and write
\hb{\vsdot_V} for the corresponding vector bundle norm.

\par
Given an open subset~$S$ of~$M$ and
\hb{q\in[1,\iy]}, the Lebesgue space
\hb{L_q\SV=\bigl(L_q\SV,\Vsdot_q\bigr)} is the Banach space of all
(equivalence classes of measurable) sections~$v$ of~$V$ over~$S$ such that
$$
\|v\|_q=\|v\|_{L_q\SV}:=\big\|\,|v|_V\big\|_{L_q(S)}<\iy,
$$
where
\hb{L_q(S)=L_q(S,\BK;dV_{\coV g})}.

\par
In the following, we write
\hb{U\is\is V} to mean that $U$ and~$V$ are open, $U$~is relatively compact,
and
\hb{\ol{U}\is V}. Since $M$~is locally compact, separable, and metrizable
it is \hbox{$\sa$-compact}. Thus there exists a sequence~$(M_j)$ such that
\hb{M_j\is\is M_{j+1}} and
\hb{\bigcup_jM_j=M}. Hence $L_{1,\loc}\MV$, the vector space of
sections~$v$ of~$V$ such that
\hb{v\sn S\in L_1\SV} for every
\hb{S\is\is M}, is a Fr\'echet space.

\par
We denote by $\cD\ciMV$ and~$\cD\MV$ the spaces of smooth sections
of~$V$ being compactly supported in $\ci M$ and~$M$, respectively. For
\hb{S\is\is\ci M}, or
\hb{S\is\is M}, we write $\cD_S\ciMV$, respectively $\cD_S\MV$, for the
linear subspace of all
\hb{v\in\cD\ciMV}, respectively
\hb{v\in\cD\MV}, with
\hb{\supp(v)\is\ol{S}}. Then $\cD_S\ciMV$ and $\cD_S\MV$ are Fr\'echet spaces
(e.g., Section~VII.2 of J.~Dieudonn\'e~\cite{Die69b}). If
\hb{S\is\is S_1}, then
\hb{\cD_S\ciMV\is\cD_{S_1}\ciMV} and $\cD_{S_1}\ciMV$ induces
on $\cD_S\ciMV$ its original topology. Hence we can endow
$\cD\ciMV$ with the $LF$~topology (the strict inductive limit topology)
with respect to all such subspaces of~$\cD\ciMV$. Similarly,
$\cD\MV$ is given the $LF$~topology with respect to the subspaces
$\cD_S\MV$. Then
\beq\label{D.DDs}
\cD'\ciMV :=\cD\ciMVs_{w^*}'
\eeq
is the space of distribution sections on~$\ci M$, endowed with the weak$^*$
topology.

\par
Given
\hb{v\in L_{1,\loc}\ciMV},
\beq\label{D.L1}
\Bigl(u\mt\dl v,u\dr_\cD:=\int_M\dl v,u\dr\,dV_{\coV g}\Bigr)
\in\cD'\ciMV,
\eeq
and the map
$$
L_{1,\loc}\ciMV\ra\cD'\ciMV
\qb v\mt\dl v,\cdot\dr_\cD
$$
is linear, continuous, and injective. We identify
\hb{v\in L_{1,\loc}\ciMV} with the distribution section~\eqref{D.L1} and
consider $L_{1,\loc}\ciMV$ as a linear subspace of $\cD'\ciMV$. Then
\beq\label{D.DDL}
\cD\ciMV\hr\cD\MV\sdh L_{1,\loc}\MV\sdh L_{1,\loc}\ciMV\hr\cD'\ciMV,
\eeq
where
\hb{{}\hr{}} means `continuous' and
\hb{{}\sdh{}} `continuous and dense' embedding. Given
\hb{f\in C^\iy(M)}, the point-wise multiplication
\hb{u\mt fu} belongs to $\cL\bigl(\cD(\ci M,V')\bigr)$. Hence, setting
$$
(fT)(u):=T(fu)
\qa T\in\cD'\ciMV
\qb u\in\cD(\ci M,V'),
$$
it follows
\hb{(T\mt fT)\in\cL\bigl(\cD'\ciMV\bigr)}. We often identify~$f$ with
this `point-wise multiplication' operator.

\par
Suppose
\hb{k,\ell\in\BN} satisfy
\hb{k+\ell\geq1} and
\hb{E=\bigl(\BK^{k\times\ell},\pr_{HS}\bigr)}, where
$$
\pr_{HS}\sco E\times E\ra\BK
\qb (a,b)\mt\trace(b^*a)
$$
is the Hilbert-Schmidt inner product,
\hb{b^*\in\BK^{\ell\times k}} being the conjugate matrix of~$b$. Then
\beq\label{D.E}
E\times E\ra\BK
\qb (a,b)\mt(a\sn\ol{b})_{HS}
\eeq
is a separating bilinear form, the duality pairing of~$E$, by which we
identify~$E'$ with~$E$.

\par
Consider the trivial bundle
\hb{M\times E}. As usual, we write $\cD\ME$ for
\hb{\cD(M,M\times E)} etc. By juxtaposition of the rows of a matrix
\hb{a\in\BK^{k\times\ell}} we fix an isomorphism from $\BK^{k\times\ell}$
onto~$\BK^n$, where
\hb{n=k\ell}.  By means of it we identify $\cD\ME$ with $\cD(M)^n$,~etc.
Then
\beq\label{D.pr}
T(u)=\sum_{i=1}^nT_i(u_i)
\qa (T,u)\in\cD'\ciME\times\cD\ciME,
\eeq
where
\hb{u=(u_1,\ldots,u_n)\in\cD(\ci M)^n},~etc.

\par
Assume
\hb{\BX=\bigl(\BX,\prsn_{g_m}\bigr)} with
\hb{\BX\in\{\BR^m,\BH^m\}}. Let $\cS\BXE$ be the Schwartz space of rapidly
decreasing smooth \hbox{$E$-valued} functions on~$\BX$. Then
$\cS\ciBXE$~is the closure of $\cD\ciBXE$ in $\cS\BXE$, and
$$
\cS'\ciBXE:=\cS\ciBXE_{w^*}'
$$
is the space of \hbox{$E$-valued} tempered
distributions on~$\ci\BX$. Since
\hb{\ci\BX=\BR^m} if
\hb{\BX=\BR^m}, our notation is consistent with the well-known fact
\hb{\cD\RmE\sdh\cS\RmE}.

\par
Set
\hb{V:=\bigl(\BX\times E,\prsn_{HS}\bigr)} and note that
\hb{\dl v,\cdot\dr_\cD}, defined by \eqref{D.L1} and \eqref{D.E}, is for each
\hb{v\in\cD\MV} continuous with respect to the topology induced by
$\cS\BXE$ on~$\cD\ciBXE$. From this it follows
\beq\label{D.DSS}
\cD\ciBXE\sdh\cS\ciBXE\hr\cS\BXE\hr\cS'\ciBXE\hr\cD'\ciBXE.
\eeq
By mollifying we further obtain
\beq\label{D.DdD}
\cD\ciBXE\sdh\cD'\ciBXE.
\eeq

\par
For
\hb{u\in\cS'\RmE} we let $r^+$ be the restriction of~$u$ to~$\ci\BH^m$
in the sense of distributions, that is,
$$
\dl r^+u,\vp\dr_{\cS\ciBHmE}=\dl u,\vp\dr_{\cS\RmE}
\qa \vp\in\cS\ciBHmE.
$$
Then
\hb{r^+\in\cL\bigl(\cS'\RmE,\cS'\ciBHmE\bigr)}.

\par
If no confusion seems likely we use the same symbol for a linear map and its
restriction to a linear subspace of its domain. Furthermore, in a diagram
arrows always represent continuous linear maps.

\par
Recall that~a \emph{retraction}
\hb{\cX\ra\cY}, where $\cX$ and~$\cY$ are locally convex spaces, is a
continuous linear map possessing a continuous right inverse, a~coretraction.
Thus the following lemma guarantees that $r^+$~is a retraction.
\begin{lemma}\label{lem-D.ex}
There exists an extension operator~$e^+$ such that the diagram
 $$
 \begin{picture}(242,63)(-73,-5)        
 \put(0,50){\makebox(0,0)[b]{\small{$e^+$}}}
 \put(0,5){\makebox(0,0)[b]{\small{$e^+$}}}
 \put(100,50){\makebox(0,0)[b]{\small{$r^+$}}}
 \put(100,5){\makebox(0,0)[b]{\small{$r^+$}}}
 \put(-50,45){\makebox(0,0)[c]{\small{$\cS\BHmE$}}}
 \put(-50,0){\makebox(0,0)[c]{\small{$\cS'\ciBHmE$}}}
 \put(50,45){\makebox(0,0)[c]{\small{$\cS\RmE$}}}
 \put(50,0){\makebox(0,0)[c]{\small{$\cS'\RmE$}}}
 \put(150,45){\makebox(0,0)[c]{\small{$\cS\BHmE$}}}
 \put(150,0){\makebox(0,0)[c]{\small{$\cS'\ciBHmE$}}}
 \put(-45,22.5){\makebox(0,0)[l]{\small{$d$}}}
 \put(55,22.5){\makebox(0,0)[l]{\small{$d$}}}
 \put(155,22.5){\makebox(0,0)[l]{\small{$d$}}}
 \put(-20,45){\vector(1,0){40}}
 \put(-20,0){\vector(1,0){40}}
 \put(80,45){\vector(1,0){40}}
 \put(80,0){\vector(1,0){40}}
 \put(-50,33){\vector(0,-1){23}}
 \put(-48,33){\oval(4,4)[t]}
 \put(50,33){\vector(0,-1){23}}
 \put(52,33){\oval(4,4)[t]}
 \put(150,33){\vector(0,-1){23}}
 \put(152,33){\oval(4,4)[t]}
 \end{picture}
 $$
is commuting and
\hb{r^+e^+=\id}.
\end{lemma}
\begin{proof}
As in \eqref{D.pr} we identify $\cS\BXE$ with $\cS(\BX)^n$ and
$\cS'\ciBXE$ with $\cS'(\ci\BX)^n$. Then the assertion follows from
Theorems 4.2.2 and~4.2.4 in~\cite{Ama09a} (with
\hb{F:=\BK}).
\end{proof}
It is a consequence of this lemma, \eqref{D.DDL}, \eqref{D.DSS}, and
\eqref{D.DdD} that
$$
\cD\BXE\hr\cS\BXE\sdh\cS'\ciBXE\sdh\cD'\ciBXE
$$
and
\beq\label{D.dDu}
\cD\BXE\sdh\cD'\ciBXE,
\eeq
due to
\hb{\cD\ciBXE\is\cD\BXE}.
\section{Localization of Distribution Sections}\label{sec-L}
Let $\sA$ be a countable index set. Suppose $\cX_\al$~is for each
\hb{\al\in\sA} a~locally convex space. We endow $\prod_\al\cX_\al$ with the
product topology, that is, the coarsest locally convex topology for
which all projections
\hb{\pro_\ba\sco\prod_\al\cX_\al\ra\cX_\ba},
\ \hb{\mf{x}=(x_\al)\mt x_\ba} are continuous. By $\bigoplus_\al\cX_\al$ we
mean the locally convex direct sum. Thus $\bigoplus_\al\cX_\al$~is the
vector subspace of~$\prod_\al\cX_\al$ consisting of all finitely supported
\hb{\mf{x}\in\prod_\al\cX_\al}, equipped with the inductive topology,
that is, the finest locally convex topology for which all  injections
\hb{\cX_\ba\ra\bigoplus_\al\cX_\al} are continuous. Let
\hb{\pw_\al} be the \hbox{$\cX_\al$-duality} pairing. Then
$$
\mfpw\sco\prod_\al\cX_\al'\times\bigoplus_\al\cX_\al\ra\BK
\qb (\mf{x}',\mf{x})\mt\sum_\al\dl x_\al',x_\al\dr_\al
$$
is a separating bilinear form, and (cf.~Corollary~1 in Section~IV.4.3 of
H.H.~Schaefer~\cite{Schae71a})
\beq\label{L.dua}
\Bigl(\bigoplus_\al\cX_\al\Bigr)_{w^*}'=\prod_\al(\cX_\al)_{w^*}'
\eeq
with respect to~%
\hb{\mfpw}, (that is,
\hb{\mfpw}~is the $\bigoplus_\al\cX_\al$-duality pairing).

\par
Throughout the rest of this paper we assume
$$
\frame{
\begin{minipage}{250pt}
$$
\bal
\bt\quad
&M=\Mg\text{ is an $m$-dimensional singular manifold}.\\
\bt\quad
&\rho\in\gT(M).\\
\bt\quad
&\sa,\tau\in\BN\text{ and }V=V_\tau^\sa:=\bigl(T_\tau^\sa M,\prsn_g\bigr).\\
\noalign{\vskip2.5\jot}
\eal
$$
\end{minipage}}
$$
It follows that we can choose
\beq\label{L.sd}
\bal
\bt\quad
&\text{a singularity datum }(\rho,\gK), \\
\bt\quad
&\text{a localization system }
 \bigl\{\,(\pi_\ka,\chi_\ka)\ ;\ \ka\in\gK\,\bigr\}
 \text{ subordinate to }\gK.
\eal
\eeq
For
\hb{K\is M} we put
\hb{\gK_K:=\{\,\ka\in\gK\ ;\ U_{\coU\ka}\cap K\neq\es\,\}}. Then, given
\hb{\ka\in\gK},
$$
\BX_\ka:=
\left\{
\bal
&\BR^m  &&\quad \text{if }\ka\in\gK\ssm\gK_{\pl M},\\
&\BH^m  &&\quad \text{otherwise},
\eal
\right.
$$
endowed with the Euclidean metric~$g_m$.

\par
We set
$$
E=E_\tau^\sa:=\bigl(\BK^{m^\sa\times m^\tau},\prsn_{HS}\bigr)
$$
and consider the trivial bundles
\hb{V_{\coV\ka}:=\bigl(\BX_\ka\times E,\prsn_{g_m}\bigr)} for
\hb{\ka\in\gK}.
For abbreviation,
$$
\mf{\cD}\ciBXE:=\bigoplus_\ka\cD\ciBXkE
\qb \mf{\cD}\BXE:=\bigoplus_\ka\cD\BXkE,
$$
and
$$
\mf{\cD}'\ciBXE:=\prod_\ka\cD'\ciBXkE.
$$
It follows from \eqref{L.dua} that
\hb{\mf{\cD}'\ciBXE=\mf{\cD}(\ci\BX,E')_{w^*}'}, where
\hb{E'=E_\sa^\tau}.

\par
We introduce linear maps
$$
\vp_\ka\sco\cD\MV\ra\cD\BXkE
\qb u\mt\ka_*(\pi_\ka u)
$$
and
$$
\psi_\ka\sco\cD\BXkE\ra\cD\MV
\qb v_\ka\mt\pi_\ka\ka^*v_\ka
$$
for
\hb{\ka\in\gK}. Here and in similar situations it is understood that a
partially defined and compactly supported section of a vector bundle is
extended over the whole base manifold by identifying it with the zero
section outside its original domain. Moreover,
$$
\vp\sco\cD\MV\ra\mf{\cD}\BXE
\qb u\mt(\vp_\ka u)
$$
and
$$
\psi\sco\mf{\cD}\BXE\ra\cD\MV
\qb \mf{v}\mt{\textstyle\sum_\ka}\psi_\ka v_\ka.
$$
The following retraction theorem shows, in particular, that these maps are
well-defined and possess unique continuous linear extensions to distribution
sections.
\begin{theorem}\label{thm-L.ret}
The diagram
$$
 \begin{picture}(237,63)(-69,-5)        
 \put(0,50){\makebox(0,0)[b]{\small{$\vp$}}}
 \put(0,5){\makebox(0,0)[b]{\small{$\vp$}}}
 \put(100,50){\makebox(0,0)[b]{\small{$\psi$}}}
 \put(100,5){\makebox(0,0)[b]{\small{$\psi$}}}
 \put(-50,45){\makebox(0,0)[c]{\small{$\cD\MV$}}}
 \put(-50,0){\makebox(0,0)[c]{\small{$\cD'\ciMV$}}}
 \put(50,45){\makebox(0,0)[c]{\small{$\mf{\cD}\BXE$}}}
 \put(50,0){\makebox(0,0)[c]{\small{$\mf{\cD}'\ciBXE$}}}
 \put(150,45){\makebox(0,0)[c]{\small{$\cD\MV$}}}
 \put(150,0){\makebox(0,0)[c]{\small{$\cD'\ciMV$}}}
 \put(-45,22.5){\makebox(0,0)[l]{\small{$d$}}}
 \put(55,22.5){\makebox(0,0)[l]{\small{$d$}}}
 \put(155,22.5){\makebox(0,0)[l]{\small{$d$}}}
 \put(-20,45){\vector(1,0){40}}
 \put(-20,0){\vector(1,0){40}}
 \put(80,45){\vector(1,0){40}}
 \put(80,0){\vector(1,0){40}}
 \put(-50,33){\vector(0,-1){23}}
 \put(-48,33){\oval(4,4)[t]}
 \put(50,33){\vector(0,-1){23}}
 \put(52,33){\oval(4,4)[t]}
 \put(150,33){\vector(0,-1){23}}
 \put(152,33){\oval(4,4)[t]}
 \end{picture}
 $$
 is commuting and
 \hb{\psi\circ\vp=\id}.
\end{theorem}
\begin{proof}
(1) We set
\beq\label{L.ph0k}
\ci\vp_\ka u:=\sqrt{\ka_*g}\,\ka_*(\pi_\ka u)
\qa u\in\cD(\ci M,V')
\qb \ka\in\gK.
\eeq
Suppose
\hb{K\is\is\ci M}. Then
\hb{L_\ka:=\ka\bigl(K\cap\dom(\chi_\ka)\bigr)\is\is\ci\BX_\ka}. Assume
\hb{u\in\cD_K(\ci M,V')}. Then $\ka_*(\pi_\ka u)$ belongs to
$\cD_{L_\ka}(\ci\BX_\ka,V_{\coV\ka}')$. Since
\hb{\sqrt{\ka_*g}\in C^\iy(Q_\ka^m)}, it follows
$$
\ci\vp_\ka\in\cL\bigl(\cD_K(\ci M,V'),\cD(\ci\BX_\ka,V_{\coV\ka}')\bigr)
\qa \ka\in\gK,
$$
due to
\hb{\cD_{L_\ka}(\ci\BX_\ka,V_{\coV\ka}')\hr\cD(\ci\BX_\ka,V')}.
This being true for each
\hb{K\is\is\ci M}, we obtain
$$
\ci\vp_\ka\in\cL\bigl(\cD(\ci M,V'),\cD(\ci\BX_\ka,V_{\coV\ka}')\bigr)
\qa \ka\in\gK.
$$

\par
(2) We put
\beq\label{L.ps0k}
\ci\psi_\ka v:=\pi_\ka\ka^*\Bigl(\bigl(\sqrt{\ka_*g}\bigr)^{-1}\chi v\Bigr)
\qa v\in\cD(\ci\BX_\ka,V_{\coV\ka}')
\qb \ka\in\gK.
\eeq
Suppose
\hb{L_\ka\is\is\ci\BX_\ka} and set
\hb{K_\ka:=\ka^{-1}\bigl(L_\ka\cap\dom(\chi)\bigr)}. Then
\hb{K_\ka\is\is\ci M}. Similarly as above, we find that $\ci\psi_\ka$~maps
$\cD_{L_\ka}(\ci\BX_\ka,V_{\coV\ka}')$ continuously into
$\cD(\ci M,V')$. Consequently,
$$
\ci\psi_\ka\in\cL\bigl(\cD(\ci\BX_\ka,V_{\coV\ka}'),\cD(\ci M,V')\bigr).
$$

\par
(3) Set
$$
\ci\vp u:=(\ci\vp_\ka u)
\qa u\in \cD(\ci M,V').
$$
Assume
\hb{K\is\is\ci M}. Since $\gK$~is uniformly shrinkable there exist
\hb{r\in(0,1)} and a finite subset~$\gL_K$ of~$\gK$ such that
\hb{\bigl\{\,\ka^{-1}(rQ_\ka^m)\ ;\ \ka\in\gL_K\,\bigr\}} is a cover
of~$K$. Put
$$
\gM_K:=\{\,\ka\in\gK
\ ;\ \text{there exists $\wt{\ka}\in\gL_K$ with }
U_{\coU\wt{\ka}}\cap U_{\coU\ka}\neq\es\,\}.
$$
Then $\gM_K$~is a finite set, due to the finite multiplicity of~$\gK$.
Since
\hb{\ci\vp_\ka u=0} for
\hb{u\in\cD_K(\ci M,V')} and
\hb{\ka\in\gK\ssm\gM_K} it follows from step~(1) that $\ci\vp$~maps
$\cD_K(\ci M,V')$ continuously into the closed linear subspace
$$
\bigl\{\,\mf{v}\in\mf{\cD}(\ci\BX,E')
\ ;\ v_\ka=0\text{ for }\ka\in\gK\ssm\gM_K\,\bigr\}
$$
of $\mf{\cD}(\ci\BX,E')$, hence into $\mf{\cD}(\ci\BX,E')$. Since this is
true for all
\hb{K\is\is\ci M},
\beq\label{L.ph0D}
\ci\vp\in\cL\bigl(\cD(\ci M,V'),\mf{\cD}(\ci\BX,E')\bigr).
\eeq

\par
(4) Put
$$
\ci\psi\mf{v}:=\sum_\ka\ci\psi_\ka v_\ka
\qa \mf{v}=(v_\ka)\in\mf{\cD}(\ci\BX,E').
$$
Let $\gL$ be a finite subset of~$\gK$ and put
$$
\cX_\gL:=\bigl\{\,\mf{v}\in\mf{\cD}(\ci\BX,E')
\ ;\ v_\ka=0\text{ if }\ka\in\gK\ssm\gL\,\bigr\}.
$$
Step~(2) implies that $\ci\psi$~maps~$\cX_\gL$ continuously into
$\cD(\ci M,V')$. Thus, since this holds for all finite subset~$\gL$ of~$\gK$,
\beq\label{L.ps0D}
\ci\psi\in\cL\bigl(\mf{\cD}(\ci\BX,E'),\cD(\ci M,V')\bigr).
\eeq

\par
(5) For
\hb{u\in \cD(\ci M,V')} and
\hb{\ka\in\gK} it follows from
\hb{\pi_\ka\chi_\ka=\pi_\ka} and
\hb{\chi_\ka=\ka^*\chi} that
\hb{(\ci\psi_\ka\circ\ci\vp_\ka)u=\pi_\ka^2u}. Hence
\hb{\sum_\ka\pi_\ka^2=\mf{1}} implies
$$
(\ci\psi\circ\ci\vp)u=\sum_\ka\psi_\ka(\vp_\ka u)
=\sum_\ka\pi_\ka^2u=u
\qa u\in\cD(\ci M,V').
$$
Thus $\ci\psi$~is a retraction from $\mf{\cD}(\ci\BX,E')$ onto
$\cD(\ci M,V')$, and $\ci\vp$~is a coretraction.

\par
(6) Steps (3) and~(4) and relations \eqref{D.DDs} and \eqref{L.dua} imply
$$
\Psi:=(\ci\vp)'\in\cL\bigl(\mf{\cD}'\ciBXE,\cD'\ciMV\bigr)\ph{.}
$$
and
$$
\Phi:=(\ci\psi)'\in\cL\bigl(\cD'\ciMV,\mf{\cD}'\ciBXE\bigr).
$$
By step~(5),
$$
\Psi\circ\Phi=(\ci\psi\circ\ci\vp)'
=(\id_{\cD(\ci M,V')})'=\id _{\cD'(\ci M,V)}.
$$

\par
(7) Suppose
\hb{v\in\cD\MV} and
\hb{\mf{u}\in\mf{\cD}(\ci\BX,E')}. Then, see \eqref{D.L1},
$$
\bal
\mfdl\Phi v,\mf{u}\mfdr
 =\dl v,\ci\psi\mf{u}\dr_\cD
&=\sum_\ka\dl v,\ci\psi_\ka u_\ka\dr_\cD
 =\sum_\ka\int_M\pi_\ka\big\dl v,(\sqrt{\ka_*g}\,)^{-1}
 \ka^*(\chi u_\ka)\big\dr\,dV_{\coV g}\\
&=\sum_\ka\int_{U_{\coU\ka}}\ka^*\bigl(\dl\ka_*(\pi_\ka v),u_\ka\dr
 \,dV_{\coV g_m}\bigr)
 =\sum_\ka\int_{\BX_\ka}\dl\vp_\ka v,u_\ka\dr\,dV_{\coV g_m}
 =\mfdl\vp v,\mf{u}\mfdr.
\eal
$$
This proves
$$
\vp=\Phi\sn\cD\MV.
$$
By the arguments of steps~(1) and~(3), with $\ci M$ replaced by~$M$ and
$\ci\BX_\ka$ by~$\BX_\ka$, respectively, we find
$$
\vp\in\cL\bigl(\cD\MV,\mf{\cD}\BXE\bigr).
$$

\par
(8) Let
\hb{v\in\mf{\cD}\BXE} and
\hb{u\in\cD(\ci M,V')}. Then
$$
\bal
\dl\Psi\mf{v},u\dr_\cD
 =\mfdl\mf{v},\ci\vp u\mfdr
&=\sum_\ka\int_{\BX_\ka}\big\dl v_\ka,\ka_*(\pi_\ka u)\big\dr
 \sqrt{\ka_*g}\,dV_{\coV g_m}
 =\sum_\ka\int_{Q_\ka^m}\ka_*\bigl(\dl\pi_\ka\ka^*v_\ka,u\dr
 \,dV_{\coV g}\bigr)\\
&=\sum_\ka\int_M\dl\psi_\ka v_\ka,u\dr \,dV_{\coV g}
 =\int_M\dl\psi\mf{v},u\dr \,dV_{\coV g}
 =\dl\psi\mf{v},u\dr_\cD.
\eal
$$
Consequently,
$$
\psi=\Psi\sn\mf\cD\BXE.
$$
Modifying the arguments of steps (2) and~(4) in the obvious way gives
\hb{\Psi\in\cL\bigl(\mf{\cD}\BXE,\cD\MV\bigr)}.

\par
(9) By collecting what has been proved so far we see that the diagram
 $$
 \begin{picture}(215,63)(-59,-5)        
 \put(6,50){\makebox(0,0)[b]{\small{$\vp$}}}
 \put(6,5){\makebox(0,0)[b]{\small{$\Phi$}}}
 \put(94,50){\makebox(0,0)[b]{\small{$\psi$}}}
 \put(94,5){\makebox(0,0)[b]{\small{$\Psi$}}}
 \put(-38,45){\makebox(0,0)[c]{\small{$\cD\MV$}}}
 \put(-38,0){\makebox(0,0)[c]{\small{$\cD'\ciMV$}}}
 \put(50,45){\makebox(0,0)[c]{\small{$\mf{\cD}\BXE$}}}
 \put(50,0){\makebox(0,0)[c]{\small{$\mf{\cD}'\ciBXE$}}}
 \put(138,45){\makebox(0,0)[c]{\small{$\cD\MV$}}}
 \put(138,0){\makebox(0,0)[c]{\small{$\cD'\ciMV$}}}
 \put(55,22.5){\makebox(0,0)[l]{\small{$d$}}}
 \put(-14,45){\vector(1,0){40}}
 \put(-14,0){\vector(1,0){40}}
 \put(74,45){\vector(1,0){40}}
 \put(74,0){\vector(1,0){40}}
 \put(-38,33){\vector(0,-1){23}}
 \put(-36,33){\oval(4,4)[t]}
 \put(50,33){\vector(0,-1){23}}
 \put(52,33){\oval(4,4)[t]}
 \put(138,33){\vector(0,-1){23}}
 \put(140,33){\oval(4,4)[t]}
 \end{picture}
 $$
is commuting, where the embeddings symbolized by the vertical arrows
follow from \eqref{D.DDL} and \eqref{D.dDu}. Furthermore, $\Psi$~is a
retraction and $\Phi$~is a coretraction. Thus we read off this diagram that
$\Psi\bigl(\mf{\cD}\BXE\bigr)$ is dense in $\cD'(\ci M,V)$
(cf.~Lemma~4.1.6 in~\cite{Ama09a}).

\par
Let $U$ be a neighborhood of~$0$ in $\cD'(\ci M,V)$. Then there exists
\hb{\mf{u}\in\mf{\cD}\BXE} such that
\hb{\Psi\mf{u}\in U}. Hence
\hb{\Psi\mf{u}=\psi\mf{u}\in\cD\MV} shows that
\hb{U\cap\cD\MV\neq\es}. This implies that $\cD\MV$ is dense
in~$\cD'(\ci M,V)$. Since $\Phi$ and~$\Psi$ are continuous linear extensions
of $\vp$ and~$\psi$, respectively, they are uniquely determined by the
density of the `vertical' embeddings in the above diagram. Thus we can
denote $\Phi$ and~$\Psi$ also by $\vp$ and~$\psi$, respectively, without
fearing confusion. This establishes the theorem.
\end{proof}
\section{Sobolev Spaces}\label{sec-W}
Henceforth, we always assume
$$
\bt\quad
1<p<\iy
\qb \lda\in\BR.
$$
Suppose
\hb{k\in\BN}. The weighted \emph{Sobolev space}
$W_{\coW p}^{k,\lda}(V;\rho)$ of
$(\sa,\tau)$-tensor fields is the completion of~$\cD\MV$ in $L_{1,\loc}\MV$
with respect to the norm
\beq\label{W.norm}
u\mt\Bigl(\sum_{i=0}^k
\big\|\rho^{\lda+\tau-\sa+i}\ |\na^iu|_g\big\|_p^p\Bigr)^{1/p}.
\eeq
If
\hb{\rho'\in\gT(M)}, then
\hb{\rho'\sim\rho} and we obtain an equivalent norm by replacing~$\rho$ in
\eqref{W.norm} by~$\rho'$. Thus the topology of
$W_{\coW p}^{k,\lda}(V;\rho)$ depends on the singularity type~$\gT(M)$ only.
Henceforth, we simply write $W_{\coW p}^{k,\lda}(V)$ for
$W_{\coW p}^{k,\lda}(V;\rho)$ and denote the norm~\eqref{W.norm} by~%
\hb{\Vsdot_{k,p;\lda}}. Moreover,
\hb{L_p^\lda(V):=W_{\coW p}^{0,\lda}(V)} and
\hb{\Vsdot_{p;\lda}:=\Vsdot_{0,p;\lda}}. If
\hb{\gT(M)=[\![\mf{1}]\!]}, then all these spaces are independent of~$\lda$
and we obtain the `standard' Sobolev spaces~$W_{\coW p}^k(V)$. The reader
should be careful not to confuse $W_{\coW p}^{k,0}(V)$
with~$W_{\coW p}^k(V)$.

\par
We also define weighted spaces of bounded smooth
$(\sa,\tau)$-tensor fields by
$$
BC^{k,\lda}(V)
:=
\bigl(\bigl\{\,u\in C^k\MV\ ;\ \|u\|_{k,\iy;\lda}<\iy\,\bigr\},
\ \Vsdot_{k,\iy;\lda}\bigr),
$$
where
$$
\|u\|_{k,\iy;\lda}
:=\max_{0\leq i\leq k}\big\|\rho^{\lda+\tau-\sa+i}\ |\na^iu|_g\big\|_\iy.
$$
The topology of $BC^{k,\lda}(V)$ is independent of the
particular choice of
\hb{\rho\in\gT(M)}.

\par
The following basic retraction theorems show that these spaces can be
characterized by means of local coordinates, similarly as in the case of
function spaces on compact manifolds. Below we make free use, usually
without further mention, of the theory of function spaces on $\BR^m$
and~$\BH^m$. Everything for which we do not give specific references can
be found in H.~Triebel~\cite{Tri78a}, for example.

\par
Let $E_\al$ be a Banach space for each~$\al$ in a countable index set. Then
\hb{\mf{E}:=\prod_\al E_\al}. For
\hb{1\leq q\leq\iy} we denote by~$\ell_q(\mf{E})$ the linear subspace
of~$\mf{E}$ consisting of all
\hb{\mf{x}=(x_\al)} such that
$$
\|\mf{x}\|_{\ell_q(\mf{E})}:=
\left\{
\bal
{\textstyle \bigl(\sum_\al} &\|x_\al\|_{E_\al}^q\bigr)^{1/q},
                            &\quad 1\leq q  &<\iy,\\
{\textstyle \sup_\al}       &\|x_\al\|_{E_\al},
                            &\quad       q  &=\iy,
\eal
\right.
$$
is finite. Then $\ell_q(\mf{E})$ is a Banach space with norm~%
\hb{\Vsdot_{\ell_q(\mf{E})}}, and
\beq\label{W.ll}
\ell_p(\mf{E})\hr\ell_q(\mf{E})
\qa 1\leq p<q\leq\iy.
\eeq
We also set
\hb{c_c(\mf{E}):=\bigoplus_\al E_\al}. Then
\beq\label{W.cl}
c_c(\mf{E})\hr\ell_q(\mf{E})
\qb 1\leq q\leq\iy
\qa  c_c(\mf{E})\sdh\ell_q(\mf{E})
\qb q<\iy.
\eeq
Furthermore, $c_0(\mf{E})$~is the closure of~$c_c(\mf{E})$
in~$\ell_\iy(\mf{E})$.

\par
If each $E_\al$ is reflexive, then $\ell_p(\mf{E})$ is reflexive as well,
and
\hb{\ell_p(\mf{E})'=\ell_{p'}(\mf{E}')} with respect to the duality pairing
\hb{\mfpw:=\sum_\al\pw_\al}. Of course,
\hb{p':=p/(p-1)},
\ \hb{\mf{E}':=\prod_\al E_\al'}, and
\hb{\pw_\al}~is the \hbox{$E_\al$-duality} pairing.

\par
Let \eqref{L.sd} be chosen. For
\hb{1\leq q\leq\iy} we set
$$
\vp_{q,\ka}^\lda:=\rho_\ka^{\lda+m/q}\vp_\ka
\qb \psi_{q,\ka}^\lda:=\rho_\ka^{-\lda-m/q}\psi_\ka
\qa \ka\in\gK,
$$
and
$$
\vp_q^\lda u:=(\vp_{q,\ka}^\lda u)
\qb \psi_q^\lda\mf{v}:=\sum_\ka\psi_{q,\ka}^\lda v_\ka
$$
for
\hb{u\in\cD'(\ci M,V)} and
\hb{\mf{v}\in\mf{\cD}'\ciBXE}. If the dependence on $(\sa,\tau)$ is
important, then we write $\vp_{q,(\sa,\tau)}^\lda$,~etc. Note
\hb{(\vp_{p,\ka}^\lda,\psi_{p,\ka}^\lda)=(\vp_{p,\ka},\psi_{p,\ka})} if
\hb{\rho=\mf{1}}.

\par
Suppose $\gF$~is a symbol for one of the standard function spaces, say,
Sobolev, Slobodeckii, Besov spaces,~etc., on~$\BR^m$. Then we put
\hb{\mf{\gF}:=\prod_\ka\gF_\ka} and
\hb{\gF_\ka:=\gF\BXkE}. For example,
\hb{\mf{W}_{\coW p}^k=\prod_\ka W_{\coW p,\ka}^k
 =\prod_\ka W_{\coW p}^k\BXkE}.
\begin{theorem}\label{thm-W.ret}
Suppose
\hb{k\in\BN}. The diagram
 $$
 \begin{picture}(214,112)(-58,-51)        
 \put(6,50){\makebox(0,0)[b]{\small{$\vp_p^\lda$}}}
 \put(6,5){\makebox(0,0)[b]{\small{$\vp_p^\lda$}}}
 \put(94,50){\makebox(0,0)[b]{\small{$\psi_p^\lda$}}}
 \put(94,5){\makebox(0,0)[b]{\small{$\psi_p^\lda$}}}
 \put(-38,45){\makebox(0,0)[c]{\small{$\cD\MV$}}}
 \put(-38,0){\makebox(0,0)[c]{\small{$W_{\coW p}^{k,\lda}(V)$}}}
 \put(50,45){\makebox(0,0)[c]{\small{$\mf{\cD}\BXE$}}}
 \put(50,0){\makebox(0,0)[c]{\small{$\ell_p(\mf{W}_{\coW p}^k)$}}}
 \put(138,45){\makebox(0,0)[c]{\small{$\cD\MV$}}}
 \put(138,0){\makebox(0,0)[c]{\small{$W_{\coW p}^{k,\lda}(V)$}}}
 \put(-31,22.5){\makebox(0,0)[l]{\small{$d$}}}
 \put(55,22.5){\makebox(0,0)[l]{\small{$d$}}}
 \put(143,22.5){\makebox(0,0)[l]{\small{$d$}}}
 \put(-14,45){\vector(1,0){40}}
 \put(-14,0){\vector(1,0){40}}
 \put(74,45){\vector(1,0){40}}
 \put(74,0){\vector(1,0){40}}
 \put(-38,33){\vector(0,-1){23}}
 \put(-36,33){\oval(4,4)[t]}
 \put(50,33){\vector(0,-1){23}}
 \put(52,33){\oval(4,4)[t]}
 \put(138,33){\vector(0,-1){23}}
 \put(140,33){\oval(4,4)[t]}
 \put(6,-40){\makebox(0,0)[b]{\small{$\vp_p^\lda$}}}
 \put(94,-40){\makebox(0,0)[b]{\small{$\psi_p^\lda$}}}
 \put(-38,-45){\makebox(0,0)[c]{\small{$\cD'\ciMV$}}}
 \put(50,-45){\makebox(0,0)[c]{\small{$\mf{\cD}'\ciBXE$}}}
 \put(138,-45){\makebox(0,0)[c]{\small{$\cD'\ciMV$}}}
 \put(-31,-22.5){\makebox(0,0)[l]{\small{$d$}}}
 \put(55,-22.5){\makebox(0,0)[l]{\small{$d$}}}
 \put(143,-22.5){\makebox(0,0)[l]{\small{$d$}}}
 \put(-14,-45){\vector(1,0){40}}
 \put(74,-45){\vector(1,0){40}}
 \put(-38,-12){\vector(0,-1){23}}
 \put(-36,-12){\oval(4,4)[t]}
 \put(50,-12){\vector(0,-1){23}}
 \put(52,-12){\oval(4,4)[t]}
 \put(138,-12){\vector(0,-1){23}}
 \put(140,-12){\oval(4,4)[t]}
 \end{picture}
 $$
is commuting and
\hb{\psi_p^\lda\circ\vp_p^\lda=\id}.
\end{theorem}
\begin{proof}
(1) It is an obvious consequence of Theorem~\ref{thm-L.ret} that
$\psi_p^\lda$~is a retraction from $\mf{\cD}\BXE$ onto~$\cD\MV$, and
from $\cD'\ciBXE$ onto~$\cD'(\ci M,V)$, and that $\vp_p^\lda$~is a
coretraction in each case.

\par
(2) Estimate~\eqref{U.LS}(iii), Leibniz' rule, and
\hb{\ka_*(\pi_\ka u)=(\ka_*\pi_\ka)\ka_*u} imply, due to
\hb{\chi_\ka\sn\supp(\pi_\ka)=\mf{1}},
\beq\label{W.Q}
\|\ka_*(\pi_\ka u)\|_{W_{\coW p,\ka}^k}
\leq c\,\|\ka_*(\chi_\ka u)\|_{W_{\coW p}^k(Q_\ka^m,E)}
\qa \ka\in\gK.
\eeq
From Lemma~\ref{lem-U.g}(iv) we deduce
\beq\label{W.Wkp}
\bal
\|\ka_*(\chi_\ka u)\|_{W_{\coW p}^k(Q_\ka^m,E)}^p
=\int_{Q_\ka^m}\chi\sum_{|\al|\leq k}|\pa(\ka_*u)|_{g_m}^p\,dV_{\coV g_m}
\leq\sum_{i=0}^k\int_{Q_\ka^m}\chi\,
 |\na_{\cona\ka_*g}^i(\ka_*u)|_{g_m}^p\,dV_{\coV g_m}.
\eal
\eeq
By part~(v) of Lemma~\ref{lem-U.g} we get, due to
\hb{\na^iu\in\cD(M,T_{\tau+i}^\sa M)} for
\hb{u\in\cD\MV},
$$
|\na_{\cona\ka_*g}^i(\ka_*u)|_{g_m}
\sim\ka_*(\rho_\ka^{\tau-\sa+i}\,|\na^iu|_g)
\qa \ka\in\gK.
$$
Thus, observing Lemma~\ref{lem-U.g}(iii) and \eqref{S.sd}(vi),
$$
\bal
\int_{Q_\ka^m}\chi\,|\na_{\cona\ka_*g}^i(\ka_*u)|_{g_m}^p\,dV_{\coV g_m}
&\sim\int_{\ka(U_{\coU\ka})}\ka_*
 \Bigl(\bigl(\chi_\ka\rho_\ka^{\tau-\sa+i-m/p}
 \,|\na^iu|_g\bigr)^p\,dV_{\coV g}\Bigr)\\
&\sim\rho_\ka^{-m}\int_M\chi_\ka\bigl(\rho^{\tau-\sa+i}
 \,|\na^iu|_g\bigr)^p\,dV_{\coV g}
\eal
$$
for
\hb{\ka\in\gK}. Thus we get from \eqref{W.Q} and \eqref{W.Wkp}
$$
\|\vp_{p,\ka}^\lda u\|_{W_{\coW p,\ka}^k}^p
\leq c\sum_{i=0}^k\int_M\chi_\ka\bigl(\rho^{\lda+\tau-\sa+i}
\,|\na^iu|_g\bigr)^p\,dV_{\coV g}.
$$
The finite multiplicity of~$\gK$ implies
\hb{0\leq\sum_\ka\chi_\ka\leq c\mf{1}_M}. Consequently,
\beq\label{W.phi}
\|\vp_p^\lda u\|_{\ell_p(\mf{W}_{\coW p}^k)}
\leq c\,\|u\|_{k,p;\lda}
\qa u\in\cD\MV.
\eeq
Since $\cD\MV$ is dense in $W_{\coW p}^{k,\lda}(V)$ it follows
\hb{\vp_p^\lda
\in\cL\bigl(W_{\coW p}^{k,\lda}(V),\ell_p(\mf{W}_{\coW p}^k)\bigr)}.

\par
(3) Similarly as in the preceding step we find
$$
\|\psi_{p,\ka}^\lda v_\ka\|_{k,p;\lda}
\leq c\,\|v_\ka\|_{W_{\coW p,\ka}^k}
\qa \ka\in\gK.
$$
Since
\hb{\chi_\ka\sn\im(\psi_{p,\ka}^\lda)=\mf{1}} it follows from the finite
multiplicity of~$\gK$ and H\"older's inequality that
$$
|\na^i(\psi_p^\lda\mf{v})|_g^p
=\Big|\sum_\ka\chi_\ka\na^i(\psi_{p,\ka}^\lda v_\ka)\Big|_g^p
\leq c\sum_\ka|\na^i(\psi_{p,\ka}^\lda v_\ka)|_g^p.
$$
Consequently,
$$
\|\psi_p^\lda\mf{v}\|_{k,p;\lda}
\leq c\,\|\mf{v}\|_{\ell_p(\mf{W}_{\coW p}^k)}
\qa v\in\ell_p(\mf{W}_{\coW p}^k).
$$
Since
\hb{\psi_p^\lda\vp_p^\lda u=u} for
\hb{u\in W_{\coW p}^{k,\lda}(V)} we have shown that $\psi_p^\lda$~is a
retraction from $\ell_p(\mf{W}_{\coW p}^k)$
onto~$W_{\coW p}^{k,\lda}(V)$.

\par
(4) For each
\hb{\ka\in\gK} it holds
\hb{\cS\BXkE\sdh W_{\coW p}^k\BXkE}. This is well-known if
\hb{\BX_\ka=\BR^m} (e.g.,~\cite{Tri78a}) and follows from (4.4.3)
in~\cite{Ama09a} if
\hb{\BX_\ka=\BH^m}. Furthermore,
\hb{\cD\BXkE\sdh\cS\BXkE}. In fact, this is standard knowledge if
\hb{\BX_\ka=\BR^m}; otherwise it follows from Section~4.2 in~\cite{Ama09a}.
Hence
\beq\label{W.DW}
\cD\BXkE\sdh W_{\coW p}^k\BXkE
\qa \ka\in\gK.
\eeq
Thus, since $c_c(\mf{W}_{\coW p}^k)$ is dense in
$\ell_p(\mf{W}_{\coW p}^k)$, we obtain
\beq\label{W.Dl}
\mf{\cD}\BXE\sdh\ell_p(\mf{W}_{\coW p}^k).
\eeq

\par
(5) Analogously we find
\hb{W_{\coW p}^k\BXkE\hr\cD'\ciBXkE} for
\hb{\ka\in\gK}. From this and the definition of the product topology it
follows
$$
\ell_p(\mf{W}_{\coW p}^k)\hr\prod_\ka W_{\coW p}^k\BXkE\hr\mf{\cD}\ciBXE.
$$
Since
\hb{\mf{\cD}\BXE\sdh\mf{\cD}'\ciBXE} we thus obtain from \eqref{W.Dl} that
\hb{\ell_p(\mf{W}_{\coW p}^k)\sdh\mf{\cD}\ciBXE}. The theorem is proved.
\end{proof}
\begin{corollary}\label{cor-W.ret}
Suppose
\hb{M=\BR^m} or
\hb{M=\BH^m}, and
\hb{V=M\times\BK}. Then the above definition yields the usual Sobolev
spaces.
\end{corollary}
\begin{proof}
This follows from \eqref{W.DW} and Example~\ref{exa-S.ex}(c).
\end{proof}
\begin{theorem}\label{thm-W.retBC}
Suppose
\hb{k\in\BN}. The diagram
 $$
 \begin{picture}(215,112)(-58,-51)        
 \put(6,50){\makebox(0,0)[b]{\small{$\vp_\iy^\lda$}}}
 \put(6,5){\makebox(0,0)[b]{\small{$\vp_\iy^\lda$}}}
 \put(94,50){\makebox(0,0)[b]{\small{$\psi_\iy^\lda$}}}
 \put(94,5){\makebox(0,0)[b]{\small{$\psi_\iy^\lda$}}}
 \put(-38,45){\makebox(0,0)[c]{\small{$\cD\MV$}}}
 \put(-38,0){\makebox(0,0)[c]{\small{$BC^{k,\lda}(V)$}}}
 \put(50,45){\makebox(0,0)[c]{\small{$\mf{\cD}\BXE$}}}
 \put(50,0){\makebox(0,0)[c]{\small{$\ell_\iy(\mf{BC}^k)$}}}
 \put(138,45){\makebox(0,0)[c]{\small{$\cD\MV$}}}
 \put(138,0){\makebox(0,0)[c]{\small{$BC^{k,\lda}(V)$}}}
 \put(-14,45){\vector(1,0){40}}
 \put(-14,0){\vector(1,0){40}}
 \put(74,45){\vector(1,0){40}}
 \put(74,0){\vector(1,0){40}}
 \put(-38,33){\vector(0,-1){23}}
 \put(-36,33){\oval(4,4)[t]}
 \put(50,33){\vector(0,-1){23}}
 \put(52,33){\oval(4,4)[t]}
 \put(138,33){\vector(0,-1){23}}
 \put(140,33){\oval(4,4)[t]}
 \put(6,-40){\makebox(0,0)[b]{\small{$\vp_\iy^\lda$}}}
 \put(94,-40){\makebox(0,0)[b]{\small{$\psi_\iy^\lda$}}}
 \put(-38,-45){\makebox(0,0)[c]{\small{$\cD'\ciMV$}}}
 \put(50,-45){\makebox(0,0)[c]{\small{$\mf{\cD}'\ciBXE$}}}
 \put(138,-45){\makebox(0,0)[c]{\small{$\cD'\ciMV$}}}
 \put(-14,-45){\vector(1,0){40}}
 \put(74,-45){\vector(1,0){40}}
 \put(-38,-12){\vector(0,-1){23}}
 \put(-36,-12){\oval(4,4)[t]}
 \put(50,-12){\vector(0,-1){23}}
 \put(52,-12){\oval(4,4)[t]}
 \put(138,-12){\vector(0,-1){23}}
 \put(140,-12){\oval(4,4)[t]}
 \end{picture}
 $$
is commuting and
\hb{\psi_\iy^\lda\circ\vp_\iy^\lda=\id}.
\end{theorem}
\begin{proof}
This is verified by modifying the preceding proof in the obvious way.
\end{proof}
\begin{remark}\label{rem-W.ret}
Define $\wt{\vp}_q^\lda$ and~$\wt{\psi}_q^\lda$ by replacing~$\pi_\ka$ in
the definition of~$(\vp_q^\lda,\psi_q^\lda)$ by~$\chi_\ka$. Then
$\wt{\vp}_q^\lda$ and~$\wt{\psi}_q^\lda$ possess the same mapping
properties as $\vp_q^\lda$ and~$\psi_q^\lda$. (Of course,
$\wt{\psi}_q^\lda$~is not a retraction.)
\end{remark}
\begin{proof}
This is clear from the preceding proofs.
\end{proof}
\section{Sobolev-Slobodeckii and Bessel Potential Spaces}\label{sec-B}
We denote by
\hb{\pe_\ta} the complex and by
\hb{\pr_{\ta,q}},
\ \hb{1\leq q\leq\iy}, the real interpolation functor for
\hb{0<\ta<1}. Definitions and proofs of the results from interpolation
theory which we use below without further mention can be found
in~\cite{Tri78a}. (Also see Section~I.2 of~\cite{Ama95a} for a summary.)
We write
\hb{X\doteq Y} if $X$ and~$Y$ are Banach spaces which are equal, except for
equivalent norms.

\par
For
\hb{s\geq0} we define weighted \emph{Bessel potential spaces} of
$(\sa,\tau)$-tensor fields by
$$
H_p^{s,\lda}=H_p^{s,\lda}(V):=
\left\{
\bal
{}
[   &W_{\coW p}^{k,\lda},W_{\coW p}^{k+1,\lda}]_{s-k},
        &\quad k<{} &s<k+1,\ \ k\in\BN,\\
[       &W_{\coW p}^{k-1,\lda},W_{\coW p}^{k+1,\lda}]_{1/2},
        &           &s=k\in\BN^\times,\\
        &L_p^\lda,
        &           &s=0,
\eal
\right.
$$
where
\hb{W_{\coW p}^{k,\lda}=W_{\coW p}^{k,\lda}(V)}. Similarly, weighted
\emph{Besov spaces} are defined for
\hb{s>0} by
$$
B_p^{s,\lda}=B_p^{s,\lda}(V):=
\left\{
\bal
{}
&(W_{\coW p}^{k,\lda},W_{\coW p}^{k+1,\lda})_{s-k,p},
        &\quad k<{} &s<k+1,\ \ k\in\BN,\\
&(W_{\coW p}^{k-1,\lda},W_{\coW p}^{k+1,\lda})_{1/2,p},
        &           &s=k\in\BN^\times.
\eal
\right.
$$

\par
In the remainder of this paper
$$
\bt\quad
\gF\in\{H,B\}.
$$
This allows us to develop the theory of Bessel potential and Besov spaces to
a large extent in one and the same setting.
\begin{theorem}\label{thm-B.ret}
Let\/ \eqref{L.sd} be chosen and
\hb{s>0}. Then $\psi_p^\lda$~is a retraction from
$\ell_p(\mf{\gF}_p^s)$ onto~$\gF_p^{s,\lda}$, and $\vp_p^\lda$~is a
coretraction.
\end{theorem}
\begin{proof}
Suppose
\hb{k,\ell\in\BN} satisfy
\hb{k<\ell}. Theorem~\ref{thm-W.ret} implies that the diagram
 $$
 \begin{picture}(196,66)(-48,-5)        
 \put(6,50){\makebox(0,0)[b]{\small{$\vp_p^\lda$}}}
 \put(6,5){\makebox(0,0)[b]{\small{$\vp_p^\lda$}}}
 \put(94,50){\makebox(0,0)[b]{\small{$\psi_p^\lda$}}}
 \put(94,5){\makebox(0,0)[b]{\small{$\psi_p^\lda$}}}
 \put(-38,45){\makebox(0,0)[c]{\small{$W_{\coW p}^{\ell,\lda}$}}}
 \put(-38,0){\makebox(0,0)[c]{\small{$W_{\coW p}^{k,\lda}$}}}
 \put(50,45){\makebox(0,0)[c]{\small{$\ell(\mf{W}_{\coW p}^\ell)$}}}
 \put(50,0){\makebox(0,0)[c]{\small{$\ell(\mf{W}_{\coW p}^k)$}}}
 \put(138,45){\makebox(0,0)[c]{\small{$W_{\coW p}^{\ell,\lda}$}}}
 \put(138,0){\makebox(0,0)[c]{\small{$W_{\coW p}^{k,\lda}$}}}
 \put(-33,22.5){\makebox(0,0)[l]{\small{$d$}}}
 \put(55,22.5){\makebox(0,0)[l]{\small{$d$}}}
 \put(143,22.5){\makebox(0,0)[l]{\small{$d$}}}
 \put(-14,45){\vector(1,0){40}}
 \put(-14,0){\vector(1,0){40}}
 \put(74,45){\vector(1,0){40}}
 \put(74,0){\vector(1,0){40}}
 \put(-38,33){\vector(0,-1){23}}
 \put(-36,33){\oval(4,4)[t]}
 \put(50,33){\vector(0,-1){23}}
 \put(52,33){\oval(4,4)[t]}
 \put(138,33){\vector(0,-1){23}}
 \put(140,33){\oval(4,4)[t]}
 \end{picture}
 $$
is commuting, and
\hb{\psi_p^\lda\circ\vp_p^\lda=\id}. From this it follows that
$\psi_p^\lda$~is a retraction from
\hb{\bigl[\ell_p(\mf{W}_{\coW p}^k),
   \ell_p(\mf{W}_{\coW p}^\ell)\bigr]_\ta } onto
\hb{[W_{\coW p}^{k,\lda},W_{\coW p}^{\ell,\lda}]_\ta}
and from
\hb{\bigl(\ell_p(\mf{W}_{\coW p}^k),
   \ell_p(\mf{W}_{\coW p}^\ell)\bigr)_{\ta,p}}  onto
\hb{(W_{\coW p}^{k,\lda},W_{\coW p}^{\ell,\lda})_{\ta,p}} for
\hb{0<\ta<1}.

\par
By Theorem~1.18.1 in~\cite{Tri78a} we obtain, using obvious notation,
$$
\bigl[\ell_p(\mf{W}_{\coW p}^k),\ell_p(\mf{W}_{\coW p}^\ell)\bigr]_\ta
 =\ell_p\bigl([\mf{W}_{\coW p}^k,\mf{W}_{\coW p}^\ell]_\ta\bigr)
\qb
 \bigl(\ell_p(\mf{W}_{\coW p}^k),\ell_p(\mf{W}_{\coW p}^\ell)\bigr)_{\ta,p}
 \doteq
 \ell_p\bigl((\mf{W}_{\coW p}^k,\mf{W}_{\coW p}^\ell)_{\ta,p}\bigr).
$$
Since
\hb{[W_{\coW p,\ka}^k,W_{\coW p,\ka}^\ell]_\ta
   \doteq H_{p,\ka}^{(1-\ta)k+\ta\ell}}, the assertion follows.
\end{proof}
For
\hb{\xi_0,\xi_1\in\BR} and
\hb{0<\ta<1} we set
\hb{\xi_\ta:=(1-\ta)\xi_0+\ta\xi_1}.
\begin{corollary}\label{cor-B.ret}
\begin{itemize}
\item[(i)]
${}$
\hb{H_p^{k,\lda}(V)\doteq W_{\coW p}^{k,\lda}(V)},
\ \hb{k\in\BN}.
\item[(ii)]
${}$
Suppose
\hb{0\leq s_0<s_1<\iy} and
\hb{\ta\in(0,1)}. Then
$$
[H_p^{s_0,\lda},H_p^{s_1,\lda}]_\ta\doteq H_p^{s_\ta,\lda}
\qb (B_p^{s_0,\lda},B_p^{s_1,\lda})_{\ta,p}=B_p^{s_\ta,\lda},
$$
provided
\hb{s_0>0} in the latter case.
\end{itemize}
\end{corollary}
\begin{proof}
(i)~follows from
\hb{H_{p,\ka}^k\doteq W_{\coW p,\ka}^k} for
\hb{k\in\BN}.

\par
(ii)~is a consequence of the reiteration theorems for the complex and real
interpolation functors.
\end{proof}
The following theorem shows that weighted Bessel potential and Besov spaces
can be characterized locally by intrinsic norms, since this is the case for
the spaces~$\gF_{p,\ka}^s$. In particular,
\hb{B_{p,\ka}^s\doteq W_{\coW p,\ka}^s} for
\hb{s\notin\BN}. For this reason we call
$$
W_{\coW p}^{s,\lda}=W_{\coW p}^{s,\lda}(V):=B_p^{s,\lda}
\qa s\in\BR^+\ssm\BN,
$$
weighted \emph{Slobodeckii space}.
\begin{theorem}\label{thm-B.loc}
Let\/ \eqref{L.sd} be selected. Suppose
\hb{s\geq0} with
\hb{s>0} if
\hb{\gF=B}. Then
\hb{u\in L_{1,\loc}\MV} belongs to~$\gF_p^{s,\lda}(V)$ iff
\hb{\ka_*(\pi_\ka u)\in\gF_{p,\ka}^s} and
$$
\vd u\vd_{\gF_p^{s,\lda}}
:=
\Bigl(\sum_\ka\bigl(\rho_\ka^{\lda+m/p}
\,\|\ka_*(\pi_\ka u)\|_{\gF_{p,\ka}^s}\bigr)^p\Bigr)^{1/p}<\iy.
$$
Moreover, $\Vvsdot_{\gF_p^{s,\lda}}$~is a norm for~$\gF_p^{s,\lda}$.
\end{theorem}
\begin{proof}
Let $X$ and~$Y$ be Banach spaces,
\hb{r\in\cL\XY} a~retraction, and
\hb{e\in\cL\YX} a~coretraction. Then
$$
\|ey\|\leq\|e\|\,\|y\|=\|e\|\,\|rey\|
\leq\|e\|\,\|r\|\,\|ey\|
\qa y\in Y,
$$
implies
\hb{\Vsdot_Y\sim\|e\cdot\|_X}. Thus the assertion follows from
Theorem~\ref{thm-B.ret}, setting
\hb{e:=\vp_p^\lda}.
\end{proof}
Of course, $\Vvsdot_{\gF_p^{s,\lda}}$~depends on the particular singularity
datum $(\rho,\gK)$ and on  the chosen localization system subordinate
to~$\gK$. Since $\gF_p^{s,\lda}$~has been invariantly defined it follows
that another choice of these data results in an equivalent norm.
\begin{theorem}\label{thm-B.ref}
$\gF_p^{s,\lda}(V)$~is a reflexive Banach space.
\end{theorem}
\begin{proof}
Since $\gF_{p,\ka}^s$ is reflexive (cf.~Theorem~4.4.4
of~\cite{Ama09a} if
\hb{\BX_\ka=\BH^m}),
\ $\ell_p(\mf{\gF}_p^s)$~is reflexive.
Theorem~\ref{thm-B.ret} implies that $\gF_p^{s,\lda}(V)$ is isomorphic
to a closed linear subspace of~$\ell_p(\mf{\gF}_p^s)$ (e.g., Lemma~I.2.3.1
in~\cite{Ama95a}). Hence  $\gF_p^{s,\lda}(V)$ is reflexive as well.
\end{proof}
The following theorem shows that the weighted Bessel potential and Besov
spaces are natural with respect to~$\na$.
\begin{theorem}\label{thm-B.Na}
Suppose
\hb{s\geq0} with
\hb{s>0} if\/
\hb{\gF=B}, and
\hb{k\in\BN^\times}. Then
$$
\na^k\in\cL\bigl(\gF_p^{s+k,\lda}(V_{\coV\tau}^\sa),
\gF_p^{s,\lda}(V_{\coV\tau+k}^\sa)\bigr).
$$
\end{theorem}
\begin{proof}
Since $\na^ku$~is a
$(\sa,\tau+k)$-tensor field if $u$~is a
$(\sa,\tau)$-tensor field, it is obvious that
$$
\na^k\in\cL\bigl(W_{\coW p}^{s+k,\lda}(V_{\coV\tau}^\sa),
W_{\coW p}^{s,\lda}(V_{\coV\tau+k}^\sa)\bigr)
$$
for
\hb{s\in\BN}. Now we obtain the assertion by interpolation, due to
Corollary~\ref{cor-B.ret}.
\end{proof}
\begin{remarks}\label{rem-B.Rm}
\hh{(a)}\quad
We consider the simplest case:
\hb{M=(\BR^m,g_m)} and
\hb{V=M\times\BK} with
\hb{\gT(M)=[\![\mf{1}]\!]}. By the arguments of the proof of
Lemma~\ref{lem-U.LS} we construct
\hb{\pi\in\cD\bigl(Q^m,[0,1]\bigr)} such that
\hb{\bigl\{\,\pi^2(\cdot+z)\ ;\ z\in\BZ^m\,\bigr\}} is a partition of unity
subordinate to the open covering
\hb{\{\,z+Q^m\ ;\ z\in\BZ^m\}} of~$\BR^m$. Consequently, fixing
\hb{\chi\in\cD\bigl(Q^m,[0,1]\bigr)} with
\hb{\chi\sn\supp(\pi)=\mf{1}}, it follows that
\hb{\bigl\{\,\pi(\cdot+z),\chi(\cdot+z)\ ;\ z\in\BZ\,\bigr\}} is~a
localization system subordinate to the `translation atlas' constructed in
the proof of Example~\ref{exa-S.ex}(c). Hence Theorem~\ref{thm-B.loc}
guarantees that
\beq\label{B.Rm}
u\mt\Bigl(\sum_{z\in\BZ^m}\|\pi u(\cdot+z)\|_{\gF_p^s\Rm}^p\Bigr)^{1/p}
=\Bigl(\sum_{z\in\BZ^m}\|\pi(\cdot-z)u\|_{\gF_p^s\Rm}^p\Bigr)^{1/p}
\eeq
is an equivalent norm for~$\gF_p^s\Rm$, where
\hb{s>0} if
\hb{\gF=B}. This assertion is equivalent to the `localization principle'
of Theorem~2.4.7 of~\cite{Tri92a} for the Bessel potential spaces $H_p^s\Rm$
with
\hb{s\geq0} and the Besov spaces  $B_p^s\Rm$ with
\hb{s>0}.

\par
\hh{(b)}\quad
Of course, it is natural to define $B_{p,q}^{s,\lda}(V)$ with
\hb{1\leq q\leq\iy} by replacing $\pr_{\ta,p}$ in the definition
of $B_p^{s,\lda}(V)$ by~$\pr_{\ta,q}$. However, in this case the proof of
Theorem~\ref{thm-B.ret} does not apply. In fact, it follows from
Theorem~2.4.7 in~\cite{Tri92a} that there is no characterization of
$B_{p,q}^s\Rm$ analogous to \eqref{B.Rm} if
\hb{p\neq q}. For this reason the spaces $B_{p,q}^{s,\lda}(V)$ with
\hb{q\neq p} are less useful and we refrain from considering them
here.\hfill$\Box$
\end{remarks}
In the case where
\hb{M=\BR^m}, a~retraction-coretraction pair~$(\psi_p,\vp_p)$ based on a
localization system equivalent to the one of Remark~\ref{rem-B.Rm}(a) has
been
introduced in H.~Amann, M.~Hieber, and G.~Simonett~\cite{AHS94a}. In that
paper, besides establishing the analogue of~\eqref{B.Rm}, it is shown
that $(\psi_p,\vp_p)$~is useful to localize partial differential equations
for deriving maximal regularity results. This localization technique has
since been applied by several authors for the study of parabolic equations
on~$\BR^m$ (eg.,~\cite{KuW04a} and the references therein). An
abstract formulation has been given by
S.~Angenent~\cite{Ang99a}. As mentioned in the introduction, the
retraction-coretraction pair~$(\psi_p^\lda,\vp_p^\lda)$ is part of the
fundament on which we build (elsewhere) a~theory of parabolic equations
on singular manifolds.
\section{H\"older Spaces}\label{sec-H}
Let \eqref{L.sd} be chosen. For
\hb{k<s<k+1} with
\hb{k\in\BN} we denote by
\hb{BC_\ka^s:=BC^s\BXkE} the Banach space of all
\hb{u\in BC^k\BXkE} such that $\pa u$~is uniformly
\hb{(s-k)}-H\"older continuous for
\hb{|\al|=k}, endowed with one of its standard norms.

\par
From
\hb{BC_\ka^{k+1}\hr BC_\ka^s\hr BC_\ka^k} and Theorem~\ref{thm-W.retBC}
it follows
 $$
 \begin{picture}(195,55)(-54,-5)        
 \put(-29,45){\makebox(0,0)[c]{\small{$\ell_\iy(\mf{BC}^{k+1})$}}}
 \put(-29,0){\makebox(0,0)[c]{\small{$BC^{{k+1},\lda}$}}}
 \put(50,45){\makebox(0,0)[c]{\small{$\ell_\iy(\mf{BC}^s)$}}}
 \put(123,45){\makebox(0,0)[c]{\small{$\ell_\iy(\mf{BC}^k)$}}}
 \put(123,0){\makebox(0,0)[c]{\small{$BC^{k,\lda}$}}}
 \put(-34,22.5){\makebox(0,0)[r]{\small{$\psi_\iy^\lda$}}}
 \put(128,22.5){\makebox(0,0)[l]{\small{$\psi_\iy^\lda$}}}
 \put(5,45){\vector(1,0){18}}
 \put(5,47){\oval(4,4)[l]}
 \put(79,45){\vector(1,0){18}}
 \put(79,47){\oval(4,4)[l]}
 \put(0,0){\vector(1,0){102}}
 \put(0,2){\oval(4,4)[l]}
 \put(-29,35){\vector(0,-1){25}}
 \put(123,35){\vector(0,-1){25}}
 \end{picture}
 $$
Now we define
\hb{BC^{s,\lda}:=BC^{s,\lda}(V)},
the weighted \emph{space of} $(s\text{-})$\emph{H\"older continuous
$(\sa,\tau)$-tensor fields}, to be the image space of
\hb{\psi_\iy^\lda\sn\ell_\iy(\mf{BC}^s)}, so that the diagram
$$
 \begin{picture}(235,55)(-74,-5)        
 \put(-49,45){\makebox(0,0)[c]{\small{$\ell_\iy(\mf{BC}^{k+1})$}}}
 \put(-49,0){\makebox(0,0)[c]{\small{$BC^{k+1,\lda}$}}}
 \put(50,45){\makebox(0,0)[c]{\small{$\ell_\iy(\mf{BC}^s)$}}}
 \put(50,0){\makebox(0,0)[c]{\small{$BC^{s,\lda}$}}}
 \put(143,45){\makebox(0,0)[c]{\small{$\ell_\iy(\mf{BC}^k)$}}}
 \put(143,0){\makebox(0,0)[c]{\small{$BC^{k,\lda}$}}}
 \put(-43,22.5){\makebox(0,0)[l]{\small{$\psi_\iy^\lda$}}}
 \put(55,22.5){\makebox(0,0)[l]{\small{$\psi_\iy^\lda$}}}
 \put(148,22.5){\makebox(0,0)[l]{\small{$\psi_\iy^\lda$}}}
 \put(-15,45){\vector(1,0){38}}
 \put(-15,47){\oval(4,4)[l]}
 \put(-15,0){\vector(1,0){38}}
 \put(-15,2){\oval(4,4)[l]}
 \put(79,45){\vector(1,0){38}}
 \put(79,47){\oval(4,4)[l]}
 \put(79,0){\vector(1,0){38}}
 \put(79,2){\oval(4,4)[l]}
 \put(-49,35){\vector(0,-1){25}}
 \put(50,35){\vector(0,-1){25}}
 \put(143,35){\vector(0,-1){25}}
 \end{picture}
 $$
is commuting. Of course, this definition depends on the choice of the
singularity datum $(\rho,\gK)$ and the localization system subordinate
to~$\gK$. The following theorem shows, however, that the topology
of~$BC^{s,\lda}$ is determined by the singularity type $\gT(M)$ only.
\begin{theorem}\label{thm-H.eq}
Suppose
\hb{k<s<k+1} with
\hb{k\in\BN}.
\begin{itemize}
\item[(i)]
$\psi_\iy^\lda$~is a retraction onto $BC^{s,\lda}$ and $\vp_\iy^\lda$~is a
coretraction.
\item[(ii)]
$BC^{s,\lda}$~is a Banach space and
$$
u\mt\vd u\vd_{s,\iy;\lda}
:=\sup_\ka\rho_\ka^\lda\,\|\ka_*(\pi_\ka u)\|_{BC_\ka^s}
$$
is a norm for it. Other choices of singularity data and  localization
systems lead to equivalent norms.
\end{itemize}
\end{theorem}
\begin{proof}
(1) Assertion~(i) and the claim that $BC^{s,\lda}$ is a Banach space and
\hb{\Vvsdot_{s,\iy;\lda}} a norm are clear.

\par
(2) Let $(\wt{\rho},\wt{\gK})$ be a singularity datum and
\hb{\bigl\{\,(\wt{\pi}_{\wt{\ka}},\wt{\chi}_{\wt{\ka}})
    \ ;\ \wt{\ka}\in\wt{\gK}\,\bigr\}} a~localization system subordinate
to~$\wt{\gK}$. Suppose
\hb{j\in\BN} and
\hb{w\in BC_{\wt{\ka}}^j}. Then
$$
\ka_*\wt{\ka}^*(\wt{\chi}w)
=(\wt{\chi}w)\circ(\wt{\ka}\circ\ka^{-1})
=(\wt{\ka}\circ\ka^{-1})^*(\wt{\chi}w).
$$
Thus it follows from Leibniz' rule, \eqref{U.LS}, and \eqref{S.eq}(iii) that
\beq\label{H.kkj}
\|\ka_*\wt{\ka}^*(\wt{\chi}w)\|_{BC_\ka^j}
\leq c(j)\,\|w\|_{BC_{\wt{\ka}}^j},
\eeq 
that is,
$$
\bigl(w\mt\ka_*\wt{\ka}^*(\wt{\chi}w)\bigr)
\in\cL(BC_{\wt{\ka}}^j,BC_\ka^j)
\qa j\in\BN.
$$
Since
\hb{BC_\ka^s\doteq(BC_\ka^k,BC_\ka^{k+1})_{s-k,\iy}},
we thus obtain
\beq\label{H.s}
\bigl(w\mt\ka_*\wt{\ka}^*(\wt{\chi}w)\bigr)
\in\cL(BC_{\wt{\ka}}^s,BC_\ka^s).
\eeq

\par
(3) Using
\hb{\sum_{\wt{\ka}}\wt{\pi}_{\wt{\ka}}^2=\mf{1}} we find
\beq\label{H.kk}
\bal
\ka_*(\rho_\ka^\lda\pi_\ka u)
&=\ka_*\Bigl(\rho_\ka^\lda\pi_\ka
 \sum_{\wt{\ka}}\wt{\pi}_{\wt{\ka}}^2u\Bigr)\\
&=(\ka_*\pi_\ka)\sum_{\wt{\ka}\in\gN(\ka)}
 (\rho_\ka/\wt{\rho}_{\wt{\ka}})^\lda
 \bigl(\ka_*\wt{\ka}^*(\wt{\ka}_*\wt{\pi}_{\wt{\ka}})\bigr)
 \,\wt{\rho}_{\wt{\ka}}^\lda\bigl(\ka_*\wt{\ka}^*
 \bigl(\wt{\ka}_*(\wt{\pi}_{\wt{\ka}}\wt{\chi}u)\bigr)\bigr).
\eal
\eeq
From \eqref{S.sd}(vi) and \eqref{S.eq}(ii) it follows
\hb{\rho_\ka\sim\wt{\rho}_{\wt{\ka}}} for
\hb{\ka\in\gK} and
\hb{\wt{\ka}\in\gN(\ka)}. Thus we infer from \eqref{U.LS}, \eqref{H.s}, and
\eqref{H.kk} that
$$
\|\rho_\ka^\lda\ka_*(\pi_\ka u)\|_{BC_\ka^s}
\leq c\sum_{\wt{\ka}\in\gN(\ka)}
\|\wt{\rho}_{\wt{\ka}}^\lda\wt{\ka}_*(\wt{\pi}_{\wt{\ka}}u)\|
_{BC_{\wt{\ka}}^s}
\qa \ka\in\gK.
$$
This implies that the norm associated with $(\wt{\rho},\wt{\gK})$ and the
corresponding localization system is stronger than the original one.
Thus the last part of the assertion follows by
interchanging the roles of the singularity data.
\end{proof}
We fix now any one of the equivalent norms for~$BC^{s,\lda}$. Then
\hb{\bigl[\,BC^{s,\lda}(V)\ ;\ s\geq0\,\bigr]} is the weighted
\emph{H\"older scale} of $(\sa,\tau)$-tensor fields on~$M$.
\begin{remark}\label{rem-H.int}
We expect
\beq\label{H.int}
BC^{s,\lda}\doteq(BC^{k,\lda},BC^{{k+1},\lda})_{s-k,\iy}
\qa k<s<k+1
\qb k\in\BN.
\eeq
However, we cannot prove this relation since we do not know whether
$$
\bigl(\ell_\iy(\mf{BC}^k),\ell_\iy(\mf{BC}^{k+1})\bigr)_{s-k,\iy}
\doteq\ell_\iy\bigl((\mf{BC}^k,\mf{BC}^{k+1})_{s-k,\iy}\bigr).
$$
Thus we leave \eqref{H.int} as an open problem.\hfill$\Box$
\end{remark}
We denote by $C_0^{s,\lda}(V)$ the closure of~$\cD\MV$ in~$BC^{s,\lda}(V)$
for
\hb{s\geq0}. Then
\hb{[\,C_0^{s,\lda}(V)\ ;\ s\geq0\,]} is called weighted \emph{small
H\"older scale}. The small H\"older space~$C_0^{s,\lda}$ should not be 
confused with the \emph{little} H\"older space~$bc^{s,\lda}$ which is the 
closure of $BC^{s+1,\lda}$ in~$BC^{s,\lda}$. Of course, 
\hb{bc^{s,\lda}=C_0^{s,\lda}} if $M$~is compact. 
\begin{theorem}\label{thm-H.ret}
Suppose
\hb{s\geq0}. Then $\psi_\iy^\lda$~is a retraction from
$c_0\bigl(\mf{C}_0^s\BXE\bigr)$ onto~$C_0^{s,\lda}(V)$, and 
$\vp_\iy^\lda$~is a coretraction.
\end{theorem}
\begin{proof}
Since $\cD\BXkE$ is dense in 
\hb{C_{0,\ka}^s:=C_0^s\BXkE}, it follows that
$\mf{\cD}\BXE$ is dense in 
$c_0\bigl(\mf{C}_0^s\BXE\bigr)$. By Theorems
\ref{thm-L.ret} and \ref{thm-H.eq} the diagram
 $$
 \begin{picture}(214,62)(-57,-5)        
 \put(-38,45){\makebox(0,0)[c]{\small{$\mf{\cD}\BXE$}}}
 \put(-38,0){\makebox(0,0)[c]{\small{$\cD\MV$}}}
 \put(50,45){\makebox(0,0)[c]{\small{$c_0(\mf{C}_0^s)$}}}
 \put(50,0){\makebox(0,0)[c]{\small{$C_0^{s,\lda}(V)$}}}
 \put(138,45){\makebox(0,0)[c]{\small{$\ell_\iy(\mf{BC}^s)$}}}
 \put(138,0){\makebox(0,0)[c]{\small{$BC^{s,\lda}(V)$}}}
 \put(-43,22.5){\makebox(0,0)[r]{\small{$\psi_\iy^\lda$}}}
 \put(143,22.5){\makebox(0,0)[l]{\small{$\psi_\iy^\lda$}}}
 \put(6,50){\makebox(0,0)[b]{\small{$d$}}}
 \put(6,5){\makebox(0,0)[b]{\small{$d$}}}
 \put(-12,45){\vector(1,0){38}}
 \put(-12,47){\oval(4,4)[l]}
 \put(-12,0){\vector(1,0){38}}
 \put(-12,2){\oval(4,4)[l]}
 \put(72,45){\vector(1,0){38}}
 \put(72,47){\oval(4,4)[l]}
 \put(72,0){\vector(1,0){38}}
 \put(72,2){\oval(4,4)[l]}
 \put(-38,35){\vector(0,-1){25}}
 \put(138,35){\vector(0,-1){25}}
 \end{picture}
 $$
is commuting. From this we read off that we can insert the missing
vertical arrow. This gives the assertion.
\end{proof}
\begin{corollary}\label{cor-H.emb}
Suppose
\hb{0\leq s_0<s_1<s_2<\iy}. Then
$$
C_0^{s_2,\lda}\sdh C_0^{s_1,\lda}\hr BC^{s_1,\lda}\hr BC^{s_0,\lda}.
$$
\end{corollary}
\begin{remarks}\label{rem-H.F}
\hh{(a)}\quad
Let \eqref{L.sd} be chosen. For
\hb{q,r\in[1,\iy]} and
\hb{s\in\BR} denote by~$F_{q,r;\ka}^s$ the \hbox{$E$-valued}
Triebel-Lizorkin spaces on~$\BX_\ka$. Define
\hb{F_{q,r}^{s,\lda}=F_{q,r}^{s,\lda}(V)} by requiring that the diagram
 $$
 \begin{picture}(212,56)(-57,-5)        
 \put(-38,45){\makebox(0,0)[c]{\small{$\mf{\cD}\BXE$}}}
 \put(-38,0){\makebox(0,0)[c]{\small{$\cD\MV$}}}
 \put(50,45){\makebox(0,0)[c]{\small{$\ell_q(\mf{F}_{q,r}^s)$}}}
 \put(50,0){\makebox(0,0)[c]{\small{$F_{q,r}^{s,\lda}(V)$}}}
 \put(138,45){\makebox(0,0)[c]{\small{$\mf{\cD}'\ciBXE$}}}
 \put(138,0){\makebox(0,0)[c]{\small{$\cD'\ciMV$}}}
 \put(-33,22.5){\makebox(0,0)[l]{\small{$\psi_p^\lda$}}}
 \put(55,22.5){\makebox(0,0)[l]{\small{$\psi_p^\lda$}}}
 \put(143,22.5){\makebox(0,0)[l]{\small{$\psi_p^\lda$}}}
 \put(-12,45){\vector(1,0){38}}
 \put(-12,47){\oval(4,4)[l]}
 \put(-12,0){\vector(1,0){38}}
 \put(-12,2){\oval(4,4)[l]}
 \put(72,45){\vector(1,0){38}}
 \put(72,47){\oval(4,4)[l]}
 \put(72,0){\vector(1,0){38}}
 \put(72,2){\oval(4,4)[l]}
 \put(-38,35){\vector(0,-1){25}}
 \put(50,35){\vector(0,-1){25}}
 \put(138,35){\vector(0,-1){25}}
 \end{picture}
 $$
be commuting. Then $F_{q,r}^{s,\lda}(V)$ is a Banach space, a~weighted
Triebel-Lizorkin space of $(\sa,\tau)$-tensor fields on~$M$, and
$$
u\mt\|\vp_q^\lda u\|_{\ell_q(\mf{F}_{q,r}^s)}
$$
is a norm for it. The topology of~$F_{q,r}^{s,\lda}$ is independent of the
particular choice of the singularity datum and the localization system.
If
\hb{M=(\BR^m,g_m)} and
\hb{\gT(M)=[\![\mf{1}]\!]}, then we recover~$F_{q,r}^s\Rm$.
\begin{proofb}
The first part of the assertion follows by obvious modifications of the
proof of Theorem~\ref{thm-H.eq} using the fact that $BC^k\Rm$~is a
point-wise multiplier space for~$F_{q,r}^s\Rm$, provided
\hb{k=k(s,q,r)} is sufficiently large (cf.~Theorem~6.1 in
W.~Yuan, W.~Sickel, and D.~Yang~\cite{YSY10a} or, if
\hb{q<\iy}, Theorem~4.2.2 in~\cite{Tri92a}). The last part is a
consequence of the invariance of~$F_{q,r}^s\Rm$ under diffeomorphisms (see
Theorem~6.7 in~\cite{YSY10a}).
\end{proofb}
\hh{(b)}\quad
It is clear that we can replace in the above construction the
Triebel-Lizorkin spaces  $F_{q,r}^s\Rm$ by any scale of spaces for which a
$BC^k$-point-wise multiplier and
the diffeomorphism theorem are valid. Thus, due to Theorems 6.1 and~6.7
in~\cite{YSY10a}, we can replace $F_{q,r}^s\Rm$ by the scales
$F_{q,r}^{s,\tau}\Rm$ and~$B_{q,r}^{s,\tau}\Rm$ of Triebel-Lizorkin and
Besov type (see~\cite{YSY10a} for precise definitions). However, this has
to be done with care. In fact, we could take, in particular, a~scale
$B_{p,q}^s\Rm$ with
\hb{q\neq p}. But then, due to Remark~\ref{rem-B.Rm}(b), the spaces
$B_{p,q}^{s,\lda}(V)$ constructed this way do not coincide with the Besov
spaces obtained in Remark~\ref{rem-B.Rm}(b) by interpolation.\hfill$\Box$
\end{remarks}
\section{Point-Wise Multipliers}\label{sec-P}
Suppose
\hb{\sa_i,\tau_i\in\BN} for
\hb{i=0,1,2}. Then
\beq\label{P.m}
V_{\tau_1}^{\sa_1}\times V_{\tau_2}^{\sa_2}\ra V_{\tau_0}^{\sa_0}
\qb (v_1,v_2)\mt v_1\bt v_2
\eeq
is called vector bundle \emph{multiplication} if it is (fiber-wise)
bilinear and satisfies
$$
|v_1\bt v_2|_g\leq c\,|v_1|_g\,|v_2|_g
\qa v_i\in V_{\tau_i}^{\sa_i}
\qb i=1,2.
$$
\begin{examples}\label{exa-P.ex}
\hh{(a)}\quad
The duality pairing
\hb{\pw\sco V_\tau^\sa\times V_\sa^\tau\ra V_0^0} is a multiplication.

\par
\hh{(b)}\quad
The map\quad
\hb{V_\tau^\sa\times V_\tau^\sa\ra V_0^0},
\ \hb{(u,v)\mt(u\sn\ol{v})_g}\quad is a multiplication.

\par
\hh{(c)}\quad
The tensor product\quad
\hb{{\otimes}\sco V_{\tau_1}^{\sa_1}\times V_{\tau_2}^{\sa_2}
   \ra V_{\tau_1+\tau_2}^{\sa_1+\sa_2}}\quad is a multiplication.

\par
\hh{(d)}\quad
Assume
\hb{1\leq i\leq\sa} and
\hb{1\leq j\leq\tau}. We denote by
\hb{\sC_j^i\sco V_\tau^\sa\ra V_{\tau-1}^{\sa-1}},
\ \hb{a\mt\sC_j^ia} \ the contraction with respect to positions $i$ and~$j$.
Then
\hb{|\sC_j^ia|_g\leq|a|_g} for
\hb{a\in V_\tau^\sa}.

\par
Suppose
\hb{1\leq i\leq\sa_1+\sa_2} and
\hb{1\leq j\leq\tau_1+\tau_2}. Then
$$
\sC_j^i\sco V_{\tau_1}^{\sa_1}\times V_{\tau_2}^{\sa_2}
\ra V_{\tau_1+\tau_2-1}^{\sa_1+\sa_2-1}
\qb (a,b)\mt\sC_j^i(a\otimes b)
$$
is a multiplication, a~\emph{contraction}.\hfill$\Box$
\end{examples}
In the following, we call the point-wise extension of \eqref{P.m}
\emph{point-wise multiplication induced} by \eqref{P.m} and denote it again
by~%
\hb{{}\bt{}}.
\begin{theorem}\label{thm-P.C}
Let\/ \eqref{P.m} be one of the  multiplications of Examples~\ref{exa-P.ex}.
Suppose
\hb{0\leq s\leq t},
\ \hb{\lda_1,\lda_2\in\BR}, and
\hb{\lda_0=\lda_1+\lda_2}. Then point-wise multiplication induced by
\eqref{P.m} is a continuous bilinear map from
$$
BC^{t,\lda_1}(V_{\tau_1}^{\sa_1})\times H_p^{s,\lda_2}(V_{\tau_2}^{\sa_2})
\text{ into }H_p^{s,\lda_0}(V_{\tau_0}^{\sa_0})
$$
if either
\hb{s=t\in\BN} or
\hb{t>s}, from
$$
BC^{t,\lda_1}(V_{\tau_1}^{\sa_1})\times B_p^{s,\lda_2}(V_{\tau_2}^{\sa_2})
\text{ into }B_p^{s,\lda_0}(V_{\tau_0}^{\sa_0})
$$
if\/
\hb{0<s<t}, and from
$$
BC^{s,\lda_1}(V_{\tau_1}^{\sa_1})\times BC^{s,\lda_2}(V_{\tau_2}^{\sa_2})
\text{ into }BC^{s,\lda_0}(V_{\tau_0}^{\sa_0}).
$$
\end{theorem}
\begin{proof}
Suppose
\hb{s>0} if
\hb{\gF=B}. Let assumption~\eqref{L.sd} be satisfied. Then, given
\hb{u\in BC^{t,\lda_1}(V_{\tau_1}^{\sa_1})} and
\hb{v\in\cD(M,V_{\tau_2}^{\sa_2})}, it follows from
\hb{\sum_{\wt{\ka}}\pi_{\wt{\ka}}^2=\mf{1}} and the definition of~$\gN(\ka)$
that
\beq\label{P.S}
\ka_*\bigl(\pi_\ka(u\bt v)\bigr)
=\sum_{\wt{\ka}\in\gN(\ka)}\ka_*(\pi_\ka u)\bt\ka_*(\pi_{\wt{\ka}}^2v)
\qa \ka\in\gK.
\eeq
Hence the point-wise multiplier properties of the H\"older spaces
\hb{BC_\ka^t=BC^t\BXkE} (see, for example, Theorem~4.7.1 in
Th.~Runst and W.~Sickel~\cite{RuS96a} for the case
\hb{t>s}; the case
\hb{s=t\in\BN} follows easily from Leibniz' rule) imply
\beq\label{P.k}
\big\|\ka_*\big(\pi_\ka(u\bt v)\big)\big\|_{\gF_{p,\ka}^s}
\leq c\,\|\ka_*(\pi_\ka u)\|_{BC_\ka^t}
\sum_{\wt{\ka}\in\gN(\ka)}\|\ka_*(\pi_{\wt{\ka}}^2v)\|_{\gF_{p,\wt{\ka}}^s}
\eeq
for
\hb{\ka\in\gK}. Note that
\beq\label{P.mu}
\card\bigl(\gN(\ka)\bigr)\leq c
\qa \ka\in\gK,
\eeq
by the finite multiplicity of~$\gK$.

\par
It is a consequence of \eqref{S.sd}(ii) and
$$
\ka_*(\pi_{\wt{\ka}}^2v)
=\ka_*\wt{\ka}^*\wt{\ka}_*(\pi_{\wt{\ka}}v)
=\bigl((\wt{\ka}_*\pi_{\wt{\ka}})(\wt{\ka}_*v)\bigr)
\circ(\wt{\ka}\circ\ka^{-1})
$$
that (cf.~\eqref{H.s} and \eqref{H.kk})
\beq\label{P.kk}
\|\ka_*(\pi_{\wt{\ka}}^2v)\|_{\gF_{p,\ka}^s}
\leq c\,\|\wt{\ka}_*(\pi_{\wt{\ka}}v)\|_{\gF_{p,\wt{\ka}}^s}
\qa \wt{\ka}\in\gN(\ka)
\qb \ka\in\gK.
\eeq
Indeed, this follows from Leibniz' rule if
\hb{s\in\BN}, and then, by interpolation if
\hb{s\notin\BN} (also see Theorem~4.3.2 in~\cite{Tri92a}). Thus we obtain
from \hbox{\eqref{P.k}--\eqref{P.kk}} and the density of $\cD\MV$
in~$\gF_p^{s,\lda_2}$
$$
\vd u\bt v\vd_{\gF_p^{s,\lda}}\leq c\,\vd u\vd_{BC^{t,\lda_1}}
\,\vd v\vd_{\gF_p^{s,\lda_2}}
$$
for
\hb{u\in BC^{t,\lda_1}(V_{\tau_1}^{\sa_1})} and
\hb{v\in\gF_p^{s,\lda_2}(V_{\tau_2}^{\sa_2})}. Now the first two assertions
are implied by Theorems \ref{thm-B.loc} and~\ref{thm-H.eq}. The last one
is a consequence of the fact that $BC^s(\BX_\ka)$ is a point-wise
multiplication algebra.
\end{proof}
In applications this theorem is perhaps the most useful multiplier
theorem. The next theorem is an extension of known multiplication
algebra results to the present setting.
\begin{theorem}\label{thm-P.p}
Suppose
\hb{\lda_1,\lda_2\in\BR},
\ \hb{\lda_0=\lda_1+\lda_2}, and
\hb{s>m/p}. Then point-wise multiplication induced by \eqref{P.m} is a
continuous bilinear map from
$$
\gF_p^{s,\lda_1}(V_{\tau_1}^{\sa_1})
\times\gF_p^{s,\lda_2}(V_{\tau_2}^{\sa_2})
\text{ into }\gF^{s,\lda_0+m/p}(V_{\tau_0}^{\sa_0}).
$$
\end{theorem}
\begin{proof}
Theorem~4.6.4 of~\cite{RuS96a} and standard extensions to the half-space case
guarantee that $\gF_{p,\ka}^s$ is a multiplication algebra. Hence we infer
from \eqref{P.S} and\eqref{P.mu}
$$
\big\|\ka_*\big(\pi_\ka(u\bt v)\big)\big\|_{\gF_{p,\ka}^s}^p
\leq c\,\|\ka_*(\pi_\ka u)\|_{\gF_{p,\ka}^s}^p
\sum_{\wt{\ka}\in\gN(\ka)}\|\wt{\ka}(\pi_{\wt{\ka}}^2v)\|
_{\gF_{p,\wt{\ka}}^s}^p
$$
for
\hb{\ka\in\gK}. This implies, due to~\eqref{P.kk},
$$
\vd u\bt v\vd_{\gF_p^{s,\lda_0+m/p}}
\leq c\,\vd u\vd_{\gF_p^{s,\lda_1}}\,\vd v\vd_{\gF_p^{s,\lda_2}},
$$
hence the assertion.
\end{proof}
\section{Traces}\label{sec-T}
Throughout this section
\hb{\pl M\neq\es}. We write~$\thV$ for the restriction~$V_{\pl M}$ of~$V$
to~$\pl M$.

\par
Since $T(\pl M)$~is a subbundle of codimension~$1$ of the vector
bundle~$(TM)_{\pl M}$ over~$\pl M$, there exists a unique vector
field~$\mf{n}$
in~$(TM)_{\pl M}$ of length~$1$, orthogonal to~$T(\pl M)$, and inward
pointing (in any local chart meeting~$\pl M$), the \emph{inward
pointing} unit \emph{normal} vector field on~$\pl M$. In local coordinates,
\hb{\ka=(x^1,\ldots,x^m)},
$$
\mf{n}=\frac1{\sqrt{g_{11}\sn\pl U_{\coU\ka}}}\,\frac\pl{\pl x^1}.
$$
Suppose
\hb{u\in\cD\MV} and
\hb{k\in\BN}. The \emph{trace of order}~$k$ of~$u$ \emph{on}~$\pl M$,
\hb{\ga_ku\in\cD(\pl M,\thV)}, is defined by
$$
\dl\ga_ku,a\dr:=\bigl\dl\na^ku\sn\pl M,a\otimes\mf{n}^{\otimes k}\bigr\dr
\qa a\in\cD(\pl M,V_{\pl M}').
$$
In local coordinates, where
\hb{u=u_{(j)}^{(i)}\frac\pl{\pl x^{(i)}}\otimes dx^{(j)}}, we infer from
\eqref{U.dak}, writing
$$
\ga_ku=(\ga_ku)_{(j)}^{(i)}\frac\pl{\pl x^{(i)}}\otimes dx^{(j)},
$$
that
\beq\label{T.gk}
\Bigl(\sqrt{g_{11}\sn\pl U_{\coU\ka}}\Bigr)^k
(\ga_ku)_{(j)}^{(i)}
=\bigg(\frac{\pl^ku_{(j)}^{(i)}}{(\pl x^1)^k}
+\sum_{\ell=0}^{k-1}b_{(j)(\wt{\imath}),\ell}^{(i)(\wt{\jmath})}
\,\frac{\pl^\ell u_{(\wt{\jmath})}^{(\wt{\imath})}}{(\pl x^1)^\ell}\bigg)
\!\bigg|\,\pl U_{\coU\ka},
\eeq
where
$b_{(j)(\wt{\imath}),\ell}^{(i)(\wt{\jmath})}$ is a polynomial in the
partial derivatives of the Christoffel symbols of order at most
\hb{k-\ell-1}. We write
\hb{\ga=\ga_0} for the \emph{trace operator on}~$\pl M$.

\par
In the next theorem, by~a \emph{universal} coretraction we mean a
continuous linear map which is the unique continuous extension of its
restriction to~$\cD(\pl M,\thV)$. In this sense it is independent of $s$
and~$p$.
\begin{theorem}\label{thm-T.tr}
Suppose
\hb{k\in\BN} and
\hb{s>k+1/p}. Then $\ga_k$~extends to a retraction from~$\gF_p^{s,\lda}(V)$
onto\\
 $B_p^{s-k-1/p,\lda+k+1/p}(\thV)$. It possesses a universal
coretraction~$\ga_k^c$ satisfying
\hb{\ga_i\circ\ga_k^c=0} for
\hb{0\leq i\leq k-1}.
\end{theorem}
\begin{proof}
(1) Let \eqref{L.sd} be chosen. It follows from Lemma~\ref{lem-U.g}(i)
and~(ii) and Lemma~1.4.2 in~\cite{Ama09a} that
$$
\|\rho_\ka^{-1}\sqrt{\ka_*g_{11}}\|_{k,\iy}
+\big\|\rho_\ka\bigl(\sqrt{\ka_*g_{11}}\bigr)^{-1}\big\|_{k,\iy}\leq c 
\qa \ka\in\gK.
$$

\par
(2) For
\hb{t>1/p} we set
$$
\thB_{p,\ka}^{t-1/p}:=
\left\{
\bal
&B_p^{t-1/p}(\BR^{m-1},E),  &&\quad \ka\in\gK_{\pl M},\\
&\{0\},                     &&\quad \ka\in\gK\ssm\gK_{\pl M},
\eal
\right.
$$
with the convention
\hb{B_p^{t-1/p}(\BR^0,E)=E}. We denote by
\hb{\ga_\ka:=\ga_{\pl\BH^m}} the usual trace operator on~$\pl\BH^m$ if
$\ka$~belongs to $\gK_{\pl M}$, and set
\hb{\ga_\ka:=0} if
\hb{\ka\in\gK\ssm\gK_{\pl M}}, where
\hb{\pl\BH^m=\{0\}\times\BR^{m-1}} is identified with~$\BR^{m-1}$. Then we
put
$$
\ga_{k,\ka}:=\rho_\ka^k\bigl(\sqrt{\ga_\ka(\ka_*g_{11})}\bigr)^{-k}
\ga_\ka\circ\pl_1^k
\qa \ka\in\gK.
$$
Note
\hb{\rho_\ka=\thrho_\ka} for
\hb{\ka\in\gK\ssm\gK_{\pl M}}.

\par
Theorems 4.6.2 and~4.6.3 of~\cite{Ama09a} imply that
\hb{\ga_\ka\circ\pl_1^k} is a retraction from~$\gF_{p,\ka}^s$
onto~$\thB_{p,\ka}^{s-k-1/p}$ and that there exists a universal
coretraction~$\wt{\ga}_{k,\ka}^c$ for it satisfying
\beq\label{T.gig}
(\ga_\ka\circ\pl_1^i)\circ\wt{\ga}_{k,\ka}^c=0
\qa 0\leq i\leq k-1,
\eeq
(setting
\hb{\wt{\ga}_{k,\ka}^c:=0} if
\hb{\ka\in\gK\ssm\gK_{\pl M}}). We put
$$
\ga_{k,\ka}^c:=\rho_\ka^{-k}\bigl(\sqrt{\ga_\ka(\ka_*g_{11})}\bigr)^k
\,\wt{\ga}_{k,\ka}^c
\qa \ka\in\gK.
$$
It follows from step~(1) that
\beq\label{T.gc}
\ga_{k,\ka}\in\cL(\gF_{p,\ka}^s,\thB_{p,\ka}^{s-k-1/p})
\qb \ga_{k,\ka}^c\in\cL(\thB_{p,\ka}^{s-k-1/p},\gF_{p,\ka}^s)
\eeq 
and
$$
\|\ga_{k,\ka}\|+\|\ga_{k,\ka}^c\|\leq c
\qa \ka\in\gK.
$$
From \eqref{T.gig} and Leibniz' rule we infer
\beq\label{T.gik}
\ga_{i,\ka}\circ\ga_{k,\ka}^c=\da_{ik}\id
\qa 0\leq i\leq k.
\eeq

\par
(3) We use the notation of Example~\ref{exa-S.ex}(e) and set
\hb{(\thpi_{\ithka},\thchi_{\ithka}):=(\pi_\ka,\chi_\ka)\sn U_{\coU\ithka}}
for
\hb{\thka\in\thgK}. Then it is verified that
\hb{\bigl\{\,(\thpi_{\ithka},\thchi_{\ithka})\ ;\ \thka\in\thgK\,\bigr\}}
is a localization system subordinate to~$\thgK$. We denote by
$$
\thpsi_p^\lda
\sco\ell_p(\thmfB_p^{s-k-1/p})\ra B_p^{s-k-1/p,\lda}(\thV)
$$
the `boundary retraction' defined analogously to~$\psi_p^\lda$.
Correspondingly, $\thvp_p^\lda$~is the `boundary coretraction'.

\par
We put
$$
T_{k,\ka}:=\thrho_\ka^k\thka_*\circ\ga_k\circ\ka^*
\qa \ka\in\gK.
$$
It follows from \eqref{T.gk} that
\beq\label{T.T}
T_{k,\ka}u_\ka
=\ga_{k,\ka}u_\ka+\sum_{\ell=0}^{k-1}b_{\ell,\ka}\ga_{\ell,\ka}u_\ka
\qa u_\ka\in\cD\BXkE,
\eeq
where, due to \eqref{U.Chl} and step~(1),
\hb{\|b_{\ell,\ka}\|_{k-1,\iy}\leq c} for
\hb{0\leq\ell\leq k-1} and
\hb{\ka\in\gK}. Hence, using
\hb{\gF_{p,\ka}^s\hr\gF_{p,\ka}^{s-k+\ell}}, we obtain
\beq\label{T.Tc}
T_{k,\ka}\in\cL(\gF_{p,\ka}^s,\thB_{p,\ka}^{s-k-1/p})
\qb \|T_{k,\ka}\|\leq c
\qa \ka\in\gK.
\eeq 

\par 
(4) For 
\hb{\wt{\ka}\in\gN(\ka)} with 
\hb{\ka,\wt{\ka}\in\gK_{\pl M}} we set 
\hb{\thS_{\ka\wt{\ka}}:=(\thka_*\thwtka^*)(\thchi_\ka\cdot)}. It 
follows from~\eqref{H.kkj} by interpolation that, given 
\hb{t>0}, 
$$ 
\thS_{\ka\wt{\ka}}\in\cL(\thB_{p,\wt{\ka}}^t,\thB_{p,\ka}^t) 
\qb \|\thS_{\ka\wt{\ka}}\|\leq c(t) 
\qa \wt{\ka}\in\gN(\ka) 
\qb \ka,\wt{\ka}\in\gK_{\pl M}. 
$$ 
From this, \eqref{T.Tc}, and 
\hb{\thB_{p,\ka}^{s-i-1/p}\hr\thB_{p,\ka}^{s-k-1/p}} we infer 
\beq\label{T.Ti} 
T_{i,\ka\wt{\ka}}:=\thS_{\ka\wt{\ka}}\circ T_{i,\wt{\ka}} 
 \in\cL(\gF_{p,\wt{\ka}}^s,\thB_{p,\ka}^{s-k-1/p}) 
\qb \|T_{i,\ka\wt{\ka}}\|\leq c 
\qa \wt{\ka}\in\gN(\ka) 
\qb \ka,\wt{\ka}\in\gK_{\pl M},  
\eeq 
for 
\hb{0\leq i\leq k}. 

\par 
The definition of~$\ga_k$ implies 
$$ 
\thpi_{\ithka}\ga_ku 
 =\ga_k(\pi_\ka u)-\sum_{j=0}^{k-1} 
 \bid{k}{j}(\ga_{k-j}\pi_\ka)\ga_j(\chi_\ka u). 
$$ 
Since 
\hb{\chi_\ka u=\sum_{\wt{\ka}\in\gN(\ka)}\pi_{\wt{\ka}}^2u} we thus get 
\beq\label{T.phT} 
\thvp_{p,\ithka}^{\lda+k+1/p}(\ga_ku) 
=T_{k,\ka}(\vp_{p,\ka}^\lda u) 
+\sum_{\wt{\ka}\in\gN(\ka)}R_{k-1,\ka\wt{\ka}}(\vp_p^\lda u), 
\eeq 
where 
$$ 
R_{k-1,\ka\wt{\ka}}\mf{v} 
:=\sum_{i=0}^{k-1}a_{i,\ka\wt{\ka}} 
 T_{i,\ka\wt{\ka}}(\chi_{\wt{\ka}}v_{\wt{\ka}}) 
\qa \mf{v}=(v_\ka), 
$$ 
with 
$$ 
a_{i,\ka\wt{\ka}} 
:=-\sum_{j=1}^{k-1}\bid{k}{j}\bid{j}{i} 
 (\rho_\ka/\rho_{\wt{\ka}})^{\lda+j+m/p} 
 T_{k-j,\ka}(\ka_*\pi_\ka)T_{j-i,\ka\wt{\ka}}(\wt{\ka}_*\pi_{\wt{\ka}})
\postdisplaypenalty=10000
$$ 
for 
\hb{\wt{\ka}\in\gN(\ka)} with 
\hb{\ka,\wt{\ka}\in\gK_{\pl M}}, and 
\hb{a_{i,\ka\wt{\ka}}:=0} otherwise. 

\par 
It follows from \eqref{S.sd}(vi), \eqref{U.LS}, \eqref{T.Tc}, \eqref{T.Ti}, 
and Leibniz' rule that 
$$ 
\|a_{i,\ka\wt{\ka}}\|_{BC^\ell(\pl\BX_\ka)}\leq c(\ell) 
\qa \ka,\wt{\ka}\in\gK 
\qb 0\leq i\leq k 
\qb \ell\in\BN. 
$$ 
Hence, using \eqref{T.Ti} once more, 
\beq\label{T.R} 
R_{k-1,\ka\wt{\ka}}\in\cL(\mf{\gF}_p^s,\thB_{p,\ka}^{s-k-1/p}) 
\qb \|R_{k-1,\ka\wt{\ka}}\|\leq c 
\qa \ka,\wt{\ka}\in\gK. 
\eeq 
Lastly, we set 
\beq\label{T.Tv} 
\mf{T}_{\coT k,\ka}\mf{v} 
:=T_{k,\ka}v_\ka+\sum_{\wt{\ka}\in\gN(\ka)}R_{k-1,\ka\wt{\ka}}(\mf{v}) 
\eeq 
and 
\hb{\mf{T}_{\coT k}\mf{v}:=(\mf{T}_{\coT k,\ka}\mf{v})}. Then we deduce 
from \eqref{T.Tc}, \eqref{T.R}, and the finite multiplicity of~$\gK$ that 
\beq\label{T.FB} 
\mf{T}_{\coT k}\in\cL(\mf{\gF}_p^s,\thmfB_p^{s-k-1/p}).  
\eeq 
Moreover, \eqref{T.phT} implies 
$$ 
\thvp_p^{\lda+k+1/p}\circ\ga_k=\mf{T}_{\coT k}\circ\vp_p^\lda. 
$$ 
Hence it follows from Theorem~\ref{thm-B.ret} and \eqref{T.FB} 
$$ 
\ga_k=\thpsi_k^{\lda+k+1/p}\circ\mf{T}_{\coT k}\circ\vp_p^\lda 
\in\cL\bigl(\gF_p^{s,\lda},\thB_p^{s-k-1/p,\lda+k+1/p}(\thV)\bigr). 
$$ 

\par 
(5) We set 
\hb{\mfga_k^c\mf{w}:=(\ga_{k,\ka}^cw_\ka)}. Then we get from \eqref{T.gc} 
$$ 
\mfga_k^c\in\cL(\thmfB_p^{s-k-1/p},\thmfgF_p^s,). 
$$ 
Note that \eqref{T.gik}, \eqref{T.T}, and \eqref{T.Tv} imply 
\beq\label{T.Tg} 
\mf{T}_{\coT i}\circ\mfga_k^c=\da_{ik}\id 
\qa 0\leq i\leq k. 
\eeq 
Furthermore, given 
\hb{\mf{v}\in\thmfgF_p^s}, 
$$ 
\bal 
\ga_k(\psi_p^\lda\mf{v}) 
&=\sum_\ka\rho_\ka^{-(\lda+m/p)}\ga_k(\pi_\ka\ka^*v_\ka)\\ 
&=\sum_\ka\rho_\ka^{-(\lda+m/p)}\Bigl(\thpi_{\ithka}\ga_\ka(\ka^*v_\ka) 
 +\sum_{j=0}^{k-1}\bid{k}{j}(\ga_{k-j}\pi_\ka)\ga_j(\ka^*v_\ka)\Bigr)\\ 
&=\sum_\ka\rho_\ka^{-(\lda+k+m/p)} 
 \Bigl(\thpi_{\ithka}\thka^*T_{k,\ka}v_\ka 
 +\thka^*\sum_{j=0}^{k-1}\bid{k}{j}T_{k-j,\ka}(\ka_*\pi_\ka) 
 T_{j,\ka}v_\ka\Bigr). 
\eal
$$ 
Thus we infer from \eqref{T.Tg} 
$$ 
\ga_k(\psi_p^\lda\mfga_k^c\mf{w}) 
=\sum_\ka\rho_\ka^{-(\lda+k+m/p)}\thpi_{\ithka}\thka^*w_\ka 
=\thpsi_p^{\lda+k+1/p}\mf{w} 
$$ 
for 
\hb{\mf{w}\in\thmfB_p^{s-k-1/p}}. Hence, by Theorem~\ref{thm-B.ret}, 
$$ 
\ga_k^c:=\psi_p^\lda\circ\mfga_k^c\circ\thvp_p^{\lda+k+1/p} 
\in\cL\bigl(B_p^{s-k-1/p,\lda+k+1/p}(\thV),\gF_p^{s,\lda}\bigr) 
$$ 
and 
\hb{\ga_k\circ\ga_k^c=\id}. This proves the theorem. 
\end{proof}
\begin{corollary}\label{cor-T.tr}
Suppose
\hb{0\leq j_1<\cdots<j_k} and
\hb{s>j_k+1/p}. Then
\beq\label{T.Bsk}
(\ga_{j_0},\ldots,\ga_{j_k})\sco\gF_p^{s,\lda}(V)
\ra\prod_{i=1}^kB_p^{s-j_i-1/p,\lda+j_i+1/p}(\thV)
\eeq
is a retraction possessing a universal coretraction.
\end{corollary}
\begin{proof}
For $(v_1,\ldots,v_k)$ belonging to the product space in \eqref{T.Bsk}
define~$u_i$ for
\hb{1\leq i\leq k} inductively by
\hb{u_1:=\ga_{j_1}^cv_1} and
\hb{u_i:=u_{i-1}+\ga_{j_i}^c(v_i-\ga_{j_i}u_{i-1})} for
\hb{2\leq i\leq k}. Then~$\ga^c$, given by
\hb{\ga^c(v_1,\ldots,v_k):=u_k}, has the claimed properties.
\end{proof}
\section{Spaces with Vanishing Boundary Values}\label{sec-V}
Throughout this section we assume
\hb{\pl M\neq\es}. We denote by
\hb{\ci\gF_p^{s,\lda}=\ci\gF_p^{s,\lda}(V)} the closure of~$\cD\ciMV$
in~$\gF_p^{s,\lda}$.

\par
Let \eqref{L.sd} be chosen. Recalling definitions \eqref{L.ph0k} and
\eqref{L.ps0k} we put
$$
\ci\vp_{p,\ka}^\lda u
:=\rho_\ka^{\lda-m/p'}\ci\vp_\ka u
\qa u\in\cD\MV,
$$
and
$$
\ci\psi_{p,\ka}^\lda v_\ka
:=\rho_\ka^{-\lda+m/p'}\ci\psi_\ka v_\ka
\qa v_\ka\in\cD\BXkE.
$$
Furthermore,
$$
\ci\vp_p^\lda u:=(\ci\vp_{p,\ka}^\lda u)
\qb \ci\psi_p^\lda\mf{v}:={\textstyle\sum_\ka}\ci\psi_{p,\ka}^\lda v_\ka
$$
for
\hb{u\in\cD\MV} and
\hb{\mf{v}\in\mf{\cD}\BXE}.
\begin{theorem}\label{thm-V.ret}
Suppose
\hb{s\in\BR^+\ssm(\BN+1/p)} with
\hb{s>0} if\/
\hb{\gF=B}. Then the diagram 
$$
 \begin{picture}(211,113)(-57,-51)        
 \put(6,50){\makebox(0,0)[b]{\small{$\ci\vp_p^\lda$}}}
 \put(6,5){\makebox(0,0)[b]{\small{$\ci\vp_p^\lda$}}}
 \put(94,50){\makebox(0,0)[b]{\small{$\ci\psi_p^\lda$}}}
 \put(94,5){\makebox(0,0)[b]{\small{$\ci\psi_p^\lda$}}}
 \put(-38,45){\makebox(0,0)[c]{\small{$\cD\ciMV$}}}
 \put(-38,0){\makebox(0,0)[c]{\small{$\ci\gF_p^{s,\lda}(V)$}}}
 \put(50,45){\makebox(0,0)[c]{\small{$\mf{\cD}\ciBXE$}}}
 \put(50,0){\makebox(0,0)[c]{\small{$\ell_p(\cimfgF_p^s)$}}}
 \put(138,45){\makebox(0,0)[c]{\small{$\cD\ciMV$}}}
 \put(138,0){\makebox(0,0)[c]{\small{$\ci\gF_p^{s,\lda}(V)$}}}
 \put(-31,22.5){\makebox(0,0)[l]{\small{$d$}}}
 \put(55,22.5){\makebox(0,0)[l]{\small{$d$}}}
 \put(143,22.5){\makebox(0,0)[l]{\small{$d$}}}
 \put(-14,45){\vector(1,0){40}}
 \put(-14,0){\vector(1,0){40}}
 \put(74,45){\vector(1,0){40}}
 \put(74,0){\vector(1,0){40}}
 \put(-38,33){\vector(0,-1){23}}
 \put(-36,33){\oval(4,4)[t]}
 \put(50,33){\vector(0,-1){23}}
 \put(52,33){\oval(4,4)[t]}
 \put(138,33){\vector(0,-1){23}}
 \put(140,33){\oval(4,4)[t]}
 \put(6,-40){\makebox(0,0)[b]{\small{$\ci\vp_p^\lda$}}}
 \put(94,-40){\makebox(0,0)[b]{\small{$\ci\psi_p^\lda$}}}
 \put(-38,-45){\makebox(0,0)[c]{\small{$\gF_p^{s,\lda}(V)$}}}
 \put(50,-45){\makebox(0,0)[c]{\small{$\ell_p(\mf{\gF}_p^s)$}}}
 \put(138,-45){\makebox(0,0)[c]{\small{$\gF_p^{s,\lda}(V)$}}}
 \put(-14,-45){\vector(1,0){40}}
 \put(74,-45){\vector(1,0){40}}
 \put(-38,-12){\vector(0,-1){23}}
 \put(-36,-12){\oval(4,4)[t]}
 \put(50,-12){\vector(0,-1){23}}
 \put(52,-12){\oval(4,4)[t]}
 \put(138,-12){\vector(0,-1){23}}
 \put(140,-12){\oval(4,4)[t]}
 \end{picture} 
 \postdisplaypenalty=10000
 $$
is commuting and
\hb{\ci\psi_p^\lda\circ\ci\vp_p^\lda=\id}.
\end{theorem}
\begin{proof}
(1) It follows from \eqref{L.ph0D} and \eqref{L.ps0D} that the assertions
concerning the first row of this diagram are valid and
\hb{\ci\psi_p^\lda\circ\ci\vp_p^\lda=\id_{\cD\ciMV}}.

\par
(2) From Lemma~\ref{lem-U.g}(i) and~(ii) and the rules for differentiating
determinants we deduce
$$
\sqrt{\ka_*g}\sim\rho_\ka^m
\qb \|\pa\det(\ka_*g)\|_\iy\leq c(\al)\rho_\ka^{2m}
\qa \al\in\BN^m
\qb \ka\in\gK.
$$

\par
For
\hb{\al,\ba\in\BN^m} with
\hb{\al=\ba+e_i}, where $e_i$~is the \hbox{$i$-th} standard basis vector
of~$\BR^m$, we get
$$
\pa(\sqrt{\ka_*g})=\pl^\ba\Bigl(\frac1{\smash[b]2\sqrt{\ka_*g}}
\,\pl_i\det(\ka_*g)\Bigr).
$$
From this, Leibniz' rule, and Lemma~1.4.2 in~\cite{Ama09a} we infer
$$
\|\sqrt{\ka_*g}\|_{k,\iy}\leq c(k)\rho_\ka^m
\qa \ka\in\gK
\qb k\in\BN.
$$
This implies
$$
\|\ci\vp_\ka u\|_{W_{\coW p,\ka}^k}
\leq c(k)\rho_\ka^m\,\|\ka_*(\chi_\ka u)\|_{W_{\coW p,\ka}^k}
\qa \ka\in\gK
\qb k\in\BN.
$$
Now we obtain
\hb{\ci\vp_p^\lda
   \in\cL\bigl(W_{\coW p}^{k,\lda},\ell_p(\mf{W}_{\coW p}^k)\bigr)} for
\hb{k\in\BN} from \eqref{W.Wkp} and the arguments leading from there to
\eqref{W.phi}. Analogously we find
\hb{\ci\psi_p^\lda
   \in\cL\bigl(\ell_p(\mf{W}_{\coW p}^k),W_{\coW p}^{k,\lda}\bigr)} for
\hb{k\in\BN}  by the arguments of step~(3) of the proof of
Theorem~\ref{thm-W.ret}, as well as
\hb{\ci\psi_p^\lda\circ\ci\vp_p^\lda=\id}.

\par
(3) Since
\hb{\cD\ciBXkE\sdh\ci W_{\coW p,\ka}^k} implies
\hb{\mf{\cD}\ciBXE\sdh c_c(\cimfW_{\coW p}^k)}, we deduce from
\eqref{W.cl} that $\mf{\cD}\ciBXE$ is dense
in~$\ell_p(\cimfW_{\coW p}^k)$. Clearly,
\hb{\ci\psi_p^\lda\bigl(\mf{\cD}\ciBXE\bigr)\is\cD\ciMV}. Thus we infer
\hb{\ci\psi_p^\lda
   \in\cL\bigl(\ell_p(\cimfW_{\coW p}^k),\ci W_{\coW p}^{k,\lda}\bigr)} for
\hb{k\in\BN} from steps (1) and~(2). Similarly, we find
\hb{\ci\vp_p^\lda\bigl(\cD\ciMV\bigr)\is\ell_p(\cimfW_{\coW p}^k)}, and
thus
\hb{\ci\vp_p^\lda
   \in\cL\bigl(\ci W_{\coW p}^{k,\lda},\ell_p(\cimfW_{\coW p}^k)\bigr)} for
\hb{k\in\BN}. This proves the theorem if
\hb{s\in\BN}.

\par
(4) Suppose
\hb{s\in\BR^+\ssm(\BN+1/p)}. For
\hb{0<\ta<1} set
\hb{\pr_\ta:=\pe_\ta} if
\hb{\gF=H}, and
\hb{\pr_\ta:=\pr_{\ta,p}} otherwise. Assume
\hb{0<s<k} with
\hb{k\in\BN}. Then
\hb{s\notin\BN+1/p} implies
\hb{\ci\gF_{p,\ka}^s\doteq(L_{p,\ka},\ci\gF_{p,\ka}^p)_{s/k}}. Thus,
cf.~the proof of Theorem~\ref{thm-B.ret},
$$
\ell_p(\cimfgF_p^s)
=\bigl(\ell_p(\mf{L}_p),\ell_p(\cimfgF_p^k)\bigr)_{s/k}.
$$
Now we infer from step~(3) that $r$~is a retraction
from~$\ell_p(\cimfgF_p^s)$  onto
\hb{(L_p^\lda,\ci W_{\coW p}^{k,\lda})_{s/k}\doteq\ci\gF_p^{s,\lda}},
since the latter interpolation space is the closure of~$\cD\ciME$ in
\hb{(L_p^\lda,W_{\coW p}^{k,\lda})_{s/k}\doteq\gF_p^{s,\lda}} by the
density properties of~$\pr_\ta$.
\end{proof}
\begin{corollary}\label{cor-V.ret}
Suppose
\hb{0\leq s_0<s_1<\iy} and
\hb{\ta\in(0,1)}. If
\hb{s_0,s_1,s_\ta\notin\BN+1/p}, then
$$
[\ci H_p^{s_0,\lda},\ci H_p^{s_1,\lda}]_\ta\doteq\ci H_p^{s_\ta,\lda}
\qb (\ci B_p^{s_0,\lda},\ci B_p^{s_1,\lda})_{\ta,p}
\doteq\ci B_p^{s_\ta,\lda},
$$
provided
\hb{s_0>0} in the latter case.
\end{corollary}
The next theorem characterizes the spaces~$\ci\gF_p^{s,\lda}$ by means of
trace operators.
\begin{theorem}\label{thm-V.tr}
\begin{itemize}
\item[(i)]
Suppose
\hb{0\leq s<1/p} with
\hb{s>0} if
\hb{\gF=B}. Then
\hb{\ci\gF_p^{s,\lda}=\gF_p^{s,\lda}}.
\item[(ii)]
Assume
\hb{k\in\BN} and
\hb{k+1/p<s<k+1+1/p}. Set
\hb{\vec\ga_k:=(\ga_0,\ldots\ga_k)}. Then
$$
\ci\gF_p^{s,\lda}=\{\,u\in\gF_p^{s,\lda}\ ;\ \vec\ga_ku=0\,\}.
$$
\end{itemize}
\end{theorem}
\begin{proof}
(i)~follows from Theorem~\ref{thm-V.ret} and the corresponding properties
of these spaces on~$\BX_\ka$.

\par
(ii) Let the assumptions of~(ii) be satisfied. If
\hb{u\in\ci\gF_p^{s,\lda}}, then it is obvious by Corollary~\ref{cor-T.tr}
that
\hb{\vec\ga_ku=0}.

\par
Conversely, suppose
\hb{u\in\gF_p^{s,\lda}} and
\hb{\vec\ga_ku=0}. Then we infer from \eqref{T.gk} that
\hb{(\ga_\ka\circ\pl_1^i)\ka_*(\pi_\ka u)=0} for
\hb{\ka\in\gK} and
\hb{0\leq i\leq k}. Hence
\hb{\ka_*(\pi_\ka u)\in\ci\gF_{p,\ka}^s} for
\hb{\ka\in\gK_{\pl M}} (cf.~Theorem~2.9.4 in~\cite{Tri78a}). Consequently,
\hb{\ci\vp_p^\lda u\in\ell_p(\cimfgF_p^s)} and, by Theorem~\ref{thm-V.ret},
\ \hb{u=\ci\psi_p^\lda(\ci\vp_p^\lda u)\in\ci\gF_p^{s,\lda}}. This proves
assertion~(ii).
\end{proof}
\begin{theorem}\label{thm-V.S}
Suppose
\hb{k\in\BN} and
\hb{k+1/p<s<k+1+1/p}. Put
$$
\pl\gF_p^{s,\lda}(\thV)
:=\prod_{i=0}^kB_p^{s-i-1/p,\lda+i+1/p}(\thV).
$$
Let $\vec\ga_k^c$ be a coretraction for~$\vec\ga_k$. Then\/
\hb{\gF_p^{s,\lda}(V)
   =\ci\gF_p^{s,\lda}(V)\oplus\vec\ga_k^c\,\pl\gF_p^{s,\lda}(\thV)}.
\end{theorem}
\begin{proof}
Let $X$ and~$Y$ be Banach spaces,
\hb{r\in\cL\XY} and
\hb{r^c\in\cL\YX} with
\hb{r\circ r^c=\id}. Then
\hb{r^c\circ r} is a projection in~$\cL(X)$ and
$$
X=\ker(r^c\circ r)\oplus\im(r^c\circ r)=\ker(r)\oplus r^cY,
$$
where $r^cY$~is the image space of~$Y$ in~$X$, so that
\hb{r^c\sco Y\ra r^cY} is an isometric isomorphism (cf.~Lemma~4.1.5
in~\cite{Ama09a} or Lemma~2.3.1 in~\cite{Ama95a}). Hence the assertion
follows from Corollary~\ref{cor-T.tr} and Theorem~\ref{thm-V.tr}.
\end{proof}
\section{Spaces of Negative Order}\label{sec-N}
For
\hb{u\in\cD(M,V')} and
\hb{v\in\cD\MV} we put
$$
\dl u,v\dr_M:=\int_M\dl u,v\dr\,dV_{\coV g}.
$$
This bilinear form extends uniquely to a separating continuous bilinear
form
$$
\pw_M\sco L_{p'}^{-\lda}(V')\times L_p^\lda(V)\ra\BK
$$
by which we identify the dual Banach space
of~$L_p^\lda(V)$ with~$L_{p'}^{-\lda}(V')$, that is,
\beq\label{N.LL}
L_p^\lda(V)'=L_{p'}^{-\lda}(V')
\text{ by means of the duality pairing  }\pw_M.
\eeq

\par
It follows from Theorem~\ref{thm-V.tr}(i) that
\beq\label{N.DFL}
\cD\ciMV\sdh\ci\gF_p^{s,\lda}(V)\sdh L_p^\lda(V)
\eeq
for
\hb{s\geq0}, with
\hb{s>0} if
\hb{\gF=B}. Theorem~\ref{thm-B.ref} implies that $\ci\gF_p^{s,\lda}(V)$ is
reflexive, being a closed linear subspace of a reflexive space. Thus we
put, in accordance with \eqref{N.LL},
\beq\label{N.HBD}
\gF_p^{-s,\lda}(V):=\bigl(\ci\gF_{p'}^{s,-\lda}(V')\bigr)'
\qa s>0.
\eeq
It is a consequence of \eqref{N.LL}, \eqref{N.DFL}, and
Theorem~\ref{thm-B.ret} that
\beq\label{N.FHF}
\gF_p^{s,\lda}(V)\sdh L_p^\lda(V)\sdh\gF_p^{-s,\lda}(V)\sdh\cD\ciMV
\qa s>0,
\eeq
with respect to the duality pairing~%
\hb{\pw_M}, that is,
$$
\dl u,v\dr_{\gF_p^{-s,\lda}(V)}=\dl u,v\dr_M
\qa s>0
\qb u\in\ci\gF_{p'}^{s,-\lda}(V')
\qb v\in L_p(V).
$$
Finally, we define
\beq\label{N.B0}
B_p^{0,\lda}(V):=\bigl(B_p^{-1,\lda}(V),B_p^{1,\lda}(V)\bigr)_{1/2,p}.
\eeq
\begin{theorem}\label{thm-N.ret}
Suppose
\hb{s\in\BR} with
\hb{s\notin-\BN^\times+1/p} if\/
\hb{\pl M\neq\es}. Then $\psi_p^\lda$~is a retraction from
$\ell_p(\mf{\gF}_p^s)$ onto~$\gF_p^{s,\lda}(V)$, and $\vp_p^\lda$~is a
coretraction.
\end{theorem}
\begin{proof}
(1) If
\hb{s\geq0} with
\hb{s>0} if
\hb{\gF=B}, then this is a restatement of Theorem~\ref{thm-B.ret}.

\par
(2) Suppose
\hb{s<0}, with
\hb{s\notin-\BN+1/p} if
\hb{\pl M\neq\es}. Then Theorem~\ref{thm-V.ret} guarantees that
$\ci\psi_{p'}^{-\lda}$~is a retraction from
$\ell_{p'}(\cimfgF_{p'}^{-s})$ onto~$\gF_{p'}^{-s,-\lda}(V')$ and
$\ci\vp_{p'}^{-\lda}$~is a coretraction. Since
\hb{(\ci\gF_{p',\ka}^{-s})'=\gF_{p,\ka}^s} with respect to the duality
pairing
\hb{\pw_\ka:=\pw_{\BX_\ka}}, it follows
$$
\bigl(\ell_{p'}(\cimfgF_{p'}^{-s})\bigr)'=\ell_p(\mf{\gF}_p^s)
$$
with respect to~%
\hb{\mfpw}. Using
$$
\ci\vp_{p',\ka}^{-\lda}=\rho_\ka^{-\lda-m/p}\ci\vp_\ka,
$$
the proof of Theorem~\ref{thm-L.ret}, and Theorem~\ref{thm-B.ret} we thus
obtain
$$
\psi_p^\lda=(\ci\vp_{p'}^{-\lda})'
\in\cL\bigl(\ell_p(\mf{\gF}_p^s),\gF_p^{s,\lda}(V)\bigr)
$$
and
$$
\vp_p^\lda=(\ci\psi_{p'}^{-\lda})'
\in\cL\bigl(\gF_p^{s,\lda}(V),\ell_p(\mf{\gF}_p^s)\bigr)
$$
with
\hb{\psi_p^\lda\circ\vp_p^\lda=\id}. This proves the assertion if
\hb{s<0}.

\par
(3) If
\hb{s=0}, then the claim for~$B_p^{0,\lda}(V)$ follows by interpolation
from \eqref{N.B0} and steps (1) and~(2).
\end{proof}
\begin{corollary}\label{cor-N.ret}
Suppose
\hb{s\in\BR} and
\hb{s\notin-\BN^\times+1/p}\, if\/
\hb{\pl M\neq\es}. Then
\hb{H_2^{s,\lda}(V)\doteq B_2^{s,\lda}(V)}.
\end{corollary}
It is convenient to denote by~$\ci\gF_p^{s,\lda}(V)$ \emph{for each}
\hb{s\in\BR} the closure of~$\cD\ciMV$ in~$\gF_p^{s,\lda}(V)$. Then
$$
\ci\gF_p^{s,\lda}(V)=\gF_p^{s,\lda}(V)
\qa s<1/p.
$$
In fact, this follows from Theorem~\ref{thm-V.tr}(i) and \eqref{N.FHF}.
\begin{theorem}\label{thm-N.ref}
The Banach spaces $\ci\gF_p^{s,\lda}(V)$ and $\gF_p^{s,\lda}(V)$ are
reflexive for
\hb{s\in\BR}. Moreover,
$$
\bigl(\ci\gF_p^{s,\lda}(V)\bigr)'\doteq\ci\gF_{p'}^{-s,-\lda}(V')
\qa s\in\BR.
$$
\end{theorem}
\begin{proof}
This follows from Theorem~\ref{thm-B.ref}, the fact that closed linear
subspaces and reflexive Banach spaces are reflexive, and the duality
properties of the real interpolation functor~%
\hb{\pr_{1/2,p}} (see~\eqref{N.B0}).
\end{proof}
Suppose
\hb{\pl M\neq\es}. Since $\gF_p^{s,\lda}(V)$ is reflexive and densely
embedded in~$L_p(V)$ for
\hb{s>0}, we can define for
\hb{s>0}
$$
\check\gF_p^{-s,\lda}(V):=\bigl(\gF_{p'}^{s,-\lda}(V')\bigr)'
$$
with respect to the duality pairing~%
\hb{\pw_M}. By Theorem~\ref{thm-V.tr}(i)
$$
\check\gF_p^{s,\lda}(V):=\gF_p^{s,\lda}(V)
\qa -1+1/p<s<0.
$$
However, if
\hb{s<-1+1/p}, then $\check\gF_p^{s,\lda}(V)$ is no longer a space of
distribution sections on~$\ci M$, but contains distribution sections
supported on~$\pl M$.
This is made precise by the next theorem in which we use the notations of
Theorem~\ref{thm-V.S}.
\begin{theorem}\label{thm-N.dM}
Suppose
\hb{\pl M\neq\es} and
\hb{-k-2+1/p<s<-k-1+1/p} with
\hb{k\in\BN}. Put
$$
\pl\gF_p^{s,\lda}(\thV)
:=\prod_{i=0}^kB_p^{s+i+1-1/p,\lda-i-1+1/p}(\thV).
$$
Then
$$
\check\gF_p^{s,\lda}(V)
=\gF_p^{s,\lda}(V)\oplus(\vec\ga_k)'\,\pl\gF_p^{s,\lda}(\thV),
$$
where $\vec\ga_k$~maps $\gF_{p'}^{-s,-\lda}(V)$ onto
$\prod_{i=0}^kB_{p'}^{-s-i-1/p',-\lda+i+1/p'}(\thV)$.
\end{theorem}
\begin{proof}
Since
\hb{\pl(\pl M)=\es} the statement follows from \eqref{N.HBD} and
Theorem~\ref{thm-V.S} by duality (cf.~Section~2 of~\cite{Ama02a}).
\end{proof}
\section{Interpolation}\label{sec-I}
Now we can improve on the interpolation results already noted in
Corollaries \ref{cor-B.ret} and~\ref{cor-V.ret}.
\begin{theorem}\label{thm-I.I}
Suppose
\hb{-\iy<s_0<s_1<\iy},
\ \hb{0<\ta<1}, and
\hb{\lda_0,\lda_1\in\BR}.
\begin{itemize}
\item[(i)]
The following interpolation relations,
$$
\bigl[H_p^{s_0,\lda_0}(V),H_p^{s_1,\lda_1}(V)\bigr]_\ta
\doteq H_p^{s_\ta,\lda_\ta}(V)
\qb
\bigl(B_p^{s_0,\lda_0}(V),B_p^{s_1,\lda_1}(V)\bigr)_{\ta ,p}
\doteq B_p^{s_\ta,\lda_\ta}(V),
$$
are valid, provided
\hb{s_0,s_1,s_\ta\notin-\BN^\times+1/p}\, if\/
\hb{\pl M\neq\es}.
\item[(ii)]
Suppose
\hb{\pl M\neq\es} and
\hb{s_0,s_1,s_\ta\in\BR^+\ssm(N+1/p)}. Then
$$
\bigl[\ci H_p^{s_0,\lda_0}(V),\ci H_p^{s_1,\lda_1}(V)\bigr]_\ta
\doteq\ci H_p^{s_\ta,\lda_\ta}(V)
\qb
\bigl(\ci B_p^{s_0,\lda_0}(V),\ci B_p^{s_1,\lda_1}(V)\bigr)_{\ta ,p}
\doteq\ci B_p^{s_\ta,\lda_\ta}(V).
$$
\item[(iii)]
If either\/
\hb{\pl M=\es} or
\hb{s_0,s_1,s_\ta\notin-\BN^\times+1/p}, then
\hb{\bigl(H_p^{s_0,\lda_0}(V),H_p^{s_1,\lda_1}(V)\bigr)_{\ta,p}
   \doteq B_p^{s_\ta,\lda_\ta}(V)}.
\item[(iv)]
Suppose\/
\hb{\pl M\neq\es} and
\hb{s_0,s_1,s_\ta\in\BR^+\ssm(\BN+1/p)}. Then
\hb{\bigl(\ci H_p^{s_0,\lda_0}(V),\ci H_p^{s_1,\lda_1}(V)\bigr)_{\ta,p}
   \doteq\ci B_p^{s_\ta,\lda_\ta}(V)}.
\end{itemize}
\end{theorem}
\begin{proof}
Fix \eqref{L.sd}

\par
(1) Set
\hb{\mu:=\lda_1-\lda_0}. Denote by
\hb{\rho_\ka^{-\mu}H_{p,\ka}^{s_1}} the image space of the self-map
\hb{u\mt\rho_\ka^{-\mu}u} of~$H_{p,\ka}^{s_1}$ so that this map is an
isometric isomorphism from~$H_{p,\ka}^{s_1}$
onto~$\rho_\ka^{-\mu}H_{p,\ka}^{s_1}$. Then Theorem~\ref{thm-N.ret}
implies that the diagram
\beq\label{I.rH}
\bal
 \begin{picture}(135,58)(-59,-5)        
\put(55,30){\vector(-2,-1){40}}
\put(-45,30){\vector(2,-1){40}}
\put(-30,40){\vector(1,0){60}}
\put(5,0){\makebox(0,0){\small$H_p^{s_1,\lda_1}$}}
\put(-50,40){\makebox(0,0){\small$H_{p,\ka}^{s_1}$}}
\put(60,40){\makebox(0,0){\small$\rho_\ka^{-\mu}H_{p,\ka}^{s_1}$}}
\put(0,45){\makebox(0,0)[b]{\small$u\mt\rho_\ka^{-\mu}u$}}
\put(0,35){\makebox(0,0)[t]{\small$\cong$}}
\put(-30,15){\makebox(0,0)[r]{\small$\psi_{p,\ka}^{\lda_1}$}}
\put(37.5,15){\makebox(0,0)[l]{\small$\psi_{p,\ka}^{\lda_0}$}}
\end{picture}
\eal
\eeq
is commuting. Interpolation theory guarantees (cf.~formula~(7) in
Section~3.4.1 of~\cite{Tri78a})
\beq\label{I.mu}
[H_{p,\ka}^{s_0},\rho_\ka^{-\mu}H_{p,\ka}^{s_1}]_\ta
\doteq\rho_\ka^{-\ta\mu}H_{p,\ka}^{s_\ta},
\eeq
uniformly with respect to
\hb{\ka\in\gK}. From Theorem~\ref{thm-N.ret} we infer that
$\psi_p^{\lda_0}$~is a retraction from~$\ell_p(\mf{H}_p^{s_0})$
onto~$H_p^{s_0,\lda_0}$ and, due to~\eqref{I.rH},
from~$\ell_p(\mf{\rho}^{-\mu}\mf{H}_p^{s_1})$
onto~$H_p^{s_1,\lda_1}$, where
\hb{\mf{\rho}^{-\mu}\mf{H}_p^s:=\prod_\ka\rho_\ka^\mu H_{p,\ka}^s}.
Thus, by \eqref{I.mu} and interpolation, $\psi_p^{\lda_0}$~is a retraction
from~$\ell_p(\mf{\rho}^{-\ta\mu}\mf{H}_p^{s_\ta})$
onto
\hb{[H_p^{s_0,\lda_0},H_p^{s_1,\lda_1}]_\ta}. By
replacing~$\mu$ in \eqref{I.rH} by~$\ta\mu$ we see that
$\psi_p^{\lda_0}$~is a retraction
from~$\ell_p(\mf{\rho}^{-\ta\mu}\mf{H}_p^{s_\ta})$
onto~$H_p^{s_\ta,\lda_\ta}$. This implies the claim for
\hb{\gF=H}. The proof for
\hb{\gF=B} is analogous.

\par
(2) The assertions of~(ii) follow by invoking in step~(1)
Theorem~\ref{thm-V.ret} instead of Theorem~\ref{thm-N.ret}. The remaining
statements are obtained by similar arguments from the corresponding results
on~$\BX_\ka$.
\end{proof}
\section{Embedding Theorems}\label{sec-E}
Weighted Bessel potential and Besov spaces on singular manifolds enjoy
embedding properties similar to the ones known for the standard non-weighted
spaces on~$\BR^m$.
\begin{theorem}\label{thm-E.p}
Suppose
\hb{s_0<s<s_1} and
\hb{\mu<\lda}.
\begin{itemize}
\item[(i)]
If\/
\hb{\pl M\neq\es} and
\hb{s_0,s,s_1\in\BR^+\ssm(\BN+1/p)}, then
\beq\label{E.HB0} 
\ci H_p^{s_1,\lda}(V)\sdh\ci B_p^{s,\lda}(V)\sdh\ci H_p^{s_0,\lda}(V).
\eeq 
If, moreover, 
\hb{\rho\leq1}, then 
\hb{\ci\gF_p^{s,\mu}(V)\sdh\ci\gF_p^{s,\lda}(V)}, whereas 
\hb{\ci\gF_p^{s,\lda}(V)\sdh\ci\gF_p^{s,\mu}(V)} if 
\hb{\rho\geq1}. 
\item[(ii)]
If either\/
\hb{\pl M=\es} or
\hb{s_0,s,s_1\notin-\BN^\times+1/p}, then
\beq\label{E.HB} 
H_p^{s_1,\lda}(V)\sdh B_p^{s,\lda}(V)\sdh H_p^{s_0,\lda}(V).
\eeq 
Furthermore, 
\hb{\gF_p^{s,\mu}(V)\sdh\gF_p^{s,\lda}(V)} if 
\hb{\rho\leq1}, whereas 
\hb{\rho\geq1} implies 
\hb{\gF_p^{s,\lda}(V)\sdh\gF_p^{s,\mu}(V)}. 
\end{itemize}
\end{theorem}
\begin{proof} 
Assertions \eqref{E.HB0} and \eqref{E.HB} follow from 
Theorem~\ref{thm-I.I}(ii) and~(i), respectively, and the general 
interrelations of the real and complex interpolation functors. 

\par 
If 
\hb{\rho\leq1}, then it is obvious that
\beq\label{E.W}
W_{\coW p}^{k,\mu}(V)\sdh W_{\coW p}^{k,\lda}(V)
\qb \ci W_{\coW p}^{k,\mu}(V)\sdh\ci W_{\coW p}^{k,\lda}(V)
\qa k\in\BN.
\eeq
Thus, by duality,
\beq\label{E.W1}
H_p^{k,\mu}(V)\sdh H_p^{k,\lda}(V)
\qa k\in-\BN^\times.
\eeq
From these embeddings we obtain, once more by interpolation, the second 
part of assertion~(i) and assertion~(ii), respectively, provided 
\hb{\rho\leq1}. If 
\hb{\rho\geq1}, then the embeddings in \eqref{E.W} and \eqref{E.W1} are 
reversed. Thus the remaining statements are also clear. 
\end{proof}
The next theorem concerns embedding theorems of Sobolev type.
\begin{theorem}\label{thm-E.Sob}
\begin{itemize}
\item[(i)]
Suppose
\hb{s_0<s_1} and
\hb{p_0,p_1\in(1,\iy)} satisfy
\hb{s_1-m/p_1=s_0-m/p_0}. Then
$$
\gF_{p_1}^{s_1,\lda}(V)\sdh\gF_{p_0}^{s_0,\lda+s_1-s_0}(V).
$$
\item[(ii)]
Assume
\hb{s\geq t+m/p} with
\hb{t\geq0} and
\hb{s>t+m/p} if\/
\hb{t\in\BN}. Then
$$
\gF_p^{s,\lda}(V)\sdh C_0^{t,\lda+m/p}(V).
$$
\end{itemize}
\end{theorem}
\begin{proof}
(1) Let the assumptions of~(i) be satisfied. Since
\hb{s_1>s_0} implies
\hb{p_1<p_0}, it follows from the known embeddings
\hb{\gF_{p_1,\ka}^{s_1}\hr\gF_{p_0,\ka}^{s_0}} and from \eqref{W.ll}   that
\hb{\ell_{p_1}(\mf{\gF}_{p_1}^{s_1})\hr\ell_{p_0}(\mf{\gF}_{p_0}^{s_0})}.
Moreover,
\hb{m/p_1=m/p_0+s_1-s_0} implies
\hb{\psi_{p_1}^\lda=\psi_{p_0}^{\lda+s_1-s_0}}. From this and
Theorem~\ref{thm-N.ret} we infer that the diagram
$$
 \begin{picture}(154,66)(-78,-5)        
 \put(6,50){\makebox(0,0)[b]{\small{$\vp_{p_1}^\lda$}}}
 \put(6,5){\makebox(0,0)[b]{\small{$\psi_{p_0}^{\lda+s_1-s_0}$}}}
 \put(-54,45){\makebox(0,0)[c]{\small{$\gF_{p_1}^{s_1,\lda}$}}}
 \put(-54,0){\makebox(0,0)[c]{\small{$\gF_p^{s_0,\lda+s_1-s_0}$}}}
 \put(60,45){\makebox(0,0)[c]{\small{$\ell_{p_1}(\mf{\gF}_{p_1}^{s_1})$}}}
 \put(60,0){\makebox(0,0)[c]{\small{$\ell_{p_0}(\mf{\gF}_{p_0}^{s_0})$}}}
 \put(-39,45){\vector(1,0){75}}
 \put(36,0){\vector(-1,0){60}}
 \put(-54,33){\vector(0,-1){23}}
 \put(-52,33){\oval(4,4)[t]}
 \put(60,33){\vector(0,-1){23}}
 \put(62,33){\oval(4,4)[t]}
 \end{picture}
 $$
is commuting. Thus the assertions of~(i) follow.

\par
(2) Let $s$ and~$t$ satisfy the hypotheses of~(ii). Then the known
embeddings
\hb{\gF_{p,\ka}^s\hr C_{0,\ka}^t} guarantee
$$
c_c(\mf{\gF}_p^s)\is c_c(\mf{C}_0^t)\hr c_0(\mf{C}_0^t).
$$
Thus
\hb{\ell_p(\mf{\gF}_p^s)\hr c_0(\mf{C}_0^t)} since $c_0(\mf{C}_0^t)$~is
closed in~$\ell_\iy(\mf{C}_0^t)$. Hence, using
\hb{\psi_p^\lda=\psi_\iy^{\lda+m/p}}, it follows from
Theorem~\ref{thm-N.ret} that the diagram
$$
 \begin{picture}(144,66)(-71,-5)        
 \put(6,50){\makebox(0,0)[b]{\small{$\vp_p^\lda$}}}
 \put(6,5){\makebox(0,0)[b]{\small{$\psi_\iy^{\lda+m/p}$}}}
 \put(-51,45){\makebox(0,0)[c]{\small{$\gF_p^{s,\lda}$}}}
 \put(-51,0){\makebox(0,0)[c]{\small{$C_0^{t,\lda+m/p}$}}}
 \put(60,45){\makebox(0,0)[c]{\small{$\ell_p(\mf{\gF}_p^s)$}}}
 \put(60,0){\makebox(0,0)[c]{\small{$c_0(\mf{C}_0^t)$}}}
 \put(-38,45){\vector(1,0){75}}
 \put(36,0){\vector(-1,0){60}}
 \put(-51,33){\vector(0,-1){23}}
 \put(-49,33){\oval(4,4)[t]}
 \put(60,33){\vector(0,-1){23}}
 \put(62,33){\oval(4,4)[t]}
 \end{picture}
 $$
is commuting. Thus claim~(ii) is implied by the density of $\cD\MV$ in each 
of the spaces.
\end{proof}
\section{Differential Forms and Exterior Derivatives}\label{sec-F}
Throughout this section
$$
\bt\quad
M\text{ is oriented}.
$$
For
\hb{0\leq k\leq m} we consider the vector subbundle
$$
\bwt^k:=\bigl(\bwt^kT^*M,\prsn_{g^*}\bigr)
$$
of
\hb{V_k^0=T_k^0M}, the \hbox{$k$-fold} exterior product of
\hb{V_1^0=T^*M}, where
\hb{\bwt^0=T_0^0M=M\times\BK}. The sections of~$\bwt^k$ are the
\hbox{$k$-forms} on~$M$, that is, the differential forms of order~$k$. We
write~$\Om^k(M)$ for the $C^\iy(M)$-module of smooth \hbox{$k$-forms}, and
we set
\hb{\Om^k(M):=\{0\}} for
\hb{k\notin\{0,1,\ldots,m\}}.

\par
We also consider the subbundle
$$
\bwt^{'k}:=\bigl(\bwt^kTM,\prsn_g\bigr)
$$
of
\hb{V_0^k=T_0^kM}. Then
\hb{\bwt^{'k}=(\bwt^k)'} with respect to the duality pairing~%
\hb{\pw} obtained by restriction from the \hbox{$V_k^0$-pairing}. It
follows from \eqref{U.G} and the (vector bundle) conjugate linearity
of~$g^\sh$ that
$$
G^k\sco\bwt^k\ra\bwt^{'k}
\qb \al\mt G_0^k\,\ol{\al}
$$
is a vector bundle isomorphism whose inverse is
$$
G_k\sco\bwt^{'k}\ra\bwt^k
\qb v\mt G_k^0\,\ol{v}.
$$
Let $\om$ be the Riemannian volume form of~$M$. The definition of the Hodge
adjoint
\hb{\Ho\ba\in\Om^{m-k}(M)} implies
\beq\label{F.abw}
(\al\sn\ba)_{g^*}\om=\al\wedge\Ho\ol{\ba}
\qa \al,\ba\in\Om^k(M),
\eeq
(cf.~Section~XX.8 of~\cite{Die69b} or Section~XI.2 in~\cite{AmE08a}). By
\eqref{U.gG}
$$
\dl v,\al\dr=\dl\al,v\dr=(\al\sn G_kv)_{g^*}
\qa \al\in\bwt^k
\qb v\in\bwt^{'k}.
$$
Consequently,
\beq\label{F.dua}
\dl\al,v\dr=\int_M\dl\al,v\dr\,dV_{\coV g}
=\int_M(\al\sn G_kv)_{g^*}\,\om
=\int_M\al\wedge\Ho G_kv
\eeq
for
\hb{\al\in\Om^k(M)} and \hb{v\in\cD(M,\bwt^{'k})}.
\begin{theorem}\label{thm-F.S}
All results obtained in the preceding sections for Bessel potential and
Besov spaces of $(\sa,\tau)$-tensor fields remain valid for the
corresponding spaces of \hbox{$k$-forms}, if\/ $(V_\tau^\sa,V_\sa^\tau)$ is
replaced by~$(\bwt^k,\bwt^{'k})$.
\end{theorem}
\begin{proof}
Obvious.
\end{proof}
Justified by this we refer in the following simply to the theorems and
formulas of the preceding sections and it is understood that we mean the
corresponding results for the spaces of differential forms.

\par
The exterior derivative
\hb{d\sco\Om^k(M)\ra\Om^{k+1}(M)} is characterized by
$$
\bal
d\al(X_0,X_1,\ldots,X_k)
&=\sum_{0\leq i\leq k}(-1)^i
 \,\na_{\cona X_i}\bigl(\al(X_0,\ldots,\wh{X}_i,\ldots,X_k)\bigr)\\
&\qquad
 {}+\sum_{0\leq i<j\leq k}(-1)^{i+j}
 \,\al\bigl([X_i,X_j],X_0,
 \ldots,\wh{X}_i,\ldots,\wh{X}_j,\ldots,X_k)\bigr)
\eal
$$
for
\hb{\al\in\Om^k(M)} and
\hb{X_0,\ldots,X_k\in\cT_0^1M}, where $[X_i,X_j]$ is the Lie bracket and
\hb{\wh{\ph{\al}}} the usual omission symbol. Since $\na$~is torsion free,
that is,
\hb{\na_{\cona X}Y-\na_YX=[X,Y]}, it follows
\beq\label{F.d}
d\al(X_0,\ldots,X_k)
=\sum_{0\leq i\leq k}
(-1)^i(\na_{\cona X_i}\al)(X_0,\ldots,\wh{X}_i,\ldots,X_k).
\eeq
The coderivative
\hb{\da\sco\Om^k(M)\ra\Om^{k-1}(M)} is defined by
\beq\label{F.cod}
\da\al:=(-1)^{m(k+1)+1}\,\Ho d\,\Ho\al
\qa \al\in\Om^k(M).
\eeq
Recall
\beq\label{F.ss}
\Ho\Ho\al=(-1)^{k(m-k)}\,\al
\qa \al\in\Om^k(M).
\eeq
Suppose
\hb{\al\in\Om^{k-1}(M)} and
\hb{\ba\in\Om^k(M)}. Then
\hb{\al\wedge\Ho\ba\in\Om^{m-1}(M)} and
\hb{d\,\Ho\ba\in\Om^{m-k+1}(M)}. Note that \eqref{F.cod} and
\eqref{F.ss} imply
\hb{\Ho\da\ba=(-1)^k\,d\,\Ho\ba\in\Om^k(M)}. From this we obtain
$$
d(\al\wedge\Ho\ba)=d\al\wedge\Ho\ba+(-1)^{k-1}\,\al\wedge d\,\Ho\ba
=d\al\wedge\Ho\ba-\al\wedge\Ho\da\ba.
$$
Hence, setting
\hb{\ba=G_kv} with
\hb{v\in\cD(M,\bwt^{'k})},
$$
d(\al\wedge\Ho G_kv)
=d\al\wedge\Ho G_kv-\al\wedge\Ho G_{k-1}G^{k-1}\da G_kv.
$$
Thus Stoke's theorem implies, as is well-known, Green's formula
which, due to \eqref{F.dua}, takes the form
$$
\dl d\al,v\dr_M-\dl\al,G^{k-1}\da G_kv\dr_M
=\int_{\pl M}\thia\,^*(\al\wedge\Ho G_kv)
$$
for
\hb{\al\in\Om^{k-1}(M)} and
\hb{v\in\cD(M,\bwt^{'k})}. In particular,
\beq\label{F.duad}
\dl d\al,v\dr_M=\dl\al,G^{k-1}\da G_kv\dr_M
\eeq
if either $\al$ or~$v$ is compactly supported in~$\ci M$; thus, in
particular, if
\hb{\pl M=\es}. Similarly, using
\hb{\al\wedge\Ho\ba=\ba\wedge\Ho\al} and setting
\hb{\al=G_{k-1}w}, we find
\beq\label{F.duac}
\dl\da\ba,w\dr_M
=\dl\ba,G^kdG_{k-1}w\dr_M
\qa \ba\in\Om^k(M)
\qb w\in\cD(M,\bwt^{'k-1}),
\eeq
if either $\ba$ or~$w$ has compact support in~$\ci M$.

\par
Now we can establish the fundamental mapping properties of the exterior
differential and codifferential operators.
\begin{theorem}\label{thm-F.d}
Suppose
\hb{s\in\BR}.
\begin{itemize}
\item[(i)]
Assume either
\hb{\pl M=\es} or
\hb{s\geq0} with
\hb{s>0} if\/
\hb{\gF=B}. Then
\beq\label{F.ds}
d\in\cL\bigl(\gF_p^{s+1,\lda}(\bwt^k),\gF_p^{s,\lda}(\bwt^{k+1})\bigr)
\eeq
and
\beq\label{F.cds}
\da\in\cL\bigl(\gF_p^{s+1,\lda}(\bwt^k),\gF_p^{s,\lda+2}(\bwt^{k-1})\bigr).
\eeq
\item[(ii)]
Assume
\hb{\pl M\neq\es} and
\hb{s>-2+1/p} with
\hb{s\neq-1+1/p}. Then
\beq\label{F.d0s}
d\in\cL\bigl(\ci\gF_p^{s+1,\lda}(\bwt^k),\gF_p^{s,\lda}(\bwt^{k+1})\bigr)
\eeq
and
\beq\label{F.cd0s}
\da\in\cL\bigl(\ci\gF_p^{s+1,\lda}(\bwt^k),
\gF_p^{s,\lda+2}(\bwt^{k-1})\bigr).
\eeq
\end{itemize}
\end{theorem}
\begin{proof}
(1) Suppose
\hb{s\geq0} with
\hb{s>0} if
\hb{\gF=B}. Then \eqref{F.ds} is a consequence of \eqref{F.d} and
Theorem~\ref{thm-B.Na}.

\par
(2) For
\hb{\al\in\Om^k(M)} it follows from \eqref{F.abw} and \eqref{F.ss} that
$$
|\Ho\al|_{g^*}^2\om
=\Ho\al\wedge\Ho\Ho\ol{\al}=(-1)^{k(m-k)}\,\Ho\al\wedge\ol{\al}
=\ol{\al}\wedge\Ho\al=|\al|_{g^*}^2\om.
$$
Hence
\hb{\rho^{\lda+2k-m+m-k}\,|\Ho\al|_{g^*}=\rho^{\lda+k}\,|\al|_{g^*}}.
This implies
\beq\label{F.Lp}
\Ho\in\Lis\bigl(L_p^\lda(\bwt^k),L_p^{\lda+2k-m}(\bwt^{m-k})\bigr).
\eeq
From \eqref{U.LT}(ii) we infer for
\hb{X\in\cT M}
\beq\label{F.Nab}
\na_{\cona X}(\al\wedge\Ho\ba)
=\na_{\cona X}\al\wedge\Ho\ba+\al\wedge\na_{\cona X}(\Ho\ba)
\qa \al,\ba\in\Om^k(M).
\eeq
Since
\hb{\na_{\cona X}\om=0} we obtain from \eqref{U.ip}
$$
\na_{\cona X}\bigl((\al\sn\ol{\ba})_{g^*}\om\bigr)
=(\na_{\cona X}\al\sn\ol{\ba})_{g^*}\om
+(\al\sn\na_{\cona X}\ol{\ba})_{g^*}\om.
$$
Using this, \eqref{F.Nab}, and \eqref{F.abw} we deduce
\hb{\al\wedge\na_{\cona X}(\Ho\ba)=\al\wedge\Ho\na_{\cona X}\ba} for
\hb{\al\in\Om^k(M)}. Consequently,
$$
\na_{\cona X}(\Ho\ba)=\Ho(\na_{\cona X}\ba)
\qa \ba\in\Om^k(M)
\qb X\in\cT M.
$$
By this and \eqref{F.Lp} we get
$$
\Ho\in\Lis\bigl(W_{\coW p}^{j,\lda}(\bwt^k),
W_{\coW p}^{j,\lda+2k-m}(\bwt^{m-k})\bigr)
\qa j\in\BN.
$$
Hence, by interpolation,
$$
\Ho\in\Lis\bigl(\gF_p^{s,\lda}(\bwt^k),
\gF_p^{s,\lda+2k-m}(\bwt^{m-k})\bigr)
\qa s\in\BR^+,
$$
provided
\hb{s>0} if
\hb{\gF=B}. Now \eqref{F.cds} follows from \eqref{F.cod} and step~(1),
provided
\hb{s\geq0} with
\hb{s>0} if
\hb{\gF=B}.

\par
(3) Definition~\eqref{U.G} implies
$$
|G^k\al|_g^2=\dl G^k\al,G_kG^k\al\dr=\dl\al,G^k\al\dr=|\al|_{g^*}^2
\qa \al\in\Om^k(M).
$$
Thus, since $\na$~commutes with~$g^\sh$, hence with~$G^k$,
$$
\rho^{\lda+2k+i-k}\,|\na^iG^k\al|_g
=\rho^{\lda+i+k}\,|G^k\na^i\al|_g
=\rho^{\lda+i+k}\,|\na^i\al|_{g^*}
$$
for
\hb{i\in\BN}. From this we deduce
$$
G^k\in\Lis\bigl(W_{\coW p}^{j,\lda}(\bwt^k),
W_{\coW p}^{j,\lda+2k}(\bwt^{'k})\bigr)
\qa (G^k)^{-1}=G_k,
$$
for
\hb{j\in\BN}. Thus, by interpolation,
\beq\label{F.Gk}
G^k\in\Lis\bigl(\gF_p^{s,\lda}(\bwt^k),
\gF_p^{s,\lda+2k}(\bwt^{'k})\bigr)
\qb (G^k)^{-1}=G_k,
\eeq
for
\hb{s\geq0} with
\hb{s>0} if
\hb{\gF=B}.

\par
The part of \eqref{F.cds} which has already been shown and \eqref{F.Gk}
imply
\beq\label{F.A}
A:=G^{k-1}\da G_k
\in\cL\bigl(\gF_{p'}^{s+1,-\lda}(\bwt^{'k}),
\gF_{p'}^{s,-\lda}(\bwt^{'k-1})\bigr).
\eeq

\par
(4) Suppose
\hb{\pl M=\es}. Then \eqref{F.A} and Theorem~\ref{thm-N.ref} imply
$$
A'\in\cL\bigl(\gF_p^{-s,\lda}(\bwt^{k-1}),
\gF_p^{-s-1,\lda}(\bwt^k)\bigr)
$$
for
\hb{s\in\BR^+} with
\hb{s>0} if
\hb{\gF=B}. From this and \eqref{F.duad} we infer, by density, that $A'$~is
the unique continuous extension of~$d$. This proves \eqref{F.ds} for all
\hb{s\in\BR} with the exception
\hb{s=0} if
\hb{\gF=B}. But now we close this gap by interpolation.

\par
(5) Suppose
\hb{\pl M=\es} and
\hb{s>0}. Then \eqref{F.ds} and \eqref{F.Gk} imply
$$
C:=G^kdG_{k-1}
\in\cL\bigl(\gF_{p'}^{s+1,-\lda-2}(\bwt^{'k-1}),
\gF_{p'}^{s,-\lda}(\bwt^{'k})\bigr).
$$
Hence
$$
C'\in\cL\bigl(\gF_p^{-s,\lda}(\bwt^k),
\gF_p^{-s-1,\lda+2}(\bwt^{k-1})\bigr).
$$
Since \eqref{F.duac} shows that $C'$~is the unique continuous extension
of~$\da$ over~$\gF_p^{-s,\lda}(\bwt^k)$ we get assertion~\eqref{F.cds} for
\hb{s<0}. The case
\hb{\gF=B} and
\hb{s=0} is once more covered by interpolation. Assertion~(i) is thus
proved.

\par
(6) Suppose
\hb{\pl M\neq\es}. If
\hb{s\geq0}, then \eqref{F.d0s} and \eqref{F.cd0s} are obvious by~(i).
Clearly, $G^k$~maps $\cD(\ci M,\bwt^k)$ into $\cD(\ci M,\bwt^{'k})$. Hence
\eqref{F.Gk} implies
\beq\label{F.Gk0}
G^k\in\Lis\bigl(\ci\gF_p^{s,\lda}(\bwt^k),
\ci\gF_p^{s,\lda+2k}(\bwt^{'k})\bigr)
\qb (G^k)^{-1}=G_k,
\eeq
for
\hb{s\geq0} with
\hb{s>0} if
\hb{\gF=B}.

\par
Suppose
\hb{-1+1/p<s<0}, that is,
\hb{0<-s<1-1/p=1/p'}. Then, by Theorem~\ref{thm-V.tr}(i),
$$
\gF_{p'}^{-s,-\lda}(\bwt^{'k-1})=\ci\gF_{p'}^{-s,-\lda}(\bwt^{'k-1}).
$$
From this, \eqref{F.Gk0}, and the observation of the beginning of this step
we infer
$$
A\in\cL\bigl(\ci\gF_{p'}^{-s+1,-\lda}(\bwt^{'k}),
\ci\gF_{p'}^{-s,-\lda}(\bwt^{'k-1})\bigr).
$$
Hence, by \eqref{N.HBD},
$$
A'\in\cL\bigl(\gF_p^{s,\lda}(\bwt^{k-1}),
\gF_p^{s-1,\lda}(\bwt^k)\bigr).
$$
Thus \eqref{F.duad} implies
$$
d\in\cL\bigl(\gF_p^{s,\lda}(\bwt^{k-1}),
\gF_p^{s-1,\lda}(\bwt^k)\bigr).
$$
This proves claim~\eqref{F.d0s} if
\hb{-2+1/p<s<-1+1/p}. Now we obtain assertion~\eqref{F.d0s} for
\hb{-1+1/p<s<0} by interpolation, thanks to Theorem~\ref{thm-I.I}. The
proof of statement~\eqref{F.cd0s} is similar.
\end{proof}
As an immediate consequence of this theorem we see that the Hodge Laplacian
\hb{\Da_{\Hodge}:=d\da+\da d} satisfies
$$
\Da_{\Hodge}
\in\cL\bigl(\gF_p^{s+2,\lda}(\bwt^k),
\gF_p^{s,\lda+2}(\bwt^k)\bigr)
$$
if either
\hb{s\in\BR} and
\hb{\pl M=\es}, or
\hb{s\geq0} with
\hb{s>0} if
\hb{\gF=B}. If
\hb{\pl M\neq\es}, then
$$
\Da_{\Hodge}
\in\cL\bigl(\ci\gF_p^{s+2,\lda}(\bwt^k),
\gF_p^{s,\lda+2}(\bwt^k)\bigr),
$$
provided
\hb{s>-2+1/p} with
\hb{s\neq-1+1/p}. Note that
\hb{\Da_{\Hodge}=-\Da_M} if
\hb{k=0}, where
\hb{\Da_M=\divgrad} is the Laplace-Beltrami operator of~$M$.

\par
Finally, we apply these results to derive the mapping properties of the
basic differential operators of vector analysis. For this we recall that
the gradient and the divergence operator can be represented (taking the
complex case into account) by
\beq\label{F.gr}
\grad=G^1\circ d\sco\cD(M)\ra\cD(M,T_0^1M)
\eeq
and
\beq\label{F.di}
\tdiv=-\da\circ G_1\sco\cD(M,T_0^1M)\ra\cD(M),
\eeq
respectively.
\begin{theorem}\label{thm-F.V}
Suppose
\hb{s\in\BR}.
\begin{itemize}
\item[(i)]
Assume either 
\hb{\pl M=\es} or
\hb{s\geq0} with
\hb{s>0} if\/
\hb{\gF=B}. Then
$$
\grad
\in\cL\bigl(\gF_p^{s+1,\lda}(M),
\gF_p^{s,\lda+2}(T_0^1M)\bigr)
\qb \tdiv
\in\cL\bigl(\gF_p^{s+1,\lda}(T_0^1M),
\gF_p^{s,\lda}(M)\bigr).
$$
\item[(ii)]
If\/
\hb{\pl M\neq\es} and
\hb{s>-2+1/p} with
\hb{s\neq-1+1/p}, then
$$
\grad
\in\cL\bigl(\ci\gF_p^{s+1,\lda}(M),
\gF_p^{s,\lda+2}(T_0^1M)\bigr)
\qb \tdiv
\in\cL\bigl(\ci\gF_p^{s+1,\lda}(T_0^1M),
\gF_p^{s,\lda}(M)\bigr).
$$
\end{itemize}
\end{theorem}
\begin{proof}
It follows from \eqref{U.ginv} that
$$
\dl\al,G^1\ba\dr=\dl G^1\al,\ba\dr
\qa \al,\ba\in\cD(M,T_1^0M).
$$
From this and \eqref{F.Gk} we obtain by duality arguments similar to the
ones used in the preceding proof that
$$
G^1
\in\Lis\bigl(\gF_p^{s,\lda}(T_1^0M),
\gF_p^{s,\lda+2}(T_0^1M)\bigr)
\qb (G^1)^{-1}=G_1,
$$
for all
\hb{s\in\BR} if
\hb{\pl M=\es}. Similarly, \eqref{F.Gk0} implies
$$
G^1
\in\Lis\bigl(\ci\gF_p^{s,\lda}(T_1^0M),
\ci\gF_p^{s,\lda+2}(T_0^1M)\bigr)
\qb s\in\BR.
$$
Now the assertion follows from \eqref{F.gr}, \eqref{F.di}, and
Theorem~\ref{thm-F.d}.
\end{proof}

\begin{acknowledgement}
The author is grateful to a referee for calling his attention to some 
early Russian references and to N.~Nistor for pointing out 
papers \hbox{\cite{AIN06a}--\cite{AN07a}}.   
\end{acknowledgement}

\def\cprime{$'$} \def\polhk#1{\setbox0=\hbox{#1}{\ooalign{\hidewidth
  \lower1.5ex\hbox{`}\hidewidth\crcr\unhbox0}}}

\end{document}